\documentclass[fontsize=11pt,english,reqno]{amsart} 
\usepackage{ucs}
\usepackage{ae}
\usepackage{microtype}
\usepackage[utf8]{inputenc}
\usepackage[T1]{fontenc}
\usepackage[english]{babel}

\title[Locally analytic representations in mixed characteristic]{Locally analytic representations in mixed characteristic via stacks}
\author{Maximilian Hauck}
\date{}

\usepackage{amsmath}
\usepackage{amssymb}
\usepackage{amstext}
\usepackage{amsthm}
\usepackage{stmaryrd}

\usepackage{enumitem}

\usepackage{scrlayer-scrpage}
\usepackage{geometry}
\usepackage{xcolor}

\usepackage{mathtools}
\usepackage{slashed}

\usepackage{mathscinet}
\usepackage[numbers]{natbib}

\usepackage{graphicx}
\usepackage{tikz-cd}
\usepackage{caption}

\usepackage{imakeidx}
\makeindex[intoc]

\usepackage{tabularx}

\usepackage[colorlinks=true,hyperindex, linkcolor=magenta, urlcolor=black, pagebackref=false, citecolor=cyan,pdfpagelabels]{hyperref}
\usepackage[capitalize]{cleveref}

\newif\ifcomment
\newcommand{\comment}[1]{\ifcomment#1\fi}

\newlength{\currentparindent}
\makeatletter
\newcommand{\@minipagerestore}{\setlength{\parindent}{\currentparindent}}
\makeatother

\newcommand{\nospacepunct}[1]{\makebox[0pt][l]{#1}}

\usepackage{relsize}
\usepackage[bbgreekl]{mathbbol}
\usepackage{amsfonts}

\usepackage{accents}

\newcommand{\overcirc}[1]{\accentset{\circ}{#1}}

\DeclareSymbolFontAlphabet{\mathbb}{AMSb}
\DeclareSymbolFontAlphabet{\mathbbl}{bbold}

\newcommand{\solid}{\mathsmaller{\square}}

\DeclareMathOperator{\Ker}{Ker}

\let\Im\im

\DeclareMathOperator{\id}{id}

\DeclareMathOperator{\Hom}{Hom}
\DeclareMathOperator{\Ext}{Ext}

\DeclareMathOperator{\Spec}{Spec}

\DeclareMathOperator{\Mod}{Mod}
\DeclareMathOperator{\Perf}{Perf}
\DeclareMathOperator{\gr}{gr}
\DeclareMathOperator{\Fil}{Fil}
\DeclareMathOperator{\Rep}{Rep}

\DeclareMathOperator{\fib}{fib}

\DeclareMathOperator{\RHom}{RHom}

\DeclareMathOperator{\CoMod}{CoMod}

\DeclareMathOperator{\AnSpec}{AnSpec}

\DeclareMathOperator{\Cont}{Cont}

\newcommand{\colim}{\operatornamewithlimits{colim}}
\newcommand{\Lim}{\operatornamewithlimits{lim}}
\let\lim\Lim

\newcommand{\sHom}{\underline{\mathrm{Hom}}}

\newcommand{\Der}{\mathcal{D}\kern -.5pt er}

\newcommand\A{\mathbb{A}}
\newcommand\F{\mathbb{F}}
\newcommand\G{\mathbb{G}}
\newcommand\Q{\mathbb{Q}}

\newcommand\Z{\mathbb{Z}}

\newcommand\T{\mathbb{T}}

\newcommand\DD{\mathbb{D}}

\let\P\PP
\newcommand\R{\mathbb{R}}

\newcommand\D{\mathcal{D}}

\let\O\cO

\newcommand\an{\mathrm{an}}

\newcommand\la{\mathrm{la}}
\newcommand\sm{\mathrm{sm}}

\newcommand\Betti{\mathrm{Betti}}

\newcommand\alg{\mathrm{alg}}

\let\epsilon\varepsilon
\let\phi\varphi
\let\ol\overline
\let\ul\underline
\let\tensor\otimes

\let\cal\mathcal

\newtheorem{introthm}{Theorem}
\newtheorem{thm}{Theorem}[section]
\newtheorem{prop}[thm]{Proposition}
\newtheorem{lem}[thm]{Lemma}
\newtheorem{cor}[thm]{Corollary}

\theoremstyle{definition}
\newtheorem{defi}[thm]{Definition}

\newtheorem{remm}[thm]{Remark}

\newtheorem{exx}[thm]{Example}
\newenvironment{ex}
  {\pushQED{\qed}\exx}
  {\popQED\endexx}
\newenvironment{rem}
  {\pushQED{\qed}\remm}
  {\popQED\endremm}
  
\AddToHook{env/prop/begin}{\crefalias{thm}{prop}}
\AddToHook{env/lem/begin}{\crefalias{thm}{lem}}
\AddToHook{env/cor/begin}{\crefalias{thm}{cor}}
\AddToHook{env/defi/begin}{\crefalias{thm}{defi}}
\AddToHook{env/thmdefi/begin}{\crefalias{thm}{thmdefi}}
\AddToHook{env/lemdefi/begin}{\crefalias{thm}{lemdefi}}
\AddToHook{env/conv/begin}{\crefalias{thm}{conv}}
\AddToHook{env/rem/begin}{\crefalias{thm}{remm}}
\AddToHook{env/warn/begin}{\crefalias{thm}{warn}}
\AddToHook{env/ex/begin}{\crefalias{thm}{exx}}

\numberwithin{equation}{section}

\begin{document}

\begin{abstract}
For any $p$-adic Lie group $G$, we construct a group stack $G^\la$ over the $p$-adic branch of Clausen--Scholze's stack of norms such that quasicoherent sheaves on its classifying stack $*/G^\la$ recover locally analytic representations of $G$ after base change to any suitable Tate Huber pair, even in positive or mixed characteristic. We also show that smooth $\F_p$-representations of $G$ embed fully faithfully into sheaves on the mod $p$ fibre of $*/G^\la$ and establish Poincaré duality for cohomology of locally analytic representations in mixed characteristic.
\end{abstract}

\maketitle

\tableofcontents

\section{Introduction}

Let $p$ be a prime number.

\subsection{Motivation and background}

The theory of $p$-adic representations of $p$-adic Lie groups has played a central role in number theory over the last two decades, most visibly through its place in the $p$-adic local Langlands programme. Since the foundational works \cite{SchneiderTeitelbaum1}, \cite{SchneiderTeitelbaum2} of Schneider--Teitelbaum and Emerton's treatment of locally analytic vectors in \cite{EmertonLocAn}, it has become clear that a very convenient class of representations of $p$-adic Lie groups is given by locally analytic representations. They have found applications in the study of the cohomology of Shimura varieties as pioneered by Pan in \cite{Pan1}, \cite{Pan2} and have proved to be a crucial tool in understanding decompletion in $p$-adic Hodge theory, a point of view first brought forward by Berger--Colmez in \cite{BergerColmezSenTheory}. 

More recently, the theory of locally analytic representations has been given new theoretical foundations using Clausen--Scholze's condensed mathematics by Rodrigues Jacinto and Rodr\'iguez Camargo in \cite{SolidLocAnReps1}, \cite{SolidLocAnReps}. Namely, for any $p$-adic Lie group $G$, they define a functor of (derived) locally analytic vectors for solid $G$-representations over a $p$-adic field and establish that the resulting category of locally analytic representations enjoys a number of convenient properties. Most notably, they also establish a new geometric approach to locally analytic representations by showing that they may be interpreted as quasicoherent sheaves on the classifying stack of a certain ``locally analytic incarnation'' $G^\la$ of $G$ as a group stack in the world of Clausen--Scholze's analytic stacks. 

The upshot of this is that various functors and cohomological comparison theorems from the classical theory of $p$-adic representations of $p$-adic Lie groups can then easily be obtained from the six-functor formalism for quasicoherent sheaves on analytic stacks. More seriously, this geometric approach towards locally analytic representations is indispensable in the recent efforts of Anschütz--Le Bras--Rodríguez Camargo--Scholze to formulate a geometric $p$-adic local Langlands conjecture: the key player in their story on the automorphic side is a stack $\mathrm{Bun}_G^\la$, which is stratified by classifying stacks of group stacks indexed by the Kottwitz set $B(G)$ such that, for $b\in B(G)$ basic, the corresponding stratum is given by $*/G_b^\la$, where $G_b$ is the pure inner form of $G$ associated to $b$.

All of this, however, takes place in the world of representations with characteristic zero coefficients. While first steps towards setting up the condensed approach to locally analytic representations in mixed characteristic have been taken by Porat in \cite{Porat1} and \cite{Porat2}, the aim of this paper is to further clarify the theory in mixed characteristic and, most importantly, to establish a geometric approach to locally analytic representations in mixed characteristic via analytic stacks. In particular, we view this as the very first step towards a geometrisation of the mod $p$ local Langlands correspondence as well as a pathway towards new techniques in the study of mod $p$ cohomology of Shimura varieties.

\subsection{Main results}

Let us now describe our results in more detail. We write $\cal{N}$ for Clausen--Scholze's stack of norms over $\AnSpec\Z_\solid$ from \cite[Lecture 20]{AnStacks} and let $\cal{N}_{|p|<1}$ be its open substack defined by the condition that $p$ have norm less than $1$. Our first step is to construct a functor from paracompact $p$-adic manifolds to analytic stacks over $\cal{N}_{|p|<1}/\R_{>0}$, where the Betti stack $\R_{>0}$ acts by rescaling the norm. 

For this, we use the algebra
\begin{equation*}
C^\la(\Z_p, \Z(\!(\pi)\!))\coloneqq \colim_{\epsilon>0} \Z(\!(\pi)\!)\left\langle\pi^{\epsilon n} \binom{x}{n}: n\geq 0\right\rangle
\end{equation*}
of locally analytic $\Z(\!(\pi)\!)$-valued functions on $\Z_p$, where each term in the colimit is obtained by $\pi$-completing the module 
\begin{equation*}
\Z[\![\pi]\!]\left[\pi^{\epsilon n}\binom{x}{n}: n\geq 0\right]
\end{equation*}
and subsequently inverting $\pi$. This definition is inspired by Amice's theorem \cite[§10, Cor.\ 2]{Amice} and variants have appeared previously in \cite{BergerRozensztajn}, \cite{Porat1} and \cite{Porat2}. Namely, we show that the analytic stack $\AnSpec C^\la(\Z_p, \Z(\!(\pi)\!))$ is equipped with a canonical descent datum to $\cal{N}/\R_{>0}$ and we define $\Z_p^\la$ as the base change of this descent to $\cal{N}_{|p|<1}/\R_{>0}$. Then the central result is the following:

\begin{introthm}[\cref{thm:man-lafunctor}, \cref{prop:man-transverse}, \cref{prop:man-transmutation}]
\label{thm:intro-A}
The association $\Z_p^d\mapsto (\Z_p^\la)^d$ extends to a functor
\begin{equation*}
\begin{split}
\{\text{paracompact $p$-adic manifolds}\}&\rightarrow\{\text{analytic stacks over $\cal{N}_{|p|<1}/\R_{>0}$}\} \\
M&\mapsto M^\la
\end{split}
\end{equation*}
with the following properties:
\begin{enumerate}[label=(\arabic*)]
\item If $X$ is a smooth $\Z_p$-scheme, then $X(\Z_p)^\la$ is obtained by transmuting $X$ with respect to the ring stack $\Z_p^\la$, i.e.\ by sheafifying the presheaf over $\cal{N}_{|p|<1}/\R_{>0}$ given by 
\begin{equation*}
R\mapsto X(\Z_p^\la(R))\;.
\end{equation*}

\item If $M_1\rightarrow M\leftarrow M_2$ are transverse maps of paracompact $p$-adic manifolds, then $M_1\times_M M_2$ is a paracompact $p$-adic manifold and 
\begin{equation*}
M_1^\la\times_{M^\la} M_2^\la\cong (M_1\times_M M_2)^\la\;.
\end{equation*}
\end{enumerate}
\end{introthm}

The key part in proving \cref{thm:intro-A} lies in establishing that any locally analytic endomorphism of $\Z_p$ gives rise to an endomorphism of the stack $\Z_p^\la$ over $\cal{N}_{|p|<1}/\R_{>0}$, which requires us to do some combinatorics involving binomial coefficients.

Next, we show that quasicoherent sheaves on the classifying stack $*/G^\la$ indeed recover the notion of locally analytic representation of $G$ from \cite[§5.3]{Porat2} in the situations where the latter is defined. More specifically, let $(B, B^+)$ be a complete Tate Huber pair equipped with a pseudouniformiser satisfying the assumptions of loc.\ cit., i.e.\ $(B, B^+)$ is of slope $\geq 1$ and residually of finite type in the sense of \cite[Def.\ 3.4]{Porat2} and \cite[Def.\ 3.6]{Porat2}, respectively. Then there is a canonical map $\AnSpec\hspace{1pt}(B, B^+)_\solid\rightarrow \cal{N}_{|p|<1}/\R_{>0}$ and we prove:

\begin{introthm}[\cref{thm:locan-compareporat}]
\label{thm:intro-B}
For any $p$-adic Lie group $G$, pullback along the map 
\begin{equation*}
\AnSpec\hspace{1pt}(B, B^+)_\solid\rightarrow \AnSpec\hspace{1pt}(B, B^+)_\solid/G^\la
\end{equation*}
induces an equivalence
\begin{equation*}
\D(\AnSpec\hspace{1pt}(B, B^+)_\solid/G^\la)\cong \D(\Rep^\la_{(B, B^+)_\solid} G)\;,
\end{equation*}
where the right-hand side is the derived category of locally analytic $B$-linear representations of $G$ as defined in \cite[Def.\ 5.9]{Porat2}.
\end{introthm}

Besides the conceptual advantages of a geometric approach to locally analytic representations, the above shows that interpreting a locally analytic representation as a quasicoherent sheaf on the classifying stack of $G^\la$ also has technical advantages: As we point out in \cref{rem:norms-maptatehuber}, any complete Tate Huber pair $(B, B^+)$ in which $p$ is topologically nilpotent admits a canonical map $\AnSpec\hspace{1pt}(B, B^+)_\solid\rightarrow\cal{N}_{|p|<1}/\R_{>0}$ and hence we can make sense of the classifying stack $\AnSpec\hspace{1pt}(B, B^+)_\solid/G^\la$. This allows us to get rid of almost all the choices and assumptions needed in \cite{Porat2} when working with locally analytic representations in mixed characteristic; i.e.\ we can omit the choice of pseudouniformiser and dispense with the slope and finiteness assumptions.

We also introduce a ``finite radius`` variant of $G^\la$ denoted $G^{h\dagger\text{-}\an}$ for any $h\in\R$ and a corresponding notion of $h\dagger$-analytic representation in mixed characteristic, the analogue of which over $\Q_p$ has recently appeared in \cite[§2]{Mikami}. Broadly speaking, being $h\dagger$-analytic instead of merely locally analytic requires that the orbit map of a representation overconverge on disks of radius $|p|^h$. While this notion is not often relevant in practice, it is of utmost technical importance for some six-functor arguments concerning quotient stacks by $G^\la$, which is why we include it here.

Given that the term ``locally analytic representation'', even in mixed characteristic, only makes sense over some Tate Huber pair $(B, B^+)$ in which $p$ is topologically nilpotent or, at best, over the somewhat mysterious base $\cal{N}_{|p|<1}/\R_{>0}$, it is natural to wonder how (smooth) $\F_p$-representations fit into the picture, which are what the mod $p$ local Langlands programme is usually concerned with. In this direction, we establish that smooth $\F_p$-representations of $G$ may actually be found inside our framework:

\begin{introthm}[\cref{thm:smoothfp-main}]
\label{thm:intro-C}
Let $G$ be a $p$-adic Lie group. Then there is a fully faithful embedding
\begin{equation*}
\widehat{\D}(\Rep^\sm_{\F_p} G)\hookrightarrow \D((\cal{N}_{\F_p}/\R_{>0})/G^\la)\;,
\end{equation*}
where $\widehat{\D}(\Rep^\sm_{\F_p} G)$ denotes the left-complete derived category of smooth $G$-representations on solid $\F_p$-modules.
\end{introthm}

In particular, we propose that a tentative geometrisation of the mod $p$ local Langlands correspondence would produce a stack $\mathrm{Bun}_G^\la$ over $\cal{N}_{\F_p}/\R_{>0}$ stratified by classifying stacks of group stacks indexed by the Kottwitz set $B(G)$ such that, for $b\in B(G)$ basic, the corresponding stratum is given by $*/G_b^\la$, where $G_b$ is the pure inner form of $G$ associated to $b$. Then the usual formulation of the mod $p$ local Langlands correspondence in terms of smooth $\F_p$-representations would be recovered by virtue of the fully faithful embedding above; this point of view was first suggested to us by Peter Scholze.

Finally, as an application of our geometric approach to locally analytic representations in mixed characteristic, we use the six-functor formalism for quasicoherent sheaves on analytic stacks to establish Poincaré duality for cohomology of locally analytic representations. This takes the following shape:

\begin{introthm}[\cref{prop:duality-prim}, \cref{thm:duality-main}]
\label{thm:intro-D}
Assume that $G$ is a compact $p$-adic Lie group and let $*\coloneqq \cal{N}_{|p|<1}/\R_{>0}$. The morphism
\begin{equation*}
f: */G^\la\rightarrow *
\end{equation*}
is weakly cohomologically proper and cohomologically smooth. In particular, we have
\begin{equation*}
f_*(E^\vee\tensor f^!1)\cong (f_*E)^\vee
\end{equation*}
for any $E\in\D(*/G^\la)$.
\end{introthm}

\subsection{Organisation of the paper}

In §2, we begin by recalling the definition of the stack of norms and proving some of its basic properties with respect to the six-functor formalism on analytic stacks. We also provide some results about its category of quasicoherent sheaves and, in particular, we show that perfect complexes on $\cal{N}_{|p|<1}/\R_{>0}$ are equivalent to perfect complexes over $\Z_p$. We then move on to constructing the functor $M\mapsto M^\la$ from (paracompact) $p$-adic manifolds to analytic stacks over $\cal{N}_{|p|<1}/\R_{>0}$ in §3 and prove its main properties. Subsequently, in §4, we establish some complements about the theory of locally analytic vectors in mixed characteristic to \cite{Porat2} and introduce the notion of $h\dagger$-analytic vectors in mixed characteristic. In §5, we define the stacks $G^{h\dagger\text{-}\an}$ over $\cal{N}_{|p|<1}$ for any $h\in\R$ and show that quasicoherent sheaves on the classifying stacks of $G^\la$ and $G^{h\dagger\text{-}\an}$ recover locally analytic and $h\dagger$-analytic representations, respectively. We then establish the fully faithful embedding of smooth $\F_p$-representations into quasicoherent sheaves on $(\cal{N}_{\F_p}/\R_{>0})/G^\la$ in §6 and, finally, prove several properties of the classifying stacks $*/G^\la$ and $*/G^{h\dagger\text{-}\an}$ with respect to the six-functor formalism on analytic stacks in §7. The appendix contains some miscellaneous results about binomial coefficients that are used in the main body, the arguments for which are mostly combinatorial.

\subsection*{Notations and conventions}

We freely use the language of Clausen--Scholze's analytic stacks throughout and, in particular, work in the light setup of condensed mathematics. Unless otherwise stated, all our modules and tensor products are solid. If $S$ is a topological space, we usually also use $S$ to denote its Betti stack, without explicitly using the superscript $(-)^{\Betti}$; in particular, we only ever write $\R_{>0}$ when we mean the Betti stack of $\R_{>0}$. We will frequently make use of the notation
\begin{equation*}
\Z(\!(\pi)\!)\langle T\rangle_{\leq 1}\coloneqq \colim_{\epsilon>0} \Z(\!(\pi)\!)\langle \pi^\epsilon T\rangle\;.
\end{equation*}

\subsection*{Acknowledgements}
 
I thank Saverio Caleca, Juan Esteban Rodríguez Camargo, Arthur-César Le Bras, Gal Porat and Peter Scholze for several helpful discussions and comments on a draft. This paper was written during my time as a PhD student at the Max Planck Institute for Mathematics in Bonn and I would like to thank the institute for its hospitality.

\section{The stack of norms}

We first recall the definition of the stack of norms $\cal{N}$ from \cite[Lecture 20]{AnStacks}. For any analytic ring $R$ over $\Z_\solid$, there is an ``algebraic'' affine line $\A^{1, \alg}_R$ over $R$ defined as the analytic spectrum of $R[T]$ endowed with the induced analytic ring structure from $R$. Moreover, gluing two copies of $\A^{1, \alg}_R$ along $\G_{m, R}^\alg=\AnSpec R[T, T^{-1}]$, where the analytic ring structure is again induced, we obtain the projective line $\P^1_R$ over $R$.

\begin{defi}
Let $R$ be an analytic ring over $\Z_\solid$. A \emph{norm} on $R$ is a map
\begin{equation*}
|\cdot|: \P^1_R\rightarrow [0, \infty]
\end{equation*}
satisfying the following properties:
\begin{enumerate}[label=(\roman*)]
\item The restriction of $|\cdot |$ along the zero section $\AnSpec R\xrightarrow{0} \P^1_R$ factors through $\{0\}\subseteq [0, \infty]$.
\item $|\cdot |$ intertwines the map $T\mapsto T^{-1}$ on $\P^1_R$ with the map $\lambda\mapsto \lambda^{-1}$ on $[0, \infty]$.
\item Let $\A^{1, \an}_R$ be the preimage of $[0, \infty)$ under $|\cdot |$. Then $|\cdot |$ intertwines the multiplication map on $\A^{1, \an}_R$ with the multiplication map on $[0, \infty)$.
\item The composite map
\begin{equation*}
\AnSpec R[\hat{T}]\rightarrow\P^1_R\xrightarrow{|\,\cdot\, |} [0, \infty]
\end{equation*}
factors through $[0, 1]$, where $R[\hat{T}]\coloneqq R\tensor \Z[\![T]\!]$ is the free $R$-algebra on a topologically nilpotent element, which we equip with the induced analytic ring structure from $R$. Moreover, the map $\AnSpec R[\hat{T}]\rightarrow\P^1_R$ becomes an isomorphism over $[0, 1)$.
\end{enumerate}
We let $\cal{N}$ denote the moduli stack of norms on analytic rings over $\Z_\solid$.
\end{defi}

\begin{rem}
\label{rem:norms-defi}
\begin{enumerate}[label=(\roman*)]
\item Note that we do not impose a compatibility condition between $|\cdot |$ and addition; in particular, we do not ask for a version of the triangle inequality.
\item Any element $f\in R(*)$ induces a map $\AnSpec R\rightarrow \AnSpec\A^{1, \alg}_R$ and hence we obtain a composite map
\begin{equation*}
|f|: \AnSpec R\rightarrow\AnSpec \A^{1, \alg}_R\xrightarrow{|\,\cdot\, |} [0, \infty]\;.
\end{equation*}
Then (i) says that $|0|$ is identically zero while (ii) implies that $|1|$ is identically $1$. 
\item Norms are functorial in the following sense: Given a normed analytic ring $R$ and a map $R\rightarrow R'$ of analytic rings, there is an induced norm on $R'$ given by
\begin{equation*}
\P^1_{R'}\rightarrow\P^1_R\xrightarrow{|\,\cdot\, |} [0, \infty]\;.
\end{equation*}
\item The map $\A^{1, \an}_R\rightarrow\P^1_R$ factors through $\A^{1, \alg}_R$. Indeed, this is because the pullback of $\A^{1, \an}_R$ along $\{\infty\}\subseteq [0, \infty]$ and hence \emph{a fortiori} along $\infty: \AnSpec R\rightarrow\P^1_R$ vanishes. \qedhere
\end{enumerate}
\end{rem}

We start by recalling some basic properties of the stack $\cal{N}$ from \cite[Lecture 20]{AnStacks}. To prepare, we record a useful lemma about $!$-covers.

\begin{lem}
\label{lem:norms-compactifycover}
Let $\AnSpec B\rightarrow \AnSpec A$ be a map of affine analytic stacks and consider its canonical compactification $\AnSpec \ol{B}^{/A}\rightarrow\AnSpec A$, i.e.\ the analytic ring $\ol{B}^{/A}$ has the same underlying condensed ring as $B$ and is equipped with the induced analytic ring structure from $A$. If $\AnSpec B\rightarrow\AnSpec A$ is a $!$-cover, then $\AnSpec\ol{B}^{/A}\rightarrow\AnSpec A$ is a $!$-cover as well.
\end{lem}
\begin{proof}
By assumption, the map $f: X\coloneqq \AnSpec B\rightarrow\AnSpec A\eqqcolon Y$ is a $!$-cover and hence the image of $f_!: \D(B)\rightarrow\D(A)$ generates the unit $A\in\D(A)$ under cofibres, shifts and retracts. Indeed, this is because the functor $f^!: \D(Y)\rightarrow \lim_{n\in\Delta} \D(X^{n/Y})$ is an equivalence by assumption, where $X^{n/Y}$ denotes the $n$-fold self-fibre product of $X$ over $Y$ and $\Delta$ is the simplex category. One easily sees that the left-adjoint of $f^!$ is given by $(M_n)_n\mapsto \colim_{n\in\Delta} f_{n!} M_n$, where $f_n: X^{n/Y}\rightarrow Y$ is the projection, and hence full faithfulness of $f^!$ amounts to the counit
\begin{equation*}
\colim_{n\in\Delta} f_{n!} f_n^!M\rightarrow M
\end{equation*}
being an isomorphism. Applied to $M=A$, this yields $\colim_{n\in\Delta} f_{n!} f_n^!A\cong A$ and since $A\in\D(A)$ is compact, we conclude that $A$ is a retract of $\colim_{n\in\Delta_{\leq m}} f_{n!} f_n^!A$ for some $m\geq 0$, as claimed; here, $\Delta_{\leq m}$ is the full subcategory of $\Delta$ spanned by totally ordered sets with at most $m$ elements.

However, for $p: \AnSpec \ol{B}^{/A}\rightarrow\AnSpec A$, this \emph{a fortiori} implies that the image of the forgetful functor $p_!=p_*: \D(\ol{B}^{/A})\rightarrow \D(A)$ generates the unit under cofibres, shifts and retracts. In other words, we see that $\ol{B}^{/A}$ is descendable in $\D(A)$ and as $\ol{B}^{/A}$ has the induced analytic ring structure from $A$, this shows that $\AnSpec \ol{B}^{/A}\rightarrow\AnSpec A$ is a $!$-cover, as desired. 
\end{proof}

\begin{prop}
\label{prop:norms-normhalfsurjects}
For any normed analytic ring $R$ over $\Z_\solid$, the map
\begin{equation*}
\{|T|=1/2\}\subseteq \P^1_R\rightarrow\AnSpec R
\end{equation*}
is surjective. 
\end{prop}
\begin{proof}
By assumption, the map $|\cdot |: \AnSpec R[\hat{T}]\rightarrow [0, \infty]$ factors through $[0, 1]$ and becomes isomorphic to $|\cdot |:\P^1_R\rightarrow [0, \infty]$ over $[0, 1)$. Let $S$ be the Cantor set and pick a surjection $S\rightarrow [0, 1]$. As $\AnSpec \Cont(S, \Z)\rightarrow [0, 1]$ is surjective, where $\Cont(S, \Z)$ is equipped with the induced analytic ring structure from $\Z_\solid$, there exists a $!$-cover $\AnSpec R'\rightarrow\AnSpec R[\hat{T}]$ such that the composite map $\AnSpec R'\rightarrow \AnSpec R[\hat{T}]\rightarrow [0, 1]$ admits a lift to $S$ and by \cref{lem:norms-compactifycover} we may assume that $R'$ is equipped with the induced analytic ring structure from $R[\hat{T}]$ and hence from $R$.

Now consider the sheaf $\Z_{\{1/2\}}\in\mathrm{Shv}(S, \D(\Z_\solid))\cong \D(S)$ supported on the preimage of $1/2\in [0, 1]$, where the equivalence is due to \cite[Cor.\ II.1.2]{RealLLC}. Since $S$ is profinite, we can write this as a colimit of constant sheaves supported on clopen neighbourhoods of the preimage of $1/2$ in $S$, which are in turn retracts of the constant sheaf. As the constant sheaf is sent to a connective object under the pullback functor $\D(S)\rightarrow \D(R')$, this implies that $\Z_{\{1/2\}}$ is sent to a connective idempotent algebra $A\in \D(R')$ and hence
\begin{equation*}
\AnSpec R'\times_{[0, 1]} \{1/2\}\cong \AnSpec A
\end{equation*}
with $A$ being equipped with the induced analytic ring structure from $R$. 

Clearly, it suffices to show that $A$ is descendable over $R$. As the pushforward of $A$ along $\AnSpec A\rightarrow \AnSpec R[\hat{T}]\times_{[0, 1]} \{1/2\}=\{|T|=1/2\}$ generates the unit object under shifts, cofibres and retracts by descendability of $R[\hat{T}]\rightarrow R'$, we are further reduced to showing that the pushforward of the unit along
\begin{equation*}
\{|T|=1/2\}\rightarrow\AnSpec R
\end{equation*} 
is descendable in $\D(R)$. For this, we will produce a section of the map $R\rightarrow R\Gamma(\{|T|=1/2\}, \O)$. Namely, observe that there is a bicartesian diagram
\begin{equation*}
\begin{tikzcd}
\O_{\P^1_R}\ar[r]\ar[d] & \O_{\{|T|\leq 1/2\}}\ar[d] \\
\O_{\{|T|\geq 1/2\}}\ar[r] & \O_{\{|T|=1/2\}}
\end{tikzcd}
\end{equation*}
in $\D(\P^1_R)$ pulled back from the corresponding diagram in $\mathrm{Shv}([0, 1], \D(\Z_\solid))$ and hence giving a map $R\Gamma(\{|T|=1/2\}, \O)\rightarrow R$ amounts to giving maps $R\Gamma(\{|T|\leq 1/2\}, \O)\rightarrow R$ and $R\Gamma(\{|T|\geq 1/2\}, \O)\rightarrow R$ which are compatible upon restriction to $R\Gamma(\P^1_R, \O)\cong R$. However, for $R\Gamma(\{|T|\leq 1/2\}, \O)\rightarrow R$ we can just take the pullback of sections along the map $0: \AnSpec R\rightarrow\P^1_R$ while for $R\Gamma(\{|T|\geq 1/2\}, \O)\rightarrow R$ we instead take pullback of sections along $\infty: \AnSpec R\rightarrow\P^1_R$. 
\end{proof}

\begin{prop}
\label{prop:norms-normviapseudouniformiser}
For any analytic ring $R$ over $\Z_\solid$, there is an equivalence of anima
\begin{equation*}
\{\AnSpec R\rightarrow\cal{N}\text{ together with }\pi\in R(*)\text{ such that }|\pi|=1/2\}\cong \{\text{maps }\Z(\!(\pi)\!)\rightarrow R\}
\end{equation*}
\end{prop}
\begin{proof}
For the map from left to right, we note that $\pi$ determines a map
\begin{equation*}
\AnSpec R\rightarrow\A^{1, \alg}_R\times_{[0, \infty]} \{1/2\}\cong \AnSpec R[\hat{T}]\times_{[0, 1]} \{1/2\}\rightarrow \AnSpec R[\hat{T}]
\end{equation*}
and hence $\pi$ is topologically nilpotent. It is also a unit since 
\begin{equation*}
\AnSpec R\xrightarrow{\pi} \P^1_R\xrightarrow{T\mapsto T^{-1}}\P^1_R\rightarrow [0, \infty]
\end{equation*}
is supported above $2\in [0, \infty]$ by assumption and hence factors through $\A^{1, \an}_R\subseteq \A^{1, \alg}_R$.

It remains to show that any topologically nilpotent unit $\pi\in R(*)$ determines a unique norm with $|\pi|=1/2$. By descent, we may assume that $\pi$ admits all roots in $R(*)$. By \cite[Prop.\ II.1.3]{RealLLC}, the data of a norm $|\cdot |: \P^1_R\rightarrow [0, \infty]$ amounts to the data of a family $\{A_\lambda\}_{\lambda\in [0, \infty]}$ of idempotent algebras in $\D(\P^1_R)$ corresponding to the pullbacks of the constant sheaves supported on the intervals $[0, \lambda]$ for $\lambda\in [0, \infty]$. As we must have $A_\infty=\O_{\P^1_R}$, we can assume $\lambda<\infty$ and then all $A_\lambda$ will be idempotent algebras in $\D(\A^{1, \alg}_R)=\D(R[T])$. 

By multiplicativity of the norm, we must have
\begin{equation*}
A_{(1/2)^{m/n}}\cong A_1\tensor_{R[T], T\mapsto \pi^{-m/n} T} R[T] 
\end{equation*}
for all $m, n\in\Z$ with $n\neq 0$; as we moreover have $A_\lambda=\colim_{\lambda'>\lambda} A_{\lambda'}$ for all $\lambda$ due to $[0, \lambda]=\bigcap_{\lambda'>\lambda} [0, \lambda']$ and rational powers of $1/2$ form a cofinal system among all $\lambda'>\lambda$, this shows that all $A_\lambda$ are uniquely determined by $A_1$. To determine $A_1$, we recall that the map
\begin{equation*}
\AnSpec R[\hat{T}]\rightarrow\P^1_R\rightarrow [0, \infty]
\end{equation*}
has to factor over $[0, 1]$ and hence there is a map $R[\hat{T}]\rightarrow A_1$. At the same time, the map $\AnSpec R[\hat{T}]\rightarrow\P^1_R$ becomes an isomorphism over $[0, 1)$ and thus there are maps $A_\lambda\rightarrow R[\hat{T}]$ for all $\lambda<1$. Using multiplicativity once again, we conclude that
\begin{equation*}
A_1\cong \colim_{\lambda>1} A_\lambda\cong \colim_{m/n>0} R[\widehat{\pi^{m/n} T}]
\end{equation*}
by cofinality and hence also $A_1$ is uniquely determined by the requirement that $|\pi|=1/2$. Finally, one easily checks that the map $\P^1_R\rightarrow [0, \infty]$ induced by the above choice of idempotent algebras $A_\lambda$ is always a norm on $R$, which finishes the proof.
\end{proof}

By the previous proposition, there is a unique norm on $\Z_\solid(\!(\pi)\!)$ satisfying $|\pi|=1/2$ and we let $\ol{\T}$ denote the subspace $\{|T|=1\}$ of $\P^1_{\Z(\!(\pi)\!)}$. Note that $\ol{\T}$ is the analytic spectrum of
\begin{equation*}
\Z(\!(\pi)\!)\langle T^{\pm 1}\rangle_{\leq 1}\coloneqq \colim_{\epsilon>0} \Z(\!(\pi)\!)\langle \pi^{\epsilon} T^{\pm 1}\rangle
\end{equation*}
equipped with the induced analytic ring structure.

\begin{cor}
\label{cor:norms-pres}
We have 
\begin{equation*}
\cal{N}\cong \AnSpec \Z_\solid(\!(\pi)\!)\,\big/\,\ol{\T}\;,
\end{equation*}
where $\ol{\T}$ acts by multiplication on $\pi$.
\end{cor}
\begin{proof}
Combining \cref{prop:norms-normhalfsurjects} and \cref{prop:norms-normviapseudouniformiser}, we already know that there is a surjective map
\begin{equation*}
\AnSpec \Z_\solid(\!(\pi)\!)\rightarrow \cal{N}
\end{equation*}
defined by requiring that $|\pi|=1/2$. Thus, we are done once we show that the self-fibre product of this map is given by $\ol{\T}$. However, given two pseudouniformisers $\pi, \pi'\in R$, the norms $|\cdot |_\pi$ and $|\cdot |_{\pi'}$ induced by requiring $|\pi|_\pi=1/2$ or $|\pi'|_{\pi'}=1/2$, respectively, are the same if and only if $x\coloneqq \pi'/\pi$ satisfies $|x|_\pi=1$, i.e.\ if and only if $x$ defines a section of $\ol{\T}$ over $\AnSpec \Z_\solid(\!(\pi)\!)$. This yields the claim.
\end{proof}

The above presentation of the stack of norms will be of utmost importance in what follows. For instance, we can use it to establish that $\cal{N}$ is prim.

\begin{lem}
\label{lem:norms-prim}
The map $f: \cal{N}\rightarrow \AnSpec\Z_\solid$ is prim.
\end{lem}
\begin{proof}
Note that $\AnSpec\Z_\solid(\!(\pi)\!)\rightarrow\AnSpec\Z_\solid$ is prim as the analytic ring structure on $\Z_\solid(\!(\pi)\!)$ is induced from the one on $\Z_\solid$; similarly, the map $g: \AnSpec\Z_\solid(\!(\pi)\!)\rightarrow\cal{N}$ is prim as pulling back along the cover $\AnSpec\Z_\solid(\!(\pi)\!)\rightarrow\cal{N}$ yields $\ol{\T}\rightarrow \AnSpec\Z_\solid(\!(\pi)\!)$, which is a map between affine analytic stacks with the induced analytic ring structure. 

By \cite[Cor.\ 4.7.5]{HeyerMann}, it thus suffices to show that $g_*\O\in \D(\cal{N})$ is descendable. For this, we use the presentation of $\cal{N}$ from \cref{cor:norms-pres} to conclude that $g_*\O$ identifies with the $\O(\ol{\T})$-comodule $\O(\ol{\T})$ over $\Z(\!(\pi)\!)$, i.e.\ $g_*\O$ identifies with the regular representation of $\ol{\T}$. Then the natural map $\O\rightarrow g_*\O$ identifies with the canonical map
\begin{equation*}
\Z(\!(\pi)\!)\rightarrow \Z(\!(\pi)\!)\langle T^{\pm 1}\rangle_{\leq 1}\;.
\end{equation*}
Noting that sending a power series in $\Z(\!(\pi)\!)\langle T^{\pm 1}\rangle_{\leq 1}$ to its constant term provides an $\O(\ol{\T})$-colinear section of this map, we deduce that $\O$ is a retract of $g_*\O$ and hence $g_*\O$ is descendable, as desired.
\end{proof}

Observe that $\cal{N}$ admits an action of the Betti stack $\R_{>0}$ given by exponentiating the norm, i.e.\ $\lambda.|\cdot |=|\cdot |^\lambda$ for any $\lambda\in\R_{>0}$. Instead of working over $\cal{N}$, we will want to work over the slightly deeper base $\cal{N}/\R_{>0}$ in the sequel. 

\begin{rem}
\label{rem:norms-maptatehuber}
Given any Tate Huber pair $(B, B^+)$, the choice of a pseudouniformiser $\pi\in B$ determines a map
\begin{equation*}
f_\pi: \AnSpec\hspace{1pt}(B, B^+)_\solid\rightarrow\AnSpec \Z_\solid(\!(\pi)\!)\xrightarrow{|\pi|=1/2}\cal{N}\;.
\end{equation*}
The upshot of replacing $\cal{N}$ by the quotient $\cal{N}/\R_{>0}$ is that, after mapping further to $\cal{N}/\R_{>0}$, all maps as above become identified: Indeed, if $\pi'$ is another pseudouniformiser, then $f_{\pi'}$ is sent to $f_\pi$ by the action of the map $\AnSpec\,(B, B^+)_\solid\rightarrow \R_{>0}$ induced by the map of topological spaces
\begin{equation*}
\operatorname{Spa}(B, B^+)\rightarrow\R_{>0}\;, \hspace{0.5cm} x\mapsto \frac{\log |\pi'(\tilde{x})|}{\log |\pi(\tilde{x})|}\;,
\end{equation*}
where $\tilde{x}$ denotes the maximal generalisation of $x$. We conclude that any Tate adic space admits a canonical map to $\cal{N}/\R_{>0}$.
\end{rem}

We first note that replacing $\cal{N}$ by $\cal{N}/\R_{>0}$ leaves the result of \cref{lem:norms-prim} intact:

\begin{lem}
\label{lem:norms-primmodr}
The map $\cal{N}/\R_{>0}\rightarrow\AnSpec\Z_\solid$ is prim. 
\end{lem}
\begin{proof}
Factoring the given map as $\cal{N}/\R_{>0}\rightarrow */\R_{>0}\rightarrow *$, where we write $*$ for $\AnSpec\Z_\solid$, we see that the first map is prim by \cref{lem:norms-prim} and hence we are reduced to showing that $*/\R_{>0}$ is prim. For this, we first note that $\R_{>0}$ is cohomologically smooth with dualising sheaf concentrated in degree $-1$ by \cite[Thm.\ 4.8.9.(iii)]{HeyerMann} and hence Barr--Beck yields
\begin{equation*}
\D(*/\R_{>0})\cong \Mod_{f^!f_!\O}(\D(\Z_\solid))
\end{equation*}
via $f^!$ for $f: *\rightarrow */\R_{>0}$ and $*=\AnSpec\Z_\solid$. However, since the map $f$ is an $\R_{>0}$-torsor, we conclude that it is suave with dualising complex $\Z[1]$ and hence $f^!f_!\O\cong f^*f_!\O[1]$ is a shift by $1$ of the compactly supported cohomology of $\R_{>0}$ by base change. As the latter is given by $\Z[-1]$, we conclude that $f^!f_!\O\cong\Z$ and as this only admits one algebra structure, we conclude that $\D(*/\R_{>0})\cong \D(\Z_\solid)$ via $f^!$. 

Using the adjunction between $f^!$ and $f_!$, we conclude that $f_!$ identifies with the identity on $\D(\Z_\solid)$ under the above equivalence, and the same is true for the functors $g_!$ and $g^!$, where $g$ denotes the map $g: */\R_{>0}\rightarrow *$. Now one easily sees that $g$ is prim with the prim dual of $\O\in\D(*/\R_{>0})$ corresponding to $\Z$ under the above equivalence by applying \cite[Lem.\ 4.2.2]{dRStack}.
\end{proof}

One advantage of working with the quotient $\cal{N}/\R_{>0}$ is that it admits a straightforward suave cover.

\begin{lem}
\label{lem:norms-suavecover}
The map $\AnSpec\Z_\solid(\!(\pi)\!)\rightarrow \cal{N}/\R_{>0}$ is suave.
\end{lem}
\begin{proof}
We already know from \cref{cor:norms-pres} that the map in question is a cover and hence we may check suaveness after pulling back the map along itself. We claim that
\begin{equation*}
\AnSpec\Z_\solid(\!(\pi)\!)\times_{\cal{N}/\R_{>0}}\AnSpec\Z_\solid(\!(\pi)\!)\cong \overcirc{\DD}^\times_{\Z(\!(\pi)\!)}=\{0<|T|<1\}\subseteq\P^1_{\Z(\!(\pi)\!)}\;,
\end{equation*}
which would prove the lemma as the punctured open disk $\overcirc{\DD}^\times_{\Z(\!(\pi)\!)}$ is open inside $\P^1_{\Z(\!(\pi)\!)}$ and hence suave over $\Z_\solid(\!(\pi)\!)$. To establish the claim, we observe that an $R$-point of the fibre product in question amounts to an $R$-point of $\Z_\solid(\!(\pi)\!)$ together with a pseudouniformiser $\pi'\in R(*)$ such that the norms $|\cdot|_\pi$ and $|\cdot |_{\pi'}$ induced by requiring $|\pi|_\pi=1/2$ and $|\pi'|_{\pi'}=1/2$, respectively, only differ by the action of $\R_{>0}$. This happens if and only if $|\pi'|_\pi=(1/2)^\lambda$ for some $\lambda\in\R_{>0}$ or, in other words, if $\pi'$ corresponds to an $R$-point of $\overcirc{\DD}^\times_{\Z(\!(\pi)\!)}$. This finishes the proof.
\end{proof}

\comment{
\begin{prop}
\label{prop:norms-dnmodrasmodules}
There is an equivalence of categories
\begin{equation*}
\D(\cal{N}/\R_{>0})\cong \D(\Z(\!(\pi)\!)\{T\}^\dagger\oplus T^{-1}\Z(\!(\pi)\!)\langle T^{-1}\rangle_{\leq 1})\;,
\end{equation*}
where 
\begin{equation*}
\Z(\!(\pi)\!)\{T\}^\dagger\coloneqq \colim_k \Z(\!(\pi)\!)\langle \pi^{-k}T\rangle
\end{equation*}
and we endow $\Z(\!(\pi)\!)\{T\}^\dagger\oplus T^{-1}\Z(\!(\pi)\!)\langle T^{-1}\rangle_{\leq 1}$ with the noncommutative $\Z$-linear algebra structure given by the relations
\begin{equation*}
\pi T^i=T^{i+1}\pi\;, \hspace{0.3cm} T^iT^j=\delta_{ij} T^i\;, \hspace{0.3cm} \pi^i\pi^j=\pi^{i+j}\;.
\end{equation*}
Under this equivalence, the cohomology of a quasicoherent sheaf $M$ on $\cal{N}/\R_{>0}$ is given by the retract $M^0$ of $M$ corresponding to the idempotent endomorphism $T^0$. 
\end{prop}
\begin{proof}
We want to apply Barr--Beck to the suave cover $f: \AnSpec\Z_\solid(\!(\pi)\!)\rightarrow \cal{N}/\R_{>0}$ from \cref{lem:norms-suavecover}. First observe that the functor $f^!f_!$ is $\Z_\solid(\!(\pi)\!)$-linear: Indeed, by suaveness, we know that $f^!$ commutes with colimits and the same is true for $f_!$ and hence the equivalence of functors $f^!f_!\cong f^!f_!\O\tensor (-)$ may be checked on generators of the category $\D(\Z_\solid(\!(\pi)\!))$. As a system of generators is given by pullbacks of solid $\Z$-modules, we have to check that
\begin{equation*}
f^!f_!f^*g^*M\cong f^!f_!\O\tensor f^*g^*M
\end{equation*}
for any $M\in\D(\Z_\solid)$ and $g: \cal{N}/\R_{>0}\rightarrow\AnSpec\Z_\solid$. However, by the projection formula, we have
\begin{equation*}
f^!f_!f^*g^*M\cong f^!(f_!\O\tensor g^*M)\cong f^!f_!\O\tensor f^*g^*M\;,
\end{equation*}
where the last step uses the suaveness of $f$, as desired. Overall, we thus conclude that
\begin{equation*}
\D(\cal{N}/\R_{>0})\cong \Mod_{f^!f_!\O}(\D(\Z_\solid(\!(\pi)\!))\;.
\end{equation*}
By suave base change as in \cite[Lem.\ 4.5.13]{HeyerMann}, we see that $f^!f_!\O$ identifies with compactly supported cohomology of the punctured open disk over $\Z(\!(\pi)\!)$ up to a shift using the calculation from the proof of \cref{lem:norms-suavecover} and hence we obtain
\begin{equation*}
f^!f_!\O\cong \Z(\!(\pi)\!)\{T\}^\dagger\oplus T^{-1}\Z(\!(\pi)\!)\langle T^{-1}\rangle_{\leq 1}
\end{equation*}
by excision, as desired.

For the second part, we note that the unit on $\cal{N}/\R_{>0}$ corresponds to the module structure on $\Z(\!(\pi)\!)$ given by
\begin{equation*}
T^i\cdot \pi^j=\delta_{ij}\pi^j\;, \hspace{0.3cm} \pi^i\cdot \pi^j=\pi^{i+j}
\end{equation*}
and our task is to compute $\RHom(\Z(\!(\pi)\!), M)$ as modules over $\Z(\!(\pi)\!)\{T\}^\dagger\oplus T^{-1}\Z(\!(\pi)\!)\langle T^{-1}\rangle_{\leq 1}$. For this, we observe that $\Z(\!(\pi)\!)$ is the retract of $\Z(\!(\pi)\!)\{T\}^\dagger\oplus T^{-1}\Z(\!(\pi)\!)\langle T^{-1}\rangle_{\leq 1}$ corresponding to the idempotent endomorphism $T^0$. Consequently, the space $\RHom(\Z(\!(\pi)\!), M)$ must be the retract of $\RHom(\Z(\!(\pi)\!)\{T\}^\dagger\oplus T^{-1}\Z(\!(\pi)\!)\langle T^{-1}\rangle_{\leq 1}, M)=M$ corresponding to the idempotent endomorphism $T^0$, as desired.
\end{proof}
}

\begin{prop}
\label{prop:norms-embed}
Let $S$ be any analytic stack over $\Z_\solid$ and let $\cal{N}_S$ denote the base change of $\cal{N}$ along $S\rightarrow\AnSpec\Z_\solid$. Then pullback along the map $\cal{N}_S/\R_{>0}\rightarrow S$ induces a fully faithful embedding
\begin{equation*}
\D(S)\hookrightarrow\D(\cal{N}_S/\R_{>0})\;.
\end{equation*}
In particular, we have fully faithful embeddings
\begin{equation*}
\D(\Z_\solid)\hookrightarrow \D(\cal{N}/\R_{>0})\;, \hspace{0.5cm} \D(\F_{p, \solid})\hookrightarrow \D(\cal{N}_{\F_p}/\R_{>0})\;.
\end{equation*}
\end{prop}
\begin{proof}
\comment{
By \cref{lem:norms-primmodr}, the map $f: \cal{N}/\R_{>0}\rightarrow \AnSpec\Z_\solid$ is prim and from \cref{prop:norms-dnmodrasmodules} we conclude that $f_*\O$ is given by the degree zero subspace of $\Z(\!(\pi)\!)$ with respect to the torus action, hence $f_*\O\cong \Z$. By base change for prim maps as in \cite[Lem.\ 4.5.13]{HeyerMann}, we deduce that any base change of $f$ is also prim and satisfies $f_*\O\cong \O$, whence the claimed full faithfulness follows from the projection formula.
}
We first note that the pullback functor $\D(\cal{N}_S/\R_{>0})\rightarrow\D(\cal{N}_S)$ is fully faithful: Indeed, as $\R_{>0}$ is contractible, this is true for pullback along any $\R_{>0}$-torsor by \cite[Cor.\ II.1.2]{RealLLC}. Thus, we are reduced to showing that $\D(S)\hookrightarrow\D(\cal{N}_S)$ is fully faithful. For this, letting $f: \cal{N}\rightarrow \AnSpec\Z_\solid$ be the projection, we will show that $f$ is prim and that $f_*\O\cong \O$. Then any base change $f'$ of $f$ is prim as well and satisfies $f'_*\O\cong \O$ by base change, see \cite[Lem.\ 4.5.13]{HeyerMann}, hence the claim follows by the projection formula.

Primness of $f$ is the content of \cref{lem:norms-prim}, so it remains to show that $f_*\O\cong \O$. Using the presentation of $\cal{N}$ from \cref{cor:norms-pres}, this amounts to checking that the augmented cosimplicial diagram
\begin{equation*}
\begin{tikzcd}
\Z\ar[r] & \Z(\!(\pi)\!)\ar[r,shift left=.5ex]
  \ar[r,shift right=.5ex] & \Z(\!(\pi)\!)\langle X_1^{\pm 1}\rangle_{\leq 1} \ar[r,shift left=1ex]\ar[r]
  \ar[r,shift right=1ex] & \Z(\!(\pi)\!)\langle X_1^{\pm 1}, X_2^{\pm 1}\rangle_{\leq 1} \ar[r, shift left=1.5ex] \ar[r,shift left=.5ex]
  \ar[r,shift right=.5ex] \ar[r, shift right=1.5ex] & \cdots
\end{tikzcd}
\end{equation*}
obtained from the \v{C}ech nerve of the surjection $\AnSpec\Z_\solid(\!(\pi)\!)\rightarrow \cal{N}$ is exact. Here, the codegeneracy maps $s^j: C^n\rightarrow C^{n-1}$ are given by $X_{j+1}\mapsto 1$; moreover, the coface maps $d^j: C^{n-1}\rightarrow C^n$ for $0<j<n$ are given by $X_j\mapsto X_jX_{j+1}, X_{j+1}\mapsto X_{j+2}, \dots, X_{n-1}\mapsto X_n$ while $d^n$ is the canonical inclusion and $d^0$ is given by $\pi\mapsto \pi X_1$ and $X_i\mapsto X_{i+1}$ for $0<i<n$.

To show exactness, we produce an extra codegeneracy $s^{-1}$ given by the formula
\begin{equation}
\label{eq:norms-extradegen}
s^{-1}\left(\sum_{\ul{\alpha}}\sum_n c_{n, \ul{\alpha}}\pi^n\ul{X}^{\ul{\alpha}}\right)=\sum_{\ul{\alpha}} \sum_n \delta_{\alpha_1, n} c_{n, \ul{\alpha}} \pi^n\ul{X}^{\ul{\alpha}'}\;,
\end{equation}
where $c_{n, \ul{\alpha}}\in\Z$ and $\ul{\alpha}=(\alpha_1,\alpha_2, \dots)$ is a multiindex with almost all entries zero, for which we set $\ul{X}^{\ul{\alpha}}\coloneqq \prod_i X_i^{\alpha_i}$; moreover, $\delta_{ij}$ is the Kronecker delta and $\ul{\alpha}'\coloneqq (\alpha_2, \alpha_3, \dots)$ denotes the left shift of $\ul{\alpha}$. We first note that $s^{-1}$ is well-defined, i.e.\ that the right-hand side of (\ref{eq:norms-extradegen}) converges: Indeed, we have $|\ul{\alpha}'|\leq |\ul{\alpha}|$, where we put $|\ul{\alpha}|\coloneqq \sum_i |\alpha_i|$, and hence convergence of the right-hand side of (\ref{eq:norms-extradegen}) follows from the fact that the input of $s^{-1}$ on the left-hand side converges. Thus, it remains to check the cosimplicial identities
\begin{equation*}
s^{-1}\circ d^0=\id\;, \hspace{0.3cm} s^{-1}\circ d^j=d^{j-1}\circ s^{-1}\;, \hspace{0.3cm} s^{-1}\circ s^j=s^{j-1}\circ s^{-1}
\end{equation*}
for $j\geq 1$, but this is easy and we omit the details.
\end{proof}

In fact, we will restrict our base even further to the (quotient by $\R_{>0}$) of the open substack of $\cal{N}$ where $p$ has norm less than $1$, which we denote by $\cal{N}_{|p|<1}$. In other words, the stack $\cal{N}_{|p|<1}$ is the moduli space of normed analytic rings $R$ such that the map
\begin{equation*}
|p|: \AnSpec R\rightarrow \A^{1, \alg}_R\xrightarrow{|\,\cdot\, |} [0, \infty]
\end{equation*}
factors through $[0, 1)$. Note that any such normed analytic ring $R$ will automatically be a $\Z_p$-algebra, i.e.\ there is a natural map $\cal{N}_{|p|<1}/\R_{>0}\rightarrow\AnSpec\Z_{p, \solid}$. We first observe that the additional restriction $|p|>0$ collapses the stack of norms (modulo $\R_{>0}$) to a point.

\begin{lem}
\label{lem:norms-0<p<1}
There is an isomorphism
\begin{equation*}
\cal{N}_{0<|p|<1}\cong \AnSpec\Q_{p, \solid}\times (0, 1)\;,
\end{equation*}
whose postcomposition with projection onto the second factor is given by the norm of $p$. In particular, we have
\begin{equation*}
\cal{N}_{0<|p|<1}/\R_{>0}\cong \AnSpec\Q_{p, \solid}\;.
\end{equation*}
\end{lem}
\begin{proof}
Clearly, there is a canonical map $\cal{N}_{0<|p|<1}\rightarrow\AnSpec\Q_{p, \solid}\times (0, 1)$ whose second component is given by the norm of $p$. As this is equivariant with respect to the exponentiation action of $\R_{>0}$ on the target, it thus suffices to show that $\cal{N}_{0<|p|<1}/\R_{>0}\cong \AnSpec\Q_{p, \solid}$. For this, we note that $\R_{>0}$ acts transitively on $(0, 1)$ via exponentiation and hence $\cal{N}_{0<|p|<1}/\R_{>0}\cong \cal{N}_{|p|=1/2}$. Finally, by \cref{prop:norms-normviapseudouniformiser}, we see that a map from $\AnSpec R$ to $\cal{N}_{|p|=1/2}$ for some analytic ring $R$ over $\Z_\solid$ corresponds to a map $\Z_\solid(\!(\pi)\!)/(\pi-p)\rightarrow R$ and this yields the claim due to $\Z(\!(\pi)\!)/(\pi-p)\cong \Q_p$.
\end{proof}

We have the following analogue of \cref{prop:norms-embed}:

\begin{prop}
\label{prop:norms-embedzp}
Pullback induces a fully faithful embedding
\begin{equation*}
\D(\Z_{p, \solid})\hookrightarrow \D(\cal{N}_{|p|<1}/\R_{>0})\;.
\end{equation*}
\end{prop}
\begin{proof}
We first observe that it suffices to show that $\D(\Z_{p, \solid})\rightarrow\D(\cal{N}_{|p|<1})$ is fully faithful. Indeed, this implies the claim as $\D(\cal{N}_{|p|<1}/\R_{>0})\rightarrow\D(\cal{N}_{|p|<1})$ due to $\R_{>0}$ being contractible, see \cite[Cor.\ II.1.2]{RealLLC}. From \cref{lem:norms-0<p<1}, we now deduce that
\begin{equation*}
\D(\cal{N}_{0<|p|<1})\cong \operatorname{Shv}((0, 1), \D(\Q_{p, \solid}))
\end{equation*}
by \cite[Prop.\ II.1.1]{RealLLC} and similarly for $\D(\cal{N}_{0<|p|\leq 1/2})$. We conclude that the essential image of $\D(\Z_{p, \solid})$ in $\D(\cal{N}_{0<|p|<1})$ embeds fully faithfully into the category $\D(\cal{N}_{0<|p|\leq 1/2})$ as this essential image precisely consists of the constant sheaves. Due to
\begin{equation*}
\D(\cal{N}_{|p|<1})\cong \D(\cal{N}_{|p|\leq 1/2})\times_{\D(\cal{N}_{0<|p|\leq 1/2})} \D(\cal{N}_{0<|p|<1})\;,
\end{equation*}
it thus suffices to show that $\D(\Z_{p, \solid})$ embeds fully faithfully into $\D(\cal{N}_{|p|\leq 1/2})$.

Now note that
\begin{equation*}
\cal{N}_{|p|\leq 1/2}\cong \AnSpec \Z_\solid(\!(\pi)\!)\langle \tfrac{p}{\pi}\rangle_{\leq 1}\,\big/\,\ol{\T}
\end{equation*}
by \cref{cor:norms-pres} and hence the same argument as in \cref{lem:norms-prim} shows that the map $f: \cal{N}_{|p|\leq 1/2}\rightarrow \AnSpec\Z_{p, \solid}$ is prim. Thus, as in the proof of \cref{prop:norms-embed}, it suffices to show that $f_*\O\cong \O$. To this end, we prove more generally that
\begin{equation*}
\begin{tikzcd}
\Z(\!(\pi)\!)\langle T\rangle_{\leq 1}\ar[r,shift left=.5ex]
  \ar[r,shift right=.5ex] & \Z(\!(\pi)\!)\langle T\rangle_{\leq 1}\langle X_1^{\pm 1}\rangle_{\leq 1} \ar[r,shift left=1ex]\ar[r]
  \ar[r,shift right=1ex] & \Z(\!(\pi)\!)\langle T\rangle_{\leq 1}\langle X_1^{\pm 1}, X_2^{\pm 1}\rangle_{\leq 1} \ar[r, shift left=1.5ex] \ar[r,shift left=.5ex]
  \ar[r,shift right=.5ex] \ar[r, shift right=1.5ex] & \cdots
\end{tikzcd}
\end{equation*}
is a cosimplicial resolution of $\Z[\![\pi T]\!]$, where the $\ol{\T}$-action on $T$ is given by division, then the claim follows by modding out $\pi T-p$. For this, we note that the cosimplicial diagram above receives a canonical augmentation map from $\Z[\![\pi T]\!]$ and hence it suffices to construct an extra codegeneracy $s^{-1}$. However, as in the proof of \cref{prop:norms-embed}, one easily checks that 
\begin{equation*}
s^{-1}\left(\sum_{\ul{\alpha}}\sum_{m, n} c_{m, n, \ul{\alpha}}\pi^n T^m\ul{X}^{\ul{\alpha}}\right)=\sum_{\ul{\alpha}} \sum_{m, n} \delta_{\alpha_1, n-m} c_{n, \ul{\alpha}} \pi^n T^m\ul{X}^{\ul{\alpha}'}\;,
\end{equation*}
does the job.
\comment{
This will be given by the formula
\begin{equation}
\label{eq:norms-extracodegeneracyzp}
s^{-1}\left(\sum_{\ul{\alpha}}\sum_{m, n} c_{m, n, \ul{\alpha}}\pi^n T^m\ul{X}^{\ul{\alpha}}\right)=\sum_{\ul{\alpha}} \sum_{m, n} \delta_{\alpha_1, n-m} c_{n, \ul{\alpha}} \pi^n T^m\ul{X}^{\ul{\alpha}'}\;,
\end{equation}
where $c_{n, \ul{\alpha}}\in\Z$ and $\ul{\alpha}=(\alpha_1,\alpha_2, \dots)$ is a multiindex with almost all entries zero, for which we set $\ul{X}^{\ul{\alpha}}\coloneqq \prod_i X_i^{\alpha_i}$; moreover, $\delta_{ij}$ is the Kronecker delta and $\ul{\alpha}'\coloneqq (\alpha_2, \alpha_3, \dots)$ denotes the left shift of $\ul{\alpha}$. We first note that $s^{-1}$ is well-defined, i.e.\ that the right-hand side of (\ref{eq:norms-extracodegeneracyzp}) converges: Indeed, we have $|\ul{\alpha}'|\leq |\ul{\alpha}|$, where we put $|\ul{\alpha}|\coloneqq \sum_i |\alpha_i|$, and hence convergence of the right-hand side of (\ref{eq:norms-extracodegeneracyzp}) follows from the fact that the input of $s^{-1}$ on the left-hand side converges. Thus, it remains to check the cosimplicial identities
\begin{equation*}
s^{-1}\circ d^0=\id\;, \hspace{0.3cm} s^{-1}\circ d^j=d^{j-1}\circ s^{-1}\;, \hspace{0.3cm} s^{-1}\circ s^j=s^{j-1}\circ s^{-1}
\end{equation*}
for $j\geq 1$, but this is easy and we omit the details.
}
\comment{
First note that the simplicial diagram above indeed receives a natural augmentation map from $\Z[\![\pi T]\!]$ as $\pi T$ is fixed by the $\ol{\T}$-action. Now we proceed as in the proof of \cref{prop:norm-embedarbitrarybase} and first observe that
\begin{equation*}
\begin{tikzcd}
\pi^{-1}\Z[\pi^{-1}][T]\ar[r,shift left=.5ex]
  \ar[r,shift right=.5ex] & \pi^{-1}\Z[\pi^{-1}][T][X^{\pm 1}] \ar[r,shift left=1ex]\ar[r]
  \ar[r,shift right=1ex] & \pi^{-1}\Z[\pi^{-1}][T][X^{\pm 1}, Y^{\pm 1}] \ar[r, shift left=1.5ex] \ar[r,shift left=.5ex]
  \ar[r,shift right=.5ex]\ar[r, shift right=1.5ex] & \cdots
\end{tikzcd}
\end{equation*}
is exact by the same argument as in loc.\ cit. Thus, it suffices to show that
\begin{equation*}
\begin{tikzcd}
\Z[\![\pi]\!]\langle T\rangle_{\leq 1}\ar[r,shift left=.5ex]
  \ar[r,shift right=.5ex] & \Z[\![\pi]\!]\langle T\rangle_{\leq 1}\langle X^{\pm 1}\rangle_{\leq 1} \ar[r,shift left=1ex]\ar[r]
  \ar[r,shift right=1ex] & \Z[\![\pi]\!]\langle T\rangle_{\leq 1}\langle X^{\pm 1}, Y^{\pm 1}\rangle_{\leq 1} \ar[r, shift left=1.5ex] \ar[r,shift left=.5ex]
  \ar[r,shift right=.5ex] \ar[r, shift right=1.5ex] & \cdots
\end{tikzcd}
\end{equation*}
is a simplicial resolution of $\Z[\![\pi T]\!]$, which again reduces to showing that
\begin{equation*}
\begin{tikzcd}
\Z[\![\pi]\!]\langle \pi^\epsilon T\rangle\ar[r,shift left=.5ex]
  \ar[r,shift right=.5ex] & \Z[\![\pi]\!]\langle \pi^{\epsilon/2} T\rangle\langle \pi^{\epsilon/2} X^{\pm 1}\rangle \ar[r,shift left=1ex]\ar[r]
  \ar[r,shift right=1ex] & \Z[\![\pi]\!]\langle \pi^{\epsilon/4} T\rangle\langle \pi^{\epsilon/4} X^{\pm 1}, \pi^{\epsilon/4} Y^{\pm 1}\rangle \ar[r, shift left=1.5ex] \ar[r,shift left=.5ex]
  \ar[r,shift right=.5ex] \ar[r, shift right=1.5ex] & \cdots
\end{tikzcd}
\end{equation*}
is a simplicial resolution of $\Z[\![\pi T]\!]$ for each $\epsilon>0$. As all terms in the above simplicial object are $\pi$-complete, the decreasing filtration $\Fil^\bullet$ whose $i$-th step is given by
\begin{equation*}
\begin{tikzcd}
\pi^i\Z[\![\pi]\!]\langle \pi^\epsilon T\rangle\ar[r,shift left=.5ex]
  \ar[r,shift right=.5ex] & \pi^i\Z[\![\pi]\!]\langle \pi^{\epsilon/2} T\rangle\langle \pi^{\epsilon/2} X^{\pm 1}\rangle \ar[r,shift left=1ex]\ar[r]
  \ar[r,shift right=1ex] & \pi^i\Z[\![\pi]\!]\langle \pi^{\epsilon/4} T\rangle\langle \pi^{\epsilon/4} X^{\pm 1}, \pi^{\epsilon/4} Y^{\pm 1}\rangle \ar[r, shift left=1.5ex] \ar[r,shift left=.5ex]
  \ar[r,shift right=.5ex] \ar[r, shift right=1.5ex] & \cdots
\end{tikzcd}
\end{equation*}
is separated and it is moreover compatible with the $\pi T$-adic filtration on $\Z[\![\pi T]\!]$ under the augmentation map we have produced. Thus, we may check the claim on graded pieces, where it turns into the assertion that
\begin{equation*}
\begin{tikzcd}
\gr^i= \pi^i\Z[T]\ar[r,shift left=.5ex]
  \ar[r,shift right=.5ex] & \pi^i\Z[T, X^{\pm 1}] \ar[r,shift left=1ex]\ar[r]
  \ar[r,shift right=1ex] & \pi^i\Z[T, X^{\pm 1}, Y^{\pm 1}] \ar[r, shift left=1.5ex] \ar[r,shift left=.5ex]
  \ar[r,shift right=.5ex] \ar[r, shift right=1.5ex] & \cdots
\end{tikzcd}
\end{equation*}
is a simplicial resolution of $\pi^i T^i\Z$ for each $i\geq 0$. However, this again follows from the identification of $\G_m^\alg$-representations with graded solid abelian groups using the fact that the degree zero piece of $\pi^i\Z[T]$ is given by $\pi^i T^i$.
}
\end{proof}
\comment{
\begin{rem}
Let us reiterate the following statement that was used in the proof: There is an isomorphism $\cal{N}_{0<|p|<1}\cong \AnSpec\Q_{p, \solid}\times (0, 1)$ and, under this isomorphism, the $\R_{>0}$-action on $\cal{N}_{0<|p|<1}$ translates to the exponentiation action of $\R_{>0}$ on $(0, 1)$. In particular, we obtain
\begin{equation*}
\cal{N}_{0<|p|<1}/\R_{>0}\cong \AnSpec\Q_{p, \solid}\;. \qedhere
\end{equation*}
\end{rem}
}

\begin{rem}
Note that the argument from the proof of \cref{prop:norms-embedzp} is stable under base change. In other words, it also shows that there is a fully faithful embedding
\begin{equation*}
\D(S)\hookrightarrow \D(\cal{N}_{|p|<1}/\R_{>0}\times_{\AnSpec\Z_{p, \solid}} S)
\end{equation*}
for any analytic stack $S$ over $\AnSpec\Z_{p, \solid}$.
\end{rem}

Finally, we show that perfect complexes on $\cal{N}_{|p|<1}/\R_{>0}$ just identify with perfect complexes over $\Z_p$.

\begin{prop}
\label{prop:norms-perfzp}
The fully faithful embedding from \cref{prop:norms-embedzp} induces an equivalence of categories
\begin{equation*}
\Perf(\Z_p)\cong \Perf(\cal{N}_{|p|<1}/\R_{>0})\;.
\end{equation*}
\end{prop}
\begin{proof}
As $\R_{>0}$ is contractible, an application of \cite[Cor.\ 2.7]{Syntomic} shows that it suffices to show that $\Perf(\Z_p)\cong \Perf(\cal{N}_{|p|<1})$ via pullback. Due to $\cal{N}_{0<|p|<1}\cong\AnSpec\Q_{p, \solid}\times (0, 1)$, another application of loc.\ cit.\ then further reduces us to showing that $\Perf(\Z_p)\cong\Perf(\cal{N}_{|p|\leq 1/2})$ via pullback along $f: \cal{N}_{|p|\leq 1/2}\rightarrow\AnSpec\Z_{p, \solid}$. As the proof of \cref{prop:norms-embedzp} already yields a fully faithful embedding $\D(\Z_{p, \solid})\hookrightarrow\D(\cal{N}_{|p|\leq 1/2})$, it suffices to show that every perfect complex on $\cal{N}_{|p|\leq 1/2}$ is pulled back from a perfect complex over $\Z_p$.

We first prove that any perfect complex is in the essential image of $f^*$. For this, note that the category $\D(\cal{N}_{|p|\leq 1/2})$ is equivalent to the category of semilinear comodules $M$ over $\Z(\!(\pi)\!)\langle\tfrac{p}{\pi}, X^{\pm 1}\rangle_{\leq 1}$ in $\D(\Z_\solid(\!(\pi)\!)\langle\tfrac{p}{\pi}\rangle_{\leq 1})$; by abuse of notation, we also use $M$ to denote the quasicoherent sheaf on $\cal{N}_{|p|\leq 1/2}$ corresponding to such a comodule. We begin by establishing the following

\bigskip

\textbf{Claim.} Assume that $M$ lies in the heart of the canonical $t$-structure on $\D(\Z_\solid(\!(\pi)\!)\langle\tfrac{p}{\pi}\rangle_{\leq 1})$ and that $M$ is finitely presented as a $\Z(\!(\pi)\!)\langle\tfrac{p}{\pi}\rangle_{\leq 1}$-module. Then $M$ admits a surjection from a finite free module on $\cal{N}_{|p|\leq 1/2}$.

\bigskip

\textit{Proof of the claim.} For simplicity, write $A\coloneqq \Z(\!(\pi)\!)\langle\tfrac{p}{\pi}\rangle_{\leq 1}$. Consider the $\Z_p$-linear map
\begin{equation*}
\mathrm{ev}_0: M\xrightarrow{\mathrm{coact}} M\langle X^{\pm 1}\rangle_{\leq 1}\xrightarrow{X\mapsto 0} M\;,
\end{equation*}
where the first map is the given comodule structure on $M$, and note that it is idempotent by coassociativity. We let $V$ denote the corresponding retract of $M$, i.e.\ the kernel (not the fibre!) of the endomorphism $\id-\,\mathrm{ev}_0$, and note that $V$ is a $\Z_p$-module. Moreover, by construction, the map
\begin{equation*}
V\tensor_{\Z_p} A\rightarrow M
\end{equation*}
is a morphism of comodules, where we use the trivial comodule structure on $V$. We claim that there exists a finite free $\Z_p$-module $W$ together with a map $W\rightarrow V$ such that the composite map
\begin{equation*}
W\tensor_{\Z_p} A\rightarrow V\tensor_{\Z_p} A\rightarrow M
\end{equation*}
is surjective, which would yield the desired conclusion.

First note that it suffices to see that
\begin{equation*}
V(*)\tensor_{\Z_p(*)} A(*)\rightarrow M(*)
\end{equation*}
is surjective. Indeed, as $M(*)$ is finitely generated over $A(*)$, this would imply that $V'\tensor_{\Z_p(*)} A(*)$ already surjects onto $M(*)$ for some finitely generated submodule $V'$ of $V(*)$, which yields the claim. To see this surjectivity, note that $M$ being finitely presented yields $M\cong H^0(M(*)\tensor_{A(*)} A)$. However, this implies that the image of $V(*)\tensor_{\Z_p(*)} A(*)$ in $M(*)$ is dense for the natural topology as counitality and semilinearity of the comodule structure on $M$ shows that any $m\in M(*)$ is a possibly infinite sum of the elements $\pi^k\mathrm{ev}_0(\pi^{-k}m)$ for $k\in\Z$.

This is enough to conclude: Picking a surjection $A^n(*)\rightarrow M(*)$, we see that the preimage in $A^n(*)$ of the image of $V(*)\tensor_{\Z_p(*)} A(*)\rightarrow M(*)$ is dense as well, so it suffices to show that any dense submodule $N$ of $A^n(*)$ is already equal to all of $A^n(*)$. However, letting $e_i$ denote the standard basis vectors, the fact that $N$ is dense in $A^n(*)$ implies that $e_i+v_i\in N$ for some $v_i\in A^n(*)$ all of whose entries are topologically nilpotent. Then the base change matrix from the $e_i$ to the $e_i+v_i$ is invertible, which yields the result. \hfill\qed

\bigskip

To apply the above, take a perfect complex $M$ on $\cal{N}_{|p|\leq 1/2}$ and write its pullback to $\Z(\!(\pi)\!)\langle\tfrac{p}{\pi}\rangle_{\leq 1}$ as a complex 
\begin{equation*}
\dots\rightarrow 0\rightarrow M^a\rightarrow M^{a+1}\rightarrow\dots\rightarrow M^{b-1}\rightarrow M^b\rightarrow 0\rightarrow\dots
\end{equation*}
of finite projective modules over $\Z(\!(\pi)\!)\langle \tfrac{p}{\pi}\rangle_{\leq 1}$, which we also denote by $M$ by abuse of notation. We claim that the map
\begin{equation*}
M^b[-b]\rightarrow M\rightarrow M\langle X^{\pm 1}\rangle_{\leq 1}\;,
\end{equation*}
where the latter arrow comes from the comodule structure on $M$, admits a factorisation through $M^b\langle X^{\pm 1}\rangle_{\leq 1}[-b]$. Indeed, there is a fibre sequence
\begin{equation*}
M^b\langle X^{\pm 1}\rangle_{\leq 1}[-b]\rightarrow M\langle X^{\pm 1}\rangle_{\leq 1}\rightarrow \sigma^{\leq b-1}M\langle X^{\pm 1}\rangle_{\leq 1}
\end{equation*}
and the term on the right is concentrated in cohomological degrees $\leq b-1$ by flatness of $\Z(\!(\pi)\!)\langle X^{\pm 1}\rangle$ as a $\Z_\solid(\!(\pi)\!)$-algebra. As $\Ext^{\geq 1}(M^b, -)$ is identically zero by projectivity of $M^b$, this implies that the composite map
\begin{equation*}
M^b[-b]\rightarrow M\langle X^{\pm 1}\rangle_{\leq 1}\rightarrow \sigma^{\leq b-1}M\langle X^{\pm 1}\rangle_{\leq 1}
\end{equation*}
is nullhomotopic and hence admits a factorisation through $M^b\langle X^{\pm 1}\rangle_{\leq 1}[-b]$, as desired. 

This equips $M^b$ with a $\Z(\!(\pi)\!)\langle\tfrac{p}{\pi}\rangle_{\leq 1}\langle X^{\pm 1}\rangle_{\leq 1}$-comodule structure compatible with the one on $M$ and by the claim we find a surjective map $f^*V\rightarrow M^b$ for some finite free $\Z_p$-module $V$, where we interpret $M^b$ as a sheaf on $\cal{N}_{|p|\leq 1/2}$ via the comodule structure we have constructed. In particular, the induced map $f^*V[-b]\rightarrow M^b[-b]\rightarrow M$ surjects onto $H^b(M)$. Replacing $M$ by the cofibre of $f^*V[-b]\rightarrow M$, which is now concentrated in cohomological degrees $\leq b-1$, and repeating this argument inductively, we obtain a possibly infinite resolution of $M$ by finite free modules on $\cal{N}_{|p|\leq 1/2}$. As the proof of \cref{prop:norms-embedzp} already shows that $f_*\O\cong \O$, any map between finite free modules on $\cal{N}_{|p|\leq 1/2}$ is pulled back from a map between finite free modules over $\Z_p$ and hence we conclude that $M$ is the pullback of a pseudocoherent complex over $\Z_p$.

To finish the proof, we have to show that if $V$ is a pseudocoherent complex over $\Z_p$ whose pullback $f^*V$ to $\cal{N}_{|p|\leq 1/2}$ is perfect, then $V$ is perfect itself. For this, note that $f^*V$ being perfect in particular implies that the base change of $V$ along $\Z_{p, \solid}\rightarrow \Z_\solid(\!(\pi)\!)\langle\tfrac{p}{\pi}\rangle_{\leq 1}$ is perfect. Moreover, while this pullback is in general computed using the solid tensor product, this agrees with the algebraic tensor product in our case as $V$ is pseudocoherent and hence our question is purely algebraic and we may ignore the condensed structures. Thus, it suffices to check that $\Z(\!(\pi)\!)\langle\tfrac{p}{\pi}\rangle_{\leq 1}$ is a faithfully flat $\Z_p$-algebra. However, this is clear: Being flat over $\Z_p$ is the same as being torsionfree and both the special and the generic point of $\Spec\Z_p(*)$ are clearly in the image of the map from $\Spec  \Z(\!(\pi)\!)\langle\tfrac{p}{\pi}\rangle_{\leq 1}(*)$.
\end{proof}

\begin{rem}
Using a similar argument, one can also show that the fully faithful embeddings $\D(\Z_{\solid})\hookrightarrow\D(\cal{N}/\R_{>0})$ and $\D(\F_{p, \solid})\hookrightarrow\D(\cal{N}/\R_{>0})$ from \cref{prop:norms-embed} induce equivalences 
\begin{equation*}
\Perf(\Z)\cong\Perf(\cal{N}/\R_{>0}) \hspace{0.5cm}\text{and}\hspace{0.5cm} \Perf(\F_p)\cong\Perf(\cal{N}_{\F_p}/\R_{>0})\;. \qedhere
\end{equation*}
\end{rem}

\section{$p$-adic manifolds in mixed characteristic}

The goal of this section is to define a functor
\begin{equation*}
\begin{split}
\{\text{paracompact $p$-adic manifolds}\}&\rightarrow\{\text{analytic stacks over $\cal{N}_{|p|<1}/\mathbb{R}_{>0}$}\} \\
M&\mapsto M^\la\;.
\end{split}
\end{equation*}
Here, we use the term \emph{$p$-adic manifold} in the sense of \cite[§8]{Schneider}. The essential difficulty will be to define $\Z_p^\la$ and then to argue that any endomorphism of $\Z_p$ as a $p$-adic manifold induces an endomorphism of the stack $\Z_p^\la$.

First recall from \cref{lem:app-prodbinom} that there are unique integers $a_j\in\Z$ such that
\begin{equation*}
\binom{x}{n}\cdot \binom{x}{m}=\sum_{j=0}^{m+n} a_j\binom{x}{j}
\end{equation*} 
in $\Q[x]$. This allows us to make sense of the following definition, a variant of which has recently been proposed by Porat, see \cite{Porat1}.

\begin{defi}
The algebra of locally analytic $\Z(\!(\pi)\!)$-valued functions on $\Z_p$ is given by
\begin{equation*}
C^\la(\Z_p, \Z(\!(\pi)\!))\coloneqq \colim_{\epsilon>0} \Z(\!(\pi)\!)\left\langle \pi^{\epsilon n}\binom{x}{n}: n\geq 0\right\rangle\;,
\end{equation*}
where the algebra in the colimit is obtained by $\pi$-completing the ring
\begin{equation*}
\Z[\pi]\left[\pi^{\lceil\epsilon n\rceil}\binom{x}{n}: n\geq 0\right]
\end{equation*}
and then inverting $\pi$.
\end{defi}

\begin{rem}
A formal power series $f=\sum_n a_n\binom{x}{n}$ with $a_n\in\Z(\!(\pi)\!)$ is contained in the ring $C^\la(\Z_p, \Z(\!(\pi)\!))$ if and only if $\pi^{-\epsilon n}a_n\rightarrow 0$ in the $\pi$-adic topology for some $\epsilon>0$. In other words, the $\pi$-adic valuations of the $a_n$ grow at least linearly. 
\end{rem}

We immediately warn the reader that the terminology above is a bit of a misnomer: Indeed, given an element $x_0\in\Z_p$, it is \emph{not} always possible to ``evaluate'' an element of $C^\la(\Z_p, \Z(\!(\pi)\!))$ at $x=x_0$ and obtain a well-defined element of $\Z(\!(\pi)\!)$. However, we \emph{can} always produce a well-defined element of $\Z(\!(\pi)\!)\langle \tfrac{p^k}{\pi}\rangle$ for any $k\geq 1$, i.e.\ there are maps
\begin{equation*}
C^\la(\Z_p, \Z(\!(\pi)\!))\tensor_{\Z(\!(\pi)\!)} \Z(\!(\pi)\!)\langle \tfrac{p^k}{\pi}\rangle\rightarrow \operatorname{Maps}(\Z_p, \Z(\!(\pi)\!)\langle \tfrac{p^k}{\pi}\rangle)\;.
\end{equation*}
This is the content of the following lemma.

\begin{lem}
\label{lem:man-eval}
For each $k\geq 1$, there is a well-defined evaluation map
\begin{equation*}
\begin{split}
\mathrm{ev}: C^\la(\Z_p, \Z(\!(\pi)\!))\times \Z_p&\rightarrow \Z(\!(\pi)\!)\langle\tfrac{p^k}{\pi}\rangle \\
(f, x)&\mapsto f(x)\;.
\end{split}
\end{equation*}
\end{lem}
\begin{proof}
We first observe that there is a straightforward map
\begin{equation*}
\mathrm{ev}: C^\la(\Z_p, \Z(\!(\pi)\!))\times \Z\rightarrow \Z(\!(\pi)\!)\;.
\end{equation*}
Indeed, for any $x\in\Z$, the binomial coefficients $\binom{x}{n}$ will be integers and hence $f(x)=\sum_n a_n\binom{x}{n}$ makes sense in $\Z(\!(\pi)\!)$ due to $a_n\rightarrow 0$ in the $\pi$-adic topology. As $\Z(\!(\pi)\!)\langle \tfrac{p^k}{\pi}\rangle$ is $\pi$-adically separated, there can be at most one extension of $\mathrm{ev}$ from $\Z$ to $\Z_p$ and our task is just to check that $\lim_{n\rightarrow\infty} \mathrm{ev}(f, x_n)$ exists in $\Z(\!(\pi)\!)\langle\tfrac{p^k}{\pi}\rangle$ for any $f\in C^\la(\Z_p, \Z(\!(\pi)\!))$ and $x_n\in\Z$ converging $p$-adically.

For this, take any $x=\sum_i p^i x_i\in\Z_p$, where the $x_i$ are integers. Our task is to make sense of the expression
\begin{equation*}
\sum_n a_n\binom{\sum_i p^i x_i}{n}
\end{equation*}
in the ring $\Z(\!(\pi)\!)\langle\tfrac{p^k}{\pi}\rangle$ for any sequence $a_n\in\Z(\!(\pi)\!)$ satisfying $\pi^{-\epsilon n}a_n\rightarrow 0$ in the $\pi$-adic topology for some $\epsilon>0$. As $\Z(\!(\pi)\!)\langle\tfrac{p^k}{\pi}\rangle$ only gets smaller if we enlarge $k$, we may assume that $\epsilon k\geq \tfrac{1}{p-1}$. Then \cref{lem:app-binompseries} shows that we can write
\begin{equation}
\label{eq:man-evalexpand}
\sum_n a_n\binom{\sum_i p^i x_i}{n}=\sum_n \sum_{|\ul{j}|=n} a_n\lambda_{n, \ul{j}}\binom{\ul{x}}{\ul{j}}
\end{equation}
for integers $\lambda_{n, \ul{j}}$ satisfying
\begin{equation*}
\frac{1}{k}v_p(\lambda_{n, \ul{j}})+\epsilon n\geq \sum_i \frac{ij_i}{k}\geq \max\left\{\frac{\max\{i: j_i\neq 0\}}{k}, \frac{|\ul{j}|}{k}\right\}\;,
\end{equation*}
where we recall that $\binom{\ul{x}}{\ul{j}}\coloneqq \prod_i \binom{x_i}{j_i}$, which is in particular an integer. Ordering the multiindices $\ul{j}$ in each inner sum by the largest index $i$ such that $j_i\neq 0$, we see that the $\pi$-adic valuation of the coefficients
\begin{equation*}
a_n\cdot \frac{\lambda_{n, \ul{j}}}{p^{k\lfloor v_p(\lambda_{n, \ul{j}})/k\rfloor}}\cdot \pi^{\lfloor v_p(\lambda_{n, \ul{j}})/k\rfloor}\cdot \left(\frac{p^k}{\pi}\right)^{\lfloor v_p(\lambda_{n, \ul{j}})/k\rfloor }
\end{equation*}
in the sum on the right-hand side of (\ref{eq:man-evalexpand}) grows faster than $i/k$ since the $\pi$-adic valuation of $a_n$ grows faster than $\epsilon n$ by assumption. In particular, it goes to infinity and hence each inner sum on the right-hand side of (\ref{eq:man-evalexpand}) makes sense in $\Z(\!(\pi)\!)\langle\tfrac{p^k}{\pi}\rangle$. To see that also the outer sum makes sense, we invoke the second bound from above: in the $n$-th inner sum, the $\pi$-adic valuation of each coefficient is at least $|\ul{j}|/k$ and this goes to infinity as $|\ul{j}|=n\rightarrow\infty$, which finishes the proof.
\end{proof}

From the previous lemma, we conclude that $C^\la(\Z_p, \Z(\!(\pi)\!))$ may indeed be interpreted as a space of functions on $\Z_p$ after base changing to the locus $\{|p|<1\}$ inside $\AnSpec\Z_\solid(\!(\pi)\!)$. Next, we check that locally analytic endomorphisms of $\Z_p$ induce endomorphisms of $C^\la(\Z_p, \Z(\!(\pi)\!))$. Again, this will only be true after base change to $\{|p|<1\}$. We begin with the following preparatory lemma:

\begin{lem}
\label{lem:man-zpdla}
Let $d\geq 0$. If $C^\la(\Z_p^d, \Z(\!(\pi)\!))$ denotes the $d$-fold tensor product of $C^\la(\Z_p, \Z(\!(\pi)\!))$ over $\Z(\!(\pi)\!))$, then 
\begin{equation*}
C^\la(\Z_p^d, \Z(\!(\pi)\!))=\colim_{\epsilon>0} \Z(\!(\pi)\!)\left\langle \pi^{\epsilon |\ul{n}|}\binom{\ul{x}}{\ul{n}}: \ul{n}\geq 0\right\rangle\;,
\end{equation*}
where $\ul{n}=(n_1, \dots, n_d)$ is a multiindex, for which we set $|\ul{n}|\coloneqq n_1+\dots+n_d$, and we recall that $\binom{\ul{x}}{\ul{n}}$ denotes the product of the $\binom{x_i}{n_i}$ for $i=1, \dots, d$.
\end{lem}
\begin{proof}
It suffices to show that
\begin{equation*}
\Z[\![\pi]\!]\left\langle \pi^{\epsilon n}\binom{x}{n}: n\geq 0\right\rangle^{\tensor_{\Z[\![\pi]\!]} d}\cong \Z[\![\pi]\!]\left\langle \pi^{\epsilon |\ul{n}|}\binom{\ul{x}}{\ul{n}}: \ul{n}\geq 0\right\rangle
\end{equation*}
for each $\epsilon>0$, then the result follows by inverting $\pi$ and taking colimits over $\epsilon>0$. As the solid tensor product preserves derived $\pi$-completeness, both sides of the claimed isomorphism are derived $\pi$-complete and hence we may check it after reducing mod $\pi$. However, then both sides are discrete free $\Z$-modules of countable dimension and the claim is clear by looking at the basis elements.
\end{proof}

\comment{
\begin{cor}
\label{cor:man-zpfuncsembed}
Let $C^\la(\Z_p^d, \Z_p)$ denote the space of maps $\Z_p^d\rightarrow \Z_p$ of $p$-adic manifolds for any $d\geq 0$. For any $k\geq 1$, there is a canonical map
\begin{equation*}
C^\la(\Z_p^d, \Z_p)\rightarrow C^\la(\Z_p^d, \Z(\!(\pi)\!))\tensor_{\Z(\!(\pi)\!)} \Z(\!(\pi)\!)\langle \tfrac{p^k}{\pi}\rangle
\end{equation*}
of $\Z_p$-algebras which is compatible with the evaluation map from \cref{lem:man-eval} under the canonical map $\Z_p\rightarrow \Z(\!(\pi)\!)\langle\tfrac{p^k}{\pi}\rangle$.
\end{cor}
\begin{proof}
By \cref{rem:app-amicemult}, any map of $p$-adic manifolds $f: \Z_p^d\rightarrow\Z_p$ is given by
\begin{equation*}
f(x_1, \dots, x_d)=\sum_{\ul{n}} a_{\ul{n}}\binom{\ul{x}}{\ul{n}}
\end{equation*}
with $a_{\ul{n}}\in\Z_p$ such that $v_p(a_{\ul{n}})$ grows faster than $\epsilon |\ul{n}|$ for some $\epsilon>0$, where $\ul{n}=(n_1, \dots, n_d)$ is a multiindex, for which we set $|\ul{n}|\coloneqq n_1+\dots+n_d$. All we have to check is that this series converges in $C^\la(\Z_p^d, \Z(\!(\pi)\!))\tensor_{\Z(\!(\pi)\!)} \Z(\!(\pi)\!)\langle \tfrac{p^k}{\pi}\rangle$. However, rewriting
\begin{equation*}
a_{\ul{n}}=\frac{a_{\ul{n}}}{p^{k\lfloor v_p(a_{\ul{n}})/k\rfloor}}\cdot \pi^{\lfloor v_p(a_{\ul{n}})/k\rfloor}\cdot\left(\frac{p^k}{\pi}\right)^{\lfloor v_p(a_{\ul{n}})/k\rfloor}\;,
\end{equation*}
we see that the $\pi$-adic valuation of the coefficients grows faster than $\epsilon |\ul{n}|/k$ and this yields the desired map by \cref{lem:man-zpdla}.
\end{proof}
}

\begin{prop}
\label{prop:man-endos}
Let $f: \Z_p^d\rightarrow \Z_p^{d'}$ be a map of $p$-adic manifolds with $d, d'\geq 0$. For any $k\geq 1$, there is an induced map
\begin{equation*}
f^*: C^\la(\Z_p^{d'}, \Z(\!(\pi)\!))\tensor_{\Z(\!(\pi)\!)} \Z(\!(\pi)\!)\langle\tfrac{p^k}{\pi}\rangle\rightarrow C^\la(\Z_p^d, \Z(\!(\pi)\!))\tensor_{\Z(\!(\pi)\!)} \Z(\!(\pi)\!)\langle\tfrac{p^k}{\pi}\rangle\;,
\end{equation*}
of $\Z(\!(\pi)\!)\langle\tfrac{p^k}{\pi}\rangle$-algebras. Moreover, the association $f\mapsto f^*$ is compatible with composition and commutes with the evaluation map $\mathrm{ev}$ from \cref{lem:man-eval}.
\end{prop}
\begin{proof}
As $f$ is given by $d'$ locally analytic functions $\Z_p^d\rightarrow\Z_p$, we may reduce to the case $d'=1$ by the universal property of the tensor product. As in the previous proof, the map $f: \Z_p^d\rightarrow\Z_p$ can then be written as
\begin{equation*}
f(x_1, \dots, x_d)=\sum_{\ul{n}} a_{\ul{n}}\binom{\ul{x}}{\ul{n}}
\end{equation*}
with $a_{\ul{n}}\in\Z_p$ such that $v_p(a_{\ul{n}})$ grows faster than $\epsilon |\ul{n}|$ for some $\epsilon>0$. By a similar density argument as in the proof of \cref{lem:man-eval}, our task is to show that, given any $\sum_m b_m\binom{y}{m}\in C^\la(\Z_p, \Z(\!(\pi)\!))$, the expression
\begin{equation}
\label{eq:man-endoscomp}
\sum_m b_m\binom{\sum_{\ul{n}} a_{\ul{n}} \binom{\ul{x}}{\ul{n}}}{m}
\end{equation}
makes sense in the ring $C^\la(\Z_p^d, \Z(\!(\pi)\!))\tensor_{\Z(\!(\pi)\!)} \Z(\!(\pi)\!)\langle\tfrac{p^k}{\pi}\rangle$. To this end, we may assume that $\pi^{-\epsilon m}b_m\rightarrow 0$ in the $\pi$-adic topology after possibly lowering $\epsilon$; furthermore, as in the proof of \cref{lem:man-eval}, we may then increase $k$ to assume that $\epsilon k\geq \tfrac{1}{p-1}$.

Now first note that
\begin{equation*}
\sum_m b_m\binom{\sum_{\ul{n}} a_{\ul{n}} \binom{\ul{x}}{\ul{n}}}{m}=\sum_m b_m\sum_{\sum_{\ul{n}} m_{\ul{n}}=m} \prod_{\ul{n}} \binom{a_{\ul{n}}\binom{\ul{x}}{\ul{n}}}{m_{\ul{n}}}
\end{equation*}
by double counting. Writing $a_{\ul{n}}=p^{e_{\ul{n}}} a'_{\ul{n}}$ with $a'_{\ul{n}}\in\Z_p^\times$, the proof of \cref{lem:app-binompseries} shows that
\begin{equation*}
\binom{a_{\ul{n}}\binom{\ul{x}}{\ul{n}}}{m_{\ul{n}}}=\sum_j \lambda_{\ul{n}, m_{\ul{n}}, j}\binom{a_{\ul{n}}'\binom{\ul{x}}{\ul{n}}}{j}
\end{equation*}
for integers $\lambda_{\ul{n}, m_{\ul{n}}, j}$ satisfying
\begin{equation*}
\frac{1}{k}v_p(\lambda_{\ul{n}, m_{\ul{n}}, j})+\epsilon m_{\ul{n}}\geq \frac{e_{\ul{n}} j}{k}\;.
\end{equation*}
As the $\pi$-adic valuation of $b_m$ grows faster than $\epsilon m$ and $e_{\ul{n}}$ grows faster than $\epsilon |\ul{n}|$, we conclude that we may rewrite (\ref{eq:man-endoscomp}) as
\begin{equation*}
\sum_{\substack{\ul{j}=(j_{\ul{n}})_{\ul{n}} \\ \text{almost all $j_{\ul{n}}=0$}}} c_{\ul{j}}\prod_{\ul{n}} \binom{a_{\ul{n}}'\binom{\ul{x}}{\ul{n}}}{j_{\ul{n}}}
\end{equation*}
with elements $c_{\ul{j}}\in \Z(\!(\pi)\!)\langle\tfrac{p^k}{\pi}\rangle$ whose $\pi$-adic valuation grows faster than $\epsilon\sum_{\ul{n}} j_{\ul{n}} |\ul{n}|/k$.

To move on, we observe that $\prod_{\ul{n}}\binom{a_{\ul{n}}'\binom{\ul{x}}{\ul{n}}}{j_{\ul{n}}}$ is a $\Z_p$-valued polynomial over $\Q_p$ in $d$ variables and hence can be written as
\begin{equation*}
\prod_{\ul{n}}\binom{a_{\ul{n}}'\binom{\ul{x}}{\ul{n}}}{j_{\ul{n}}}=\sum_{|\ul{i}|\leq \sum_{\ul{n}} j_{\ul{n}}|\ul{n}|} \mu_{\ul{i}}\binom{\ul{x}}{\ul{i}}
\end{equation*}
for some $\mu_{\ul{i}}\in\Z_p$ by \cref{rem:app-intbasiszpmult}. Using this, we conclude that (\ref{eq:man-endoscomp}) can be written as
\begin{equation*}
\sum_{\ul{i}} c'_{\ul{i}}\binom{\ul{x}}{\ul{i}}
\end{equation*}
for elements $c'_{\ul{i}}\in\Z(\!(\pi)\!)\langle\tfrac{p^k}{\pi}\rangle$ whose $\pi$-adic valuation grows faster than $\epsilon |\ul{i}|/k$ and this now yields the claim using the description of $C^\la(\Z_p^d, \Z(\!(\pi)\!))$ from \cref{lem:man-zpdla}.
\end{proof}

We are now almost in position to construct our desired functor from $p$-adic manifolds to analytic stacks over $\cal{N}_{|p|<1}/\R_{>0}$. However, we first have to observe the following descent property of our constructions so far:

\begin{lem}
\label{lem:man-descent}
The analytic stack $\AnSpec C^\la(\Z_p, \Z(\!(\pi)\!))$ is equipped with a canonical descent datum along the cover
\begin{equation*}
\AnSpec\Z_\solid(\!(\pi)\!)\rightarrow \cal{N}/\R_{>0}
\end{equation*}
from \cref{cor:norms-pres}.
\end{lem}
\begin{proof}
We first show descent to $\cal{N}$, for which we have to give a canonical isomorphism
\begin{equation*}
C^\la(\Z_p, \Z(\!(\pi)\!))\tensor_{\Z(\!(\pi)\!), \pi\mapsto \pi} \Z(\!(\pi)\!)\langle T^{\pm 1}\rangle_{\leq 1}\cong C^\la(\Z_p, \Z(\!(\pi)\!))\tensor_{\Z(\!(\pi)\!), \pi\mapsto \pi T} \Z(\!(\pi)\!)\langle T^{\pm 1}\rangle_{\leq 1}\;,
\end{equation*}
which after unravelling definitions comes down to showing that
\begin{equation*}
\colim_{\epsilon>0} \Z(\!(\pi)\!)\left\langle\pi^{\epsilon} T^{\pm 1}, \pi^{\epsilon n}\binom{x}{n}: n\geq 0\right\rangle\cong \colim_{\epsilon>0} \Z(\!(\pi)\!)\left\langle\pi^{\epsilon} T^{\pm 1}, (\pi T)^{\epsilon n}\binom{x}{n}: n\geq 0\right\rangle\;.
\end{equation*}
To establish this, we will show that the filtered systems on both sides are cofinal. Namely, given
\begin{equation*}
\sum_{m, n} a_{mn} T^m\binom{x}{n}\in \Z(\!(\pi)\!)\left\langle\pi^{\epsilon} T^{\pm 1}, \pi^{\epsilon n}\binom{x}{n}: n\geq 0\right\rangle\;,
\end{equation*}
we have $v_\pi(a_{mn})-\epsilon(|m|+n)\rightarrow\infty$ as $|m|+n\rightarrow\infty$ and may assume $\epsilon<1$ without loss of generality. Then $\epsilon'=\epsilon/2$ satisfies 
\begin{equation*}
\begin{split}
v_\pi(a_{mn})-\epsilon'(|m-\lfloor n\epsilon\rfloor|+n)&\geq v_\pi(a_{mn})-\epsilon'(|m|+\lfloor n\epsilon\rfloor +n) \\
&\geq v_\pi(a_{mn})-\epsilon(|m|/2+(\epsilon+1)n/2) \\
&\geq v_\pi(a_{mn})-\epsilon(|m|+n)\rightarrow\infty
\end{split}
\end{equation*}
and hence
\begin{equation*}
\sum_{m, n} a_{mn} T^m\binom{x}{n}=\sum_{m, n} a_{mn} T^{m-\lfloor n\epsilon\rfloor} T^{\lfloor n\epsilon\rfloor}\binom{x}{n}\in \Z(\!(\pi)\!)\left\langle\pi^{\epsilon} T^{\pm 1}, (\pi T)^{\epsilon n}\binom{x}{n}: n\geq 0\right\rangle\;,
\end{equation*}
as desired. The argument for the other direction is similar. Finally, for the further descent along $\cal{N}\rightarrow\cal{N}/\R_{>0}$, note that the action of $\lambda\in \R_{>0}$ on
\begin{equation*}
C^\la(\Z_p, \Z(\!(\pi)\!))=\colim_{\epsilon>0} \Z(\!(\pi)\!)\left\langle\pi^{\epsilon n} \binom{x}{n}: n\geq 0\right\rangle
\end{equation*}
replaces $\epsilon$ by $\lambda\epsilon$, but this does not change the colimit.
\end{proof}

By the previous lemma, the analytic stack $\AnSpec C^\la(\Z_p, \Z(\!(\pi)\!))$ canonically descends to $\cal{N}/\R_{>0}$ and we let $\Z_p^\la$ denote the base change of this descent along $\cal{N}_{|p|<1}/\R_{>0}\rightarrow\cal{N}/\R_{>0}$. Then observe that also the result of \cref{prop:man-endos} descends to $\cal{N}_{|p|<1}/\R_{>0}$, i.e.\ any map $\Z_p^d\rightarrow\Z_p^{d'}$ of $p$-adic manifolds functorially induces a morphism
\begin{equation*}
(\Z_p^\la)^d\rightarrow (\Z_p^\la)^{d'}
\end{equation*}
over $\cal{N}_{|p|<1}/\R_{>0}$. 

\begin{thm}
\label{thm:man-lafunctor}
The association $\Z_p^d\mapsto (\Z_p^\la)^d$ extends to a functor
\begin{equation*}
\begin{split}
\{\text{paracompact $p$-adic manifolds}\}&\rightarrow\{\text{analytic stacks over $\cal{N}_{|p|<1}/\R_{>0}$}\} \\
M&\mapsto M^\la\;.
\end{split}
\end{equation*}
\end{thm}
\begin{proof}
Given a paracompact $p$-adic manifold $M$, we can write it as a disjoint union of open balls by \cite[Prop.\ 8.7]{Schneider} using the fact that any open subset of $\Q_p^d$ is a disjoint union of open balls $M=\bigsqcup_{i\in I} B_i$. Noting that $B_i\cong \Z_p^{d_i}$ for each $i\in I$ and some $d_i\geq 0$, we set
\begin{equation*}
M^\la\coloneqq \bigsqcup_{i\in I} \,(\Z_p^\la)^{d_i}\;.
\end{equation*}
If we can show that this association is well-defined, i.e.\ independent of the chosen cover, then \cref{prop:man-endos} guarantees that it is functorial. 

For well-definedness, we take two covers of $M$ by open balls and can then produce a common refinement which is a cover by disjoint open balls as well. Then we are reduced to proving the following statement: Assume that
\begin{equation*}
\bigsqcup_{i\in I} \,\Z_p^d\xrightarrow{\cong} \Z_p^d
\end{equation*}
is a cover of $\Z_p^d$ by disjoint open balls. Then the induced map
\begin{equation*}
\bigsqcup_{i\in I} \,(\Z_p^d)^\la\xrightarrow{\cong} (\Z_p^d)^\la
\end{equation*}
is an isomorphism. By compactness we know that $I$ is finite and then we may reduce to the case where all balls have the same radius. From here, we may inductively reduce to the case of the cover
\begin{equation*}
\bigsqcup_{a\in [p]^d} (a+p\Z_p^d)\xrightarrow{\cong} \Z_p^d\;,
\end{equation*}
where $[p]=\{0, 1, \dots, p-1\}$. As disjoint unions commute with base change, we may furthermore reduce to the case $d=1$. Then we have to prove that
\begin{equation*}
\bigsqcup_{a\in [p]} (a+p\Z_p)^\la\rightarrow \Z_p^\la
\end{equation*}
is an isomorphism.

To check this, we may pull back to the cover $\colim_k \AnSpec\Z_\solid(\!(\pi)\!)\langle\tfrac{p^k}{\pi}\rangle$ of $\cal{N}_{|p|<1}/\R_{>0}$. Then we have to show that the map
\begin{equation}
\label{eq:man-decomp}
\begin{split}
C^\la(\Z_p, \Z(\!(\pi)\!))&\rightarrow \prod_{a\in [p]} C^\la(a+p\Z_p, \Z(\!(\pi)\!)) \\
\binom{x}{n}&\mapsto \left(\binom{a+px}{n}\right)_{a\in [p]}
\end{split}
\end{equation}
is an isomorphism after base changing to $\Z(\!(\pi)\!)\langle\tfrac{p^k}{\pi}\rangle$ for any $k\geq 1$. It suffices to give an inverse and we claim that this is provided by summing the maps
\begin{equation}
\label{eq:man-decompinverse}
\begin{split}
C^\la(a+p\Z_p, \Z(\!(\pi)\!))&\rightarrow C^\la(\Z_p, \Z(\!(\pi)\!))\tensor_{\Z(\!(\pi)\!)} \Z(\!(\pi)\!)\langle\tfrac{p^k}{\pi}\rangle \\
\binom{x}{n}&\mapsto \sum_m\left(\sum_{a+p\ell\leq m} (-1)^{m-a-p\ell}\binom{m}{a+p\ell}\binom{\ell}{n}\right)\binom{x}{m}
\end{split}
\end{equation}
for $a\in [p]$, where we note that the right-hand side is the Mahler expansion of the $\Z_p$-valued function on $\Z_p$ given by $\binom{(x-a)/p}{n}$ if $x\in a+p\Z_p$ and zero otherwise. We point out that we do \emph{not} need to show that this map is multiplicative: Once we know that it is inverse to (\ref{eq:man-decomp}), this will follow from the fact that (\ref{eq:man-decomp}) is a ring map.

We first claim that (\ref{eq:man-decompinverse}) induces a well-defined $\Z(\!(\pi)\!)$-linear map, for which we have to check that
\begin{equation}
\label{eq:man-decompinverseconverges}
\sum_n c_n\sum_m\left(\sum_{a+p\ell\leq m} (-1)^{m-a-p\ell}\binom{m}{a+p\ell}\binom{\ell}{n}\right)\binom{x}{m}
\end{equation}
makes sense in $C^\la(\Z_p, \Z(\!(\pi)\!))\tensor_{\Z(\!(\pi)\!)} \Z(\!(\pi)\!)\langle\tfrac{p^k}{\pi}\rangle$ whenever the $\pi$-adic valuations of $c_n\in\Z(\!(\pi)\!)$ grow faster than $\epsilon n$ for some $\epsilon>0$; without loss of generality, we may assume that $\epsilon<1/k$. In other words, we have to show that the $\pi$-adic valuation of
\begin{equation*}
\sum_n c_n\sum_{a+p\ell\leq m} (-1)^{m-a-p\ell}\binom{m}{a+p\ell}\binom{\ell}{n}
\end{equation*}
grows at least linearly in $m$. Denoting the inner sum by $b_{mn}$, we notice that $b_{mn}=0$ for $n>(m-a)/p$ and hence the above simplifies to
\begin{equation*}
\sum_{n\leq (m-a)/p} c_n b_{mn}\;.
\end{equation*}

We now estimate the $\pi$-adic valuation of each summand individually. By assumption, we have $v_\pi(c_n)\geq \epsilon n-C$ for some constant $C\gg 0$. Moreover, by \cref{lem:app-valuationmahlercoeffs}, we have 
\begin{equation*}
v_p(b_{mn})\geq \frac{m-pn}{p-1}-1\;.
\end{equation*}
Since we have base changed to $\Z(\!(\pi)\!)\langle\tfrac{p^k}{\pi}\rangle$, we may replace factors $p^k$ by $\pi$ up to a powerbounded element and thus we conclude from the above that
\begin{equation*}
v_\pi(b_{mn})\geq \frac{m-pn}{k(p-1)}-1\;.
\end{equation*}
Overall, this implies
\begin{equation*}
\begin{split}
v_\pi(c_nb_{mn})&\geq \epsilon n+\frac{m-pn}{k(p-1)}-C-1=\frac{m}{k(p-1)}+n\left(\epsilon-\frac{p}{k(p-1)}\right)-C-1 \\
&\geq \frac{m}{k(p-1)}+\frac{m-a}{p}\left(\epsilon-\frac{p}{k(p-1)}\right)-C-1\geq \frac{m\epsilon}{p}-C'
\end{split}
\end{equation*}
for some $C'\gg 0$, where the first estimate in the second line uses that $\epsilon<1/k<p/(k(p-1))$ and $n\leq (m-a)/p$ by assumption. This yields the claimed linear growth of the $\pi$-adic valuation of $\sum_n c_nb_{mn}$ in $m$ and thus the expression (\ref{eq:man-decompinverseconverges}) makes sense in $C^\la(\Z_p, \Z(\!(\pi)\!))\tensor_{\Z(\!(\pi)\!)} \Z(\!(\pi)\!)\langle\tfrac{p^k}{\pi}\rangle$, as desired.

It remains to show that the maps (\ref{eq:man-decomp}) and (\ref{eq:man-decompinverse}) are pairwise inverse. As both the source and target are $\pi$-adically separated, it suffices to check this on a dense subset and thus linearity reduces us to checking the identities
\begin{equation*}
\sum_{m\geq 0}\left(\sum_{a\in [p]} \sum_{a+p\ell\leq m} (-1)^{m-a-p\ell}\binom{m}{a+p\ell}\binom{a+p\ell}{n}\right)\binom{x}{m}=\binom{x}{n}
\end{equation*}
for $n\geq 0$ and
\begin{equation*}
\sum_{m\geq 0}\left(\sum_{a+p\ell\leq m} (-1)^{m-a-p\ell}\binom{m}{a+p\ell}\binom{\ell}{n}\right)\binom{a+px}{m}=\binom{x}{n}
\end{equation*}
for $a\in [p]$ and $n\geq 0$ in the ring of formal power series over $\Q_p$. For the first identity, we see that the inner double sum on the left-hand side collapses to
\begin{equation*}
\sum_{\ell\leq m} (-1)^{m-\ell}\binom{m}{\ell}\binom{\ell}{n}\;,
\end{equation*}
which is the $m$-th Mahler coefficient of the function $\binom{x}{n}$ on $\Z_p$ and this yields the claim. The second identity is proved similarly: The inner sum is the $m$-th Mahler coefficient of the function on $\Z_p$ given by $\binom{(y-a)/p}{n}$ if $y\in a+p\Z_p$ and zero otherwise and evaluating this for $y=a+px$ yields $\binom{x}{n}$, as desired.
\end{proof}

\begin{ex}
For $M=\Q_p$, we have
\begin{equation*}
\Q_p^\la\cong \colim_n\,(p^{-n}\Z_p)^\la\;.
\end{equation*}
Indeed, following the proof above, we have
\begin{equation*}
\Q_p^\la\cong \bigsqcup_{a\in \Q_p/\Z_p} (a+\Z_p)^\la\cong \colim_n \bigsqcup_{a\in p^{-n}\Z_p/\Z_p} (a+\Z_p)^\la\cong \colim_n \,(p^{-n}\Z_p)^\la\;.\qedhere
\end{equation*}
\end{ex}

\begin{ex}
Assume that $M=G$ is a $p$-adic Lie group. Then the paracompactness is automatic and, by \cite[Thm.\ 8.32]{ProPGroups}, there is an open subgroup $G_0\subseteq G$ such that $G_0\cong \Z_p^d$ as $p$-adic manifolds. This yields
\begin{equation*}
G^\la\cong \bigsqcup_{g\in G/G_0} (gG_0)^\la\;. \qedhere
\end{equation*}
\end{ex}

Our next goal is to prove that the functor $(-)^\la$ preserves fibre products of transverse maps. For this, recall that maps $f: M'\rightarrow M$ and $g: M''\rightarrow M$ of $p$-adic manifolds are called \emph{transverse} if 
\begin{equation*}
\Im T_x(f) + \Im T_y(g) = T_zM
\end{equation*}
for all points $x\in M', y\in M''$ with $f(x)=g(y)\eqqcolon z$. Here, $T_zM$ denotes the tangent space of $M$ at $z$ and $T_x(f): T_xM'\rightarrow T_zM$ is the map on tangent spaces induced by $f$ at $x$. We first establish the following $p$-adic version of the ``transverse intersection theorem'' for real manifolds:

\begin{lem}
\label{lem:man-transverseintersect}
Let $M$ be a $p$-adic manifold and assume that $M', M''\subseteq M$ are submanifolds such that the inclusion maps are transverse. If $x\in M'\cap M''$ is any point in the intersection of the two submanifolds, then there is an open neighbourhood $x\in U\subseteq M$ with coordinates $x_1, \dots, x_d$ such that 
\begin{equation*}
M'\cap U=V(x_1, \dots, x_r)\subseteq U\;, \hspace{0.5cm} M''\cap U=V(x_{r+1}, \dots, x_{r+s})\subseteq U
\end{equation*}
for some $r, s\geq 0$ with $r+s\leq d$. In particular, the intersection $M'\cap M''$ is a $p$-adic manifold.
\end{lem}
\begin{proof}
By definition, there is some neighbourhood $x\in U\subseteq M$ with coordinates $x_1, \dots, x_d$ and $\Q_p$-valued analytic functions $f_1, \dots, f_r$ whose derivatives $\mathrm{d}_xf_1, \dots, \mathrm{d}_xf_r$ at $x$ are linearly independent in $(T_xM)^\vee$ such that $M'\cap U$ is cut out by $f_1=\dots=f_r=0$; similarly, we find functions $f_{r+1}, \dots, f_{r+s}$ with linearly independent derivatives at $x$ such that $M''\cap U$ is the common vanishing locus of $f_{r+1}, \dots, f_{r+s}$. Now note that $\mathrm{d}_xf_1, \dots, \mathrm{d}_x f_r$ form a basis of the conormal space $N_x(M'/M)^\vee=\Ker((T_xM)^\vee\rightarrow (T_xM')^\vee)$ and, similarly, $\mathrm{d}_xf_{r+1}, \dots, \mathrm{d}_xf_{r+s}$ form a basis of $N_x(M''/M)^\vee$. As $T_xM'+T_xM''=T_xM$ by assumption, the conormal spaces $N_x(M'/M)^\vee$ and $N_x(M''/M)^\vee$ intersect trivially, hence $\mathrm{d}_xf_1, \dots, \mathrm{d}_xf_{r+s}$ are jointly linearly independent. 

In particular, we have $r+s\leq d$ and, using the derivatives of the coordinate functions $x_i$, we can extend $\mathrm{d}_xf_1, \dots, \mathrm{d}_xf_{r+s}$ to a basis $\mathrm{d}_xf_1, \dots, \mathrm{d}_xf_{r+s}, \mathrm{d}_xg_{r+s+1}, \dots, \mathrm{d}_xg_d$ of $T_xM$. By the inverse function theorem in the form from \cite[Prop.\ 9.3]{Schneider}, the functions $f_1, \dots, f_{r+s}, g_{r+s+1}, \dots, g_d$ then induce an isomorphism from $U$ onto a small open subset of $\Q_p^d$ after possibly shrinking $U$. In these new coordinates $x_1', \dots, x_d'$, we will then have $M'\cap U=V(x_1', \dots, x_r')$ and $M''\cap U=V(x_{r+1}', \dots, x_{r+s}')$ by construction, as desired. For the second part, we just note that $M'\cap M''\cap U=V(x_1, \dots, x_{r+s})$ is a $p$-adic manifold. This concludes the proof.
\end{proof}

\begin{prop}
\label{prop:man-transverse}
Let $M_1\rightarrow M\leftarrow M_2$ be transverse maps of paracompact $p$-adic manifolds. Then the fibre product $M_1\times_M M_2$ has the structure of a $p$-adic manifold and the diagram
\begin{equation*}
\begin{tikzcd}
(M_1\times_M M_2)^\la\ar[r]\ar[d] & M_2^\la\ar[d] \\
M_1^\la\ar[r] & M^\la
\end{tikzcd}
\end{equation*}
of analytic stacks over $\cal{N}_{|p|<1}/\R_{>0}$ is cartesian. In other words, the functor $(-)^\la$ commutes with transverse fibre products.
\end{prop}
\begin{proof}
Let $f_1: M_1\rightarrow M$ and $f_2: M_2\rightarrow M$ be the given maps. Then first note that $M_1\times_M M_2$ identifies with the intersection of the two submanifolds
\begin{equation*}
\begin{split}
M_1\times M_2&\hookrightarrow M_1\times M_2\times M\;, \hspace{0.3cm} (x, y)\mapsto (x, y, f_1(x)) \\
M_1\times M_2&\hookrightarrow M_1\times M_2\times M\;, \hspace{0.3cm} (x, y)\mapsto (x, y, f_2(y))
\end{split}
\end{equation*}
of $M_1\times M_2\times M$. By assumption, these are transverse submanifolds and hence their intersection is a $p$-adic manifold by \cref{lem:man-transverseintersect}. As the fibre product $M_1^\la\times_{M^\la} M_2^\la$ can be described similarly and the functor $M\mapsto M^\la$ commutes with finite products by construction, we may thus reduce to the case where $M_1$ and $M_2$ are submanifolds of $M$. Working locally on $M$, we can furthermore assume that $M\cong \Z_p^d$ and $M_1=V(x_1, \dots, x_r)$ while $M_2=V(x_{r+1}, \dots, x_{r+s})$ by \cref{lem:man-transverseintersect}, where $x_1, \dots, x_d$ are the coordinates on $\Z_p^d$. By induction, we can furthermore reduce to the case $r=1$ and compatibility with finite products then furthermore reduces us to $r+s=d$. In this case, we have $M\cong M_1\times M_2$ and then the claim follows from compatibility with finite products, so we are done.
\end{proof}

Finally, we want to explain that, in the case where the $p$-adic manifold $M$ arises as the $\Z_p$-points of a smooth $\Z_p$-scheme $X$, we can also view $M^\la$ in terms of transmutation. Namely, note that $\Z_p^\la$ is a $\Z_p$-algebra stack by functoriality of $(-)^\la$ and hence, for each analytic ring $R$ over $\cal{N}_{|p|<1}/\R_{>0}$, the $R$-points of $\Z_p^\la$ form an animated $\Z_p$-algebra. In particular, we can consider the $\Z_p^\la(R)$-valued points $X(\Z_p^\la(R))$ of the $\Z_p$-scheme $X$.

\begin{prop}
\label{prop:man-transmutation}
Let $X$ be a smooth $\Z_p$-scheme. Then $X(\Z_p)^\la$ is the sheafification of the presheaf over $\cal{N}_{|p|<1}/\R_{>0}$ given by
\begin{equation*}
R\mapsto X(\Z_p^\la(R))\;,
\end{equation*}
where the mapping space on the right is computed over $\Z_p$. Moreover, the same is true if we replace $\Z_p$ by $\Q_p$ everywhere and additionally assume that $X$ is separated and of finite type.
\end{prop}
\begin{proof}
We only do the $\Z_p$-case, the $\Q_p$-case is analogous; the separatedness and finite type assumptions are only in place to ensure that $X(\Q_p)$ is a paracompact $p$-adic manifold. Clearly, the assignment $X\mapsto X(\Z_p)^\la$ satisfies Zariski descent and we first argue that the same is true for the assignment sending $X$ to the sheafification of $R\mapsto X(\Z_p^\la(R))$, which we momentarily call $X(\Z_p^\la)$. As the latter commutes with fibre products in $X$, we only need to show that $\bigsqcup_i U_i(\Z_p^\la)\rightarrow X(\Z_p^\la)$ is surjective for any Zariski cover $X=\bigcup_i U_i$. In fact, we are going to show

\bigskip

\textbf{Claim.} For any map $\Spec \Z_p^\la(R)\rightarrow X$, there is a finite decomposition $R\cong\prod_j R_j$ such that each induced map $\Spec\Z_p^\la(R_j)\rightarrow \Spec\Z_p^\la(R)\rightarrow X$ factors through a single $U_i$.

\bigskip

\textit{Proof of the claim.} Pulling back the given open cover of $X$ to an open cover of $\Spec\Z_p^\la(R)$, we may assume that $X=\Spec\Z_p^\la(R)$ and that the map $\Spec\Z_p^\la(R)\rightarrow X$ is the identity. Then $X$ is affine and we may thus assume that the given open cover is by distinguished opens $D(f_i)$.

By \cref{lem:smoothfp-zpladagger} below, the map $\Z_p^\la\rightarrow\Z_p^{\Betti}$ is surjective with kernel contained in $p\Z_p^\la\subseteq\Z_p^\la$, which is in turn contained in the Jacobson radical of the ring stack $\Z_p^\la$: Indeed, the inversion map $1+p\Z_p\rightarrow \Z_p$ is locally analytic, hence induces a map $1+p\Z_p^\la\rightarrow \Z_p^\la$ by \cref{prop:man-endos}. Thus, checking whether an element is a unit in $\Z_p^\la(R)$ may be done after mapping to $\Z_p^{\Betti}(R)$. In other words, it suffices to prove the claim for the ring stack $\Z_p^{\Betti}$ in place of $\Z_p^\la$.

Now note that $\Z_p^{\Betti}(R)\cong \Cont(\pi_0 \Spec R, \Z_p)$ with the ring structure induced by the one on $\Z_p$ and where $S\coloneqq \pi_0\Spec R$ is a profinite set. As the $f_i$ generate the unit ideal and $\Z_p$ is local, the open subsets $\{s\in S: f_i(s)\in \Z_p^\times\}$ cover $S$ and we may choose a refinement $S=\bigsqcup_j S_j$ of this cover by finitely many pairwise disjoint clopen subsets of $S$. As $S=\pi_0 \Spec R$, the idempotents of $R$ corresponding to the $S_j$ induce the desired product decomposition $R=\prod_j R_j$ of $R$. \hfill\qedhere

\bigskip

We conclude that both $X(\Z_p)^\la$ and $X(\Z_p^\la)$ satisfy Zariski descent in $X$. Thus, working Zariski locally on $X$ now, we may assume that $X$ is standard smooth cut out by the equations $f_1=\dots=f_m=0$ in $\A^n_{\Z_p}$; in particular, the diagram
\begin{equation*}
\begin{tikzcd}
X\ar[r]\ar[d] & \A^n_{\Z_p}\ar[d, "{(f_1, \dots, f_m)}"] \\
\{0\}\ar[r] & \A^m_{\Z_p}
\end{tikzcd}
\end{equation*}
of derived $\Z_p$-schemes is cartesian. As the claim is clearly true for $X=\A^n_{\Z_p}$, it suffices to check that the diagram
\begin{equation*}
\begin{tikzcd}
X(\Z_p)^\la\ar[r]\ar[d] & (\Z_p^n)^\la\ar[d, "{(f_1, \dots, f_m)}"] \\
\{0\}\ar[r] & (\Z_p^m)^\la
\end{tikzcd}
\end{equation*}
is cartesian as well. However, by construction, the Jacobian matrix of $f_1, \dots, f_m$ has full rank everywhere on $X$, i.e.\ the map $(f_1, \dots, f_m): \Z_p^n\rightarrow \Z_p^m$ of $p$-adic manifolds induces a surjection onto the tangent space of $\Z_p^m$ at zero. Thus, the $p$-adic manifold $X(\Z_p)$ is the transverse fibre product of $\Z_p^n$ and $\{0\}$ over $\Z_p^m$, whence the claim follows from \cref{prop:man-transverse}.
\end{proof}

\begin{rem}
In fact, using similar arguments, one can prove that the sheafification is not necessary, i.e.\ the assignment $R\mapsto X(\Z_p^\la(R))$ is already a sheaf.
\end{rem}

\begin{ex}
Let $\mathbb{G}$ be a smooth algebraic group over $\Q_p$. Then $\mathbb{G}(\Q_p)^\la$ is the transmutation of $\mathbb{G}$ with respect to the $\Q_p$-algebra stack $\Q_p^\la$ over $\cal{N}_{|p|<1}/\R_{>0}$. 
\end{ex}

\section{Complements to \cite{Porat2}}
\label{sect:porat}

Before we turn to discussing the analytic stacks over $\cal{N}_{|p|<1}$ associated to $p$-adic Lie groups and their classifying stacks, we collect some complements to the theory of locally analytic representations in mixed characteristic from \cite{Porat2} and, in particular, introduce the notion of \emph{$h\dagger$-analytic vectors} in mixed characteristic. Let us emphasise that these are all direct integer-coefficients analogues of results and definitions from \cite{SolidLocAnReps} and \cite{Mikami} and our proofs do not require any substantially new inputs.

We begin by quickly recalling the main definitions from \cite{Porat2}. For this, fix a $p$-adic Lie group $G$ and a complete Tate Huber pair $(B, B^+)$ over $(\Z(\!(\pi)\!), \Z[\![\pi]\!])$ of slope $\geq 1$ in the sense of \cite[Def.\ 3.4]{Porat2}, i.e.\ the map $\Z_p[\![\pi]\!][X]\rightarrow B^+$ given by $X\mapsto p$ extends to $\colim_k \Z_p[\![\pi]\!]\langle \pi^{-k-1}X^k\rangle$. We also assume that $(B, B^+)$ is residually of finite type in the sense of \cite[Def.\ 3.6]{Porat2}, i.e.\ the quotient $B^+/\pi$ is a finitely generated $\Z$-algebra. Recall that any $p$-adic Lie group $G$ admits an open subgroup $G_0\subseteq G$ together with elements $g_1, \dots, g_d\in G_0$ such that the map 
\begin{equation*}
\Z_p^d\rightarrow G_0\;, \hspace{0.3cm} (x_1, \dots, x_d)\mapsto g_1^{x_1}\cdots g_d^{x_d}
\end{equation*}
is an isomorphism of $p$-adic manifolds and we fix such a choice once and for all.

\begin{defi}[{\cite[Def.\ 4.8]{Porat2}}]
For $h\in\R$,\footnote{In \cite{Porat2}, it is always assumed that $h\geq 0$, but in fact everything works for general $h\in\R$.} we define
\begin{equation*}
C^{h\text{-}\an}(G_0, B^+)\coloneqq B^+\left\langle\pi^{v^h(\ul{n})}\binom{\ul{x}}{\ul{n}}: \ul{n}\geq 0\right\rangle\;,
\end{equation*}
where we recall that $\binom{\ul{x}}{\ul{n}}\coloneqq \prod_{i=1}^d \binom{x_i}{n_i}$ and $v^h(\ul{n})\coloneqq \sum_{i=1}^d v_p(\lfloor n_i/p^h\rfloor!)$. Similarly, we set
\begin{equation*}
C_{h\text{-}\an}(G_0, B^+)\coloneqq B^+\left\langle\pi^{v_h(\ul{n})}\binom{\ul{x}}{\ul{n}}: \ul{n}\geq 0\right\rangle
\end{equation*}
for $v_h(\ul{n})\coloneqq \lfloor |\ul{n}|/p^h(p-1)\rfloor$, where $|\ul{n}|\coloneqq \sum_{i=1}^d n_i$. We let $C^{h\text{-}\an}(G_0, B)$ and $C_{h\text{-}\an}(G_0, B)$ denote the base changes of these $B^+$-modules to $B$ and call
\begin{equation*}
C^{h\text{-}\an}(G, B)\coloneqq \prod_{g\in G/G_0} C^{h\text{-}\an}(gG_0, B)\;, \hspace{0.3cm} C_{h\text{-}\an}(G, B)\coloneqq \prod_{g\in G/G_0} C_{h\text{-}\an}(gG_0, B)
\end{equation*}
the \emph{upper $h$-analytic} and \emph{lower $h$-analytic functions} on $G$, respectively. Finally, the \emph{locally analytic functions} on $G$ are defined by
\begin{equation*}
C^\la(G, B)\coloneqq \prod_{g\in G/G_0} \colim_h C^{h\text{-}\an}(gG_0, B)=\prod_{g\in G/G_0} \colim_h C_{h\text{-}\an}(gG_0, B)\;.
\end{equation*}
\end{defi}

\begin{rem}
Note that $C^{h\text{-}\an}(G, B)$ and $C_{h\text{-}\an}(G, B)$ depend on the choice of $G_0$ even though this is not apparent from the notation. In contrast, by \cite[Cor.\ 4.31]{Porat2}, the algebra $C^\la(G, B)$ of locally analytic functions on $G$ is independent of the choice of $G_0$.
\end{rem}

It follows from \cite[Lem.\ 4.6]{Porat2} that $C^{h\text{-}\an}(G, B)$ and hence $C^\la(G, B)$ are $(B, B^+)_\solid$-algebras. Moreover, each $C^{h\text{-}\an}(G_0, B)$ is a $B$-Banach space in the sense of \cite[Def.\ 3.12]{Porat2} and, in particular, reflexive over $B$, i.e.\ isomorphic to its own double dual, by \cite[Thm.\ 3.15]{Porat2}. Let us also introduce the following variant, which is the mixed-characteristic analogue of \cite[Def.\ 2.2.(2)]{Mikami}:

\begin{defi}
For any $h\in\R$, the \emph{$h\dagger$-analytic functions} on $G$ are defined by
\begin{equation*}
C^{h\dagger\text{-}\an}(G, B)\coloneqq \prod_{g\in G/G_0} \colim_{h'<h} C^{h'\text{-}\an}(gG_0, B)=\prod_{g\in G/G_0} \colim_{h'<h} C_{h'\text{-}\an}(gG_0, B)\;,
\end{equation*}
where the equality is due to \cite[Lem.\ 4.4.2]{Porat2}.
\end{defi}

\begin{rem}
Note that formally setting $h=\infty$ recovers the definition of locally analytic functions.
\end{rem}

By passing to duals, we obtain spaces of distributions on $G$.

\begin{defi}[{\cite[Def.\ 4.8]{Porat2}}]
For any $h\in\R$, we define
\begin{equation*}
\begin{split}
D^{h\text{-}\an}(G_0, B)&\coloneqq \sHom_{(B, B^+)_\solid}(C^{h\text{-}\an}(G_0, B), B)\;, \\
D_{h\text{-}\an}(G_0, B)&\coloneqq \sHom_{(B, B^+)_\solid}(C_{h\text{-}\an}(G_0, B), B)\;.
\end{split}
\end{equation*}
The spaces of \emph{upper $h$-analytic} and \emph{lower $h$-analytic distributions} on $G$ are defined by
\begin{equation*}
\begin{split}
D^{h\text{-}\an}(G, B)&\coloneqq (B, B^+)_\solid[G]\tensor_{(B, B^+)_\solid[G_0]} D^{h\text{-}\an}(G_0, B) \\
D_{h\text{-}\an}(G, B)&\coloneqq (B, B^+)_\solid[G]\tensor_{(B, B^+)_\solid[G_0]} D_{h\text{-}\an}(G_0, B)\;,
\end{split}
\end{equation*}
respectively.
\end{defi}

Note that $D_{h\text{-}\an}(G, B)$ is a $(B, B^+)_\solid$-algebra, i.e.\ it has a well-defined multiplicative structure, by \cite[Prop.\ 4.10]{Porat2}. We can now make the following crucial definition:

\begin{defi}[{\cite[Def.s 5.3, 5.8, 5.9]{Porat2}}]
\label{defi:porat-lagcompact}
Assume that $G$ is compact and fix $h\in\R$. For any $V\in\D((B, B^+)_\solid[G])$, its (derived) \emph{upper $h$-analytic} and (derived) \emph{lower $h$-analytic vectors} are defined by
\begin{equation*}
V^{h\text{-}\an}\coloneqq \sHom_{(B, B^+)_\solid[G]}(D^{h\text{-}\an}(G, B), V)\;, \hspace{0.3cm} V_{h\text{-}\an}\coloneqq \sHom_{(B, B^+)_\solid[G]}(D_{h\text{-}\an}(G, B), V)\;,
\end{equation*}
respectively. We define the (derived) \emph{locally analytic vectors} of $V$ as
\begin{equation*}
V^\la\coloneqq \colim_h V^{h\text{-}\an}=\colim_h V_{h\text{-}\an}
\end{equation*}
and call $V$ itself \emph{locally analytic} if the natural map $V^\la\rightarrow V$ is an isomorphism. Finally, we let 
\begin{equation*}
\D(\Rep^\la_{(B, B^+)_\solid} G)\subseteq \D((B, B^+)_\solid[G])
\end{equation*}
denote the full subcategory spanned by locally analytic representations.
\end{defi}

\begin{rem}
It follows from \cite[Prop.\ 4.29]{Porat2} that the locally analytic vectors do not depend on the choice of $G_0$. One can also deduce this from the description of $V^\la$ from \cref{lem:porat-removenuclear} below.
\end{rem}

Again, we also introduce an overconvergent variant:

\begin{defi}
\label{defi:porat-hdaggergcompact}
Assume that $G$ is compact and fix $h\in\R$. For any $V\in\D((B, B^+)_\solid[G])$, its (derived) \emph{$h\dagger$-analytic vectors} are defined by
\begin{equation*}
V^{h\dagger\text{-}\an}\coloneqq \colim_{h'<h} V^{h'\text{-}\an}=\colim_{h'<h} V_{h'\text{-}\an}\;.
\end{equation*}
We call $V$ itself \emph{$h\dagger$-analytic} if the natural map $V^{h\dagger\text{-}\an}\rightarrow V$ is an isomorphism and let $\D(\Rep^{h\dagger\text{-}\an}_{(B, B^+)_\solid} G)\subseteq \D((B, B^+)_\solid[G])$ denote the full subcategory spanned by $h\dagger$-analytic representations.
\end{defi}

\begin{rem}
Again, we point out that formally setting $h=\infty$ recovers the definition of locally analytic vectors.
\end{rem}

Our first goal is to extend the definition of locally analytic and $h\dagger$-analytic vectors to representations of arbitrary $p$-adic Lie groups, i.e.\ to remove the compactness assumption. For this, we establish the following alternative description of locally analytic and $h\dagger$-analytic vectors, which is akin to the one from \cite[Def.\ 3.2.3.(3)]{SolidLocAnReps}. Before that, let us make the following definition:

\begin{defi}
For any $V\in\D((B, B^+)_\solid)$, the \emph{locally analytic functions} on $G$ with values in $V$ are defined by
\begin{equation*}
C^\la(G, V)\coloneqq \prod_{g\in G/G_0} C^\la(gG_0, B)\tensor_{(B, B^+)_\solid} V\;.
\end{equation*}
Similarly, we define the \emph{$h\dagger$-analytic functions} on $G$ with values in $V$ as
\begin{equation*}
C^{h\dagger\text{-}\an}(G, V)\coloneqq \prod_{g\in G/G_0} C^{h\dagger\text{-}\an}(gG_0, B)\tensor_{(B, B^+)_\solid} V
\end{equation*}
for any $h\in\R$.
\end{defi}

\begin{rem}
If $G$ is compact, we have $C^\la(G, V)\cong C^\la(G, B)\tensor_{(B, B^+)_\solid} V$ and similarly $C^{h\dagger\text{-}\an}(G, V)\cong C^{h\dagger\text{-}\an}(G, B)\tensor_{(B, B^+)_\solid} V$ for any $h\in\R$.
\end{rem}

The following lemma is key for our subsequent arguments. A similar statement appears in \cite[Prop.\ 5.4]{Porat2}, but needs an additional nuclearity assumption, which we are crucially able to drop here.

\begin{lem}
\label{lem:porat-removenuclear}
Let $G$ be a compact Lie group and $V\in\D((B, B^+)_\solid[G])$.
\begin{enumerate}[label=(\roman*)]
\item There is an isomorphism
\begin{equation*}
V^\la\cong\sHom_{(B, B^+)_\solid[G]}(B, C^\la(G, V)_{\star_{1, 3}})\;,
\end{equation*}
where the subscript $(-)_{\star_{1, 3}}$ indicates that $G$ acts on $C^\la(G, V)$ according to $g.(f\tensor v)=f(g^{-1}(-))\tensor gv$.
\item For any $h\in\R$, there is an isomorphism 
\begin{equation*}
V^{h\dagger\text{-}\an}\cong\sHom_{(B, B^+)_\solid[G]}(B, C^{h\dagger\text{-}\an}(G, V)_{\star_{1, 3}})\;.
\end{equation*}
\end{enumerate}
\end{lem}
\begin{proof}
We only prove the second part, the first part is proved by formally setting $h=\infty$ in the proof below. Unravelling definitions and using that $B$ is compact as a $(B, B^+)_\solid[G]$-module, which e.g.\ follows from \cite[Thm.\ 6.2]{Porat2}, we have to prove that 
\begin{equation*}
\colim_{h'<h} \sHom_{(B, B^+)_\solid[G]}(D^{h'\text{-}\an}(G, B), V)\cong \colim_{h'<h} \sHom_{(B, B^+)_\solid[G]}(B, C^{h'\text{-}\an}(G, V)_{\star_{1, 3}})\;.
\end{equation*}
In turn, this would follow from
\begin{equation*}
\colim_{h'<h}  \sHom_B(D^{h'\text{-}\an}(G, B), V)\cong \colim_{h'<h} C^{h'\text{-}\an}(G, B)\tensor_{(B, B^+)_\solid} V
\end{equation*}
by applying $\sHom_{(B, B^+)_\solid[G]}(B, -)$ and again using that $B$ is a compact $(B, B^+)_\solid[G]$-module.

However, as $D^{h'\text{-}\an}(G, B)$ is the $B$-linear dual of $C^{h'\text{-}\an}(G, B)$ by definition, there is a natural map 
\begin{equation*}
C^{h'\text{-}\an}(G, B)\tensor_{(B, B^+)_\solid} V\rightarrow \sHom_B(D^{h'\text{-}\an}(G, B), V)
\end{equation*}
and our task is to show that it becomes an isomorphism after taking colimits over $h'<h$. To this end, as both $D^{h'\text{-}\an}(G, B)$ and $C^{h'\text{-}\an}(G, B)$ are a finite direct sum of copies of $D^{h'\text{-}\an}(G_0, B)$ and $C^{h'\text{-}\an}(G_0, B)$, respectively, we may assume that $G=G_0$. We claim that the desired isomorphism follows once we prove that the transition maps $C^{h'\text{-}\an}(G_0, B)\rightarrow C^{h''\text{-}\an}(G_0, B)$ for $h'<h''<h$ are trace-class. To see this, we write $C^{h'}\coloneqq C^{h'\text{-}\an}(G_0, B)$ and $D^{h'}\coloneqq D^{h'\text{-}\an}(G_0, B)$ to ease notation and similarly for $h''$; moreover, we will implicitly take all tensor products over $(B, B^+)_\solid$ in the following. Then the definition of trace-class maps tells us that $C^{h'}\rightarrow C^{h''}$ is induced by a map $\eta: B\rightarrow D^{h'}\tensor C^{h''}$ and hence we obtain a commutative diagram
\begin{equation*}
\begin{tikzcd}
C^{h'}\tensor V\ar[r, "\id\tensor\eta\tensor\id"]\ar[d] & C^{h'}\tensor D^{h'}\tensor C^{h''}\tensor V\ar[r, "\mathrm{ev}\tensor\id\tensor\id"]\ar[d] & C^{h''}\tensor V\ar[d, equals] \\
\sHom_B(D^{h'}, V)\ar[r, "\id\tensor\eta"] & \sHom_B(D^{h'}, V)\tensor D^{h'}\tensor C^{h''}\ar[r, "\mathrm{ev}\tensor\id"] & V\tensor C^{h''}\nospacepunct{\;.}
\end{tikzcd}
\end{equation*}
In particular, we see that the map $C^{h'}\tensor V\rightarrow C^{h''}\tensor V$ factors through $\sHom_B(D^{h'}, V)$ and hence the natural map $C^{h'}\tensor V\rightarrow\sHom_B(D^{h'}, V)$ indeed induces an isomorphism on colimits over $h'<h$ by cofinality.

It thus remains to show that the maps $C^{h'\text{-}\an}(G_0, B)\rightarrow C^{h''\text{-}\an}(G_0, B)$ are indeed trace-class. For this, recall that
\begin{equation*}
C^{h'\text{-}\an}(G_0, B)\cong B^+\left\langle \pi^{v^{h'}(\ul{n})}\binom{\ul{x}}{\ul{n}}: \ul{n}\geq 0\right\rangle\tensor_{B^+} B\;.
\end{equation*}
In particular, we see that
\begin{equation*}
\sum_{\ul{n}} \left(\pi^{v^{h'}(\ul{n})}\binom{\ul{x}}{\ul{n}}\right)^\vee\in D^{h'\text{-}\an}(G_0, B)\;,
\end{equation*}
where $(-)^\vee$ denotes the dual vector; in other words, the dual vectors of $\pi^{v^{h'}(\ul{n})}\binom{\ul{x}}{\ul{n}}$ form a nullsequence in $D^{h'\text{-}\an}(G_0, B)$. As the $\pi^{v^{h'}(\ul{n})}\binom{\ul{x}}{\ul{n}}$ form a nullsequence in $C^{h''\text{-}\an}(G_0, B)$ for any $h'<h''<h$, we see that 
\begin{equation*}
\sum_{\ul{n}} \left(\pi^{v^{h'}(\ul{n})}\binom{\ul{x}}{\ul{n}}\right)^\vee\tensor\pi^{v^{h'}(\ul{n})}\binom{\ul{x}}{\ul{n}}
\end{equation*}
makes sense in $D^{h'\text{-}\an}(G_0, B)\tensor_{(B, B^+)_\solid} C^{h''\text{-}\an}(G_0, B)$ as nullsequences are summable in solid abelian groups. As the map $B\rightarrow D^{h'\text{-}\an}(G_0, B)\tensor_{(B, B^+)_\solid} C^{h''\text{-}\an}(G_0, B)$ classifying this sum induces the map $C^{h'\text{-}\an}(G_0, B)\rightarrow C^{h''\text{-}\an}(G_0, B)$, we are done.
\end{proof}

\begin{rem}
Note that the analogue of \cref{lem:porat-removenuclear} for (both upper and lower) $h$-analytic instead of locally analytic vectors is stated in \cite[Prop.\ 5.4]{Porat2} under the assumption that $V$ is nuclear as a $(B, B^+)_\solid$-module. The key idea above is that one can get rid of the nuclearity assumption by taking the colimit over $h$.
\end{rem}

This motivates the following definition:

\begin{defi}
Let $G$ be an arbitrary $p$-adic Lie group, possibly non-compact, and $V$ a $(B, B^+)_\solid[G]$-module. 
\begin{enumerate}[label=(\roman*)]
\item The (derived) \emph{locally analytic vectors} $V^\la$ of $V$ are defined as
\begin{equation*}
V^\la\coloneqq\sHom_{(B, B^+)_\solid[G]}(B, C^\la(G, V)_{\star_{1, 3}})\;.
\end{equation*}
If the natural map $V^\la\rightarrow V$ is an isomorphism, we call $V$ a \emph{locally analytic} $G$-representation and the corresponding full subcategory of $\D((B, B^+)_\solid[G])$ is denoted by $\D(\Rep^\la_{(B, B^+)_\solid} G)$.
\item For any $h\in\R$, the (derived) \emph{$h\dagger$-analytic vectors} $V^{h\dagger\text{-}\an}$ of $V$ are defined as
\begin{equation*}
V^{h\dagger\text{-}\an}\coloneqq\sHom_{(B, B^+)_\solid[G]}(B, C^{h\dagger\text{-}\an}(G, V)_{\star_{1, 3}})\;.
\end{equation*}
If the natural map $V^{h\dagger\text{-}\an}\rightarrow V$ is an isomorphism, we call $V$ an \emph{$h\dagger$-analytic} $G$-representation and the corresponding full subcategory of $\D((B, B^+)_\solid[G])$ is denoted by $\D(\Rep^{h\dagger\text{-}\an}_{(B, B^+)_\solid} G)$.
\end{enumerate}
\end{defi}

\begin{rem}
By \cref{lem:porat-removenuclear}, we see that the definition above recovers \cref{defi:porat-lagcompact} and \cref{defi:porat-hdaggergcompact}, respectively, in the case where $G$ is compact. 
\end{rem}

Next, we establish that the notions of locally analytic and $h\dagger$-analytic vectors are invariant under passing to an open subgroup:

\begin{lem}
\label{lem:porat-laopensub}
Let $G'\subseteq G$ be any open subgroup and $V\in\D((B, B^+)_\solid[G])$. 
\begin{enumerate}[label=(\roman*)]
\item We have
\begin{equation*}
V^{G\text{-}\la}\cong V^{G'\text{-}\la}
\end{equation*} 
via the natural map, where $V^{G'\text{-}\la}$ refers to the locally analytic vectors in $V$ as a $(B, B^+)_\solid[G']$-module.
\item Assume that $G'$ contains $G_0$ as an open subgroup. Then
\begin{equation*}
V^{G\text{-}h\dagger\text{-}\an}\cong V^{G'\text{-}h\dagger\text{-}\an}
\end{equation*}
for any $h\in\R$ via the natural map.
\end{enumerate}
\end{lem}
\begin{proof}
We only prove (ii), the proof of (i) is then formally obtained by setting $h=\infty$. By construction, we have
\begin{equation*}
C^{h\dagger\text{-}\an}(G, V)=\sHom_{(B, B^+)_\solid[G']}((B, B^+)_\solid[G], C^{h\dagger\text{-}\an}(G', V))\;.
\end{equation*}
Thus, we compute
\begin{equation*}
\begin{split}
V^{G\text{-}h\dagger\text{-}\an}&=\sHom_{(B, B^+)_\solid[G]}(B, C^{h\dagger\text{-}\an}(G, V)_{\star_{1, 3}}) \\
&\cong \sHom_{(B, B^+)_\solid[G]}(B, \sHom_{(B, B^+)_\solid[G']}((B, B^+)_\solid[G], C^{h\dagger\text{-}\an}(G', V))) \\
&\cong \sHom_{(B, B^+)_\solid[G']}(B, C^{h\dagger\text{-}\an}(G', V)_{\star_{1, 3}})=V^{G'\text{-}h\dagger\text{-}\an}
\end{split}
\end{equation*}
using the tensor-hom adjunction.
\end{proof}

\begin{cor}
\label{cor:porat-colimsla}
The endofunctors $V\mapsto V^\la$ and $V\mapsto V^{h\dagger\text{-}\an}$ of the category $\D((B, B^+)_\solid[G])$ commute with all colimits.
\end{cor}
\begin{proof}
By \cref{lem:porat-laopensub}, we may assume that $G$ is compact. Then $B$ is compact as a $(B, B^+)_\solid[G]$-module by \cite[Thm.\ 6.2]{Porat2} and this implies the claim as $V\mapsto C^\la(G, V)=C^\la(G, B)\tensor_{(B, B^+)_\solid} V$ and $V\mapsto C^\la(G, V)=C^{h\dagger\text{-}\an}(G, B)\tensor_{(B, B^+)_\solid} V$ commute with colimits in $V$.
\end{proof}

For compact $G$, recall from \cite[Thm.\ 6.10]{Porat2} that $D_{h\text{-}\an}(G, B)$ is an idempotent $(B, B^+)_\solid[G]$-algebra. In particular, we conclude that $\D(D_{h\text{-}\an}(G, B))$ is a full subcategory of $\D((B, B^+)_\solid[G])$ for any $h\in\R$.

\begin{lem}
\label{lem:porat-dlagenerators}
Assume that $G$ is compact. 
\begin{enumerate}[label=(\roman*)]
\item The category $\D(\Rep_{(B, B^+)_\solid}^\la G)$ is the full subcategory of $\D((B, B^+)_\solid[G])$ generated under colimits by objects $V\in\D(D_{h\text{-}\an}(G, B))$ for varying $h\in\R$.
\item For any $h\in\R$, the category $\D(\Rep_{(B, B^+)_\solid}^{h\dagger\text{-}\an} G)$ is the full subcategory of $\D((B, B^+)_\solid[G])$ generated under colimits by objects $V\in\D(D_{h'\text{-}\an}(G, B))$ for varying $h'<h$.
\end{enumerate}
\end{lem}
\begin{proof}
Again, we only prove (ii) as the proof of (i) is then obtained by formally setting $h=\infty$. By \cref{cor:porat-colimsla}, we see that $\D(\Rep_{(B, B^+)_\solid}^{h\dagger\text{-}\an} G)$ is indeed stable under colimits. As the $D_{h'\text{-}\an}(G, B)$ are idempotent, the claim immediately follows from $V^{h\dagger\text{-}\an}=\colim_{h'<h} V_{h'\text{-}\an}$ and $V_{h'\text{-}\an}=\sHom_{(B, B^+)_\solid[G]}(D_{h'\text{-}\an}(G, B), V)$ for any $V\in\D((B, B^+)_\solid[G])$ by definition.
\end{proof}

\begin{lem}
\label{lem:porat-laidempotent}
For all $h\leq h'$, we have 
\begin{equation*}
(V^{h\dagger\text{-}\an})^{h'\dagger\text{-}\an}\cong V^{h\dagger\text{-}\an}\cong (V^{h'\dagger\text{-}\an})^{h\dagger\text{-}\an}
\end{equation*}
via the natural maps. Similarly, we have 
\begin{equation*}
(V^{h\dagger\text{-}\an})^\la\cong V^{h\dagger\text{-}\an}\cong (V^\la)^{h\dagger\text{-}\an}
\end{equation*}
and $(V^\la)^\la\cong V^\la$. In particular, any $h\dagger$-analytic representation is $h'\dagger$-analytic for all $h'\geq h$ and locally analytic.
\end{lem}
\begin{proof}
By \cref{lem:porat-laopensub}, we may assume that $G$ is compact. As $(-)^{h\dagger\text{-}\an}$ and $(-)^\la$ commute with colimits by \cref{cor:porat-colimsla}, it then suffices to show that 
\begin{equation*}
(V_{h\text{-}\an})_{h'\text{-}\an}\cong V_{h\text{-}\an}\cong (V_{h'\text{-}\an})_{h\text{-}\an}
\end{equation*}
for all $h\leq h'$ via the natural maps. However, by definition, we have
\begin{equation*}
\begin{split}
(V_{h\text{-}\an})_{h'\text{-}\an}&=\sHom_{(B, B^+)_\solid[G]}(D_{h'\text{-}\an}(G, B), \sHom_{(B, B^+)_\solid[G]}(D_{h\text{-}\an}(G, B), V)) \\
&\cong \sHom_{(B, B^+)_\solid[G]}(D_{h'\text{-}\an}(G, B)\tensor_{(B, B^+)_\solid[G]} D_{h\text{-}\an}(G, B), V) \\
&\cong \sHom_{(B, B^+)_\solid[G]}(D_{h\text{-}\an}(G, B), V)=V_{h\text{-}\an}\;,
\end{split}
\end{equation*}
where the first isomorphism in the last row uses that $D_{h\text{-}\an}(G, B)$ is an idempotent $(B, B^+)_\solid[G]$-algebra and $D_{h'\text{-}\an}(G, B)$ is a $D_{h\text{-}\an}(G, B)$-module. One similarly proves that $(V_{h'\text{-}\an})_{h\text{-}\an}\cong V_{h\text{-}\an}$ and this finishes the proof.
\end{proof}

\begin{cor}
\label{cor:porat-adjunction}
The functors $V\mapsto V^\la$ and $V\mapsto V^{h\dagger\text{-}\an}$ are right-adjoint to the inclusions 
\begin{equation*}
\D(\Rep_{(B, B^+)_\solid}^\la G)\hookrightarrow \D((B, B^+)_\solid[G]) \hspace{0.3cm}\text{and}\hspace{0.3cm} \D(\Rep_{(B, B^+)_\solid}^{h\dagger\text{-}\an} G)\hookrightarrow \D((B, B^+)_\solid[G])\;,
\end{equation*}
respectively.
\end{cor}
\begin{proof}
We have $V^\la\in \D(\Rep_{(B, B^+)_\solid}^\la G)$ and $V^{h\dagger\text{-}\an}\in \D(\Rep_{(B, B^+)_\solid}^{h\dagger\text{-}\an} G)$ by \cref{lem:porat-laidempotent}, so it suffices to give the units and counits of the adjunctions. However, the counits are the natural maps $V^\la\rightarrow V$ and $V^{h\dagger\text{-}\an}\rightarrow V$, respectively, and the units come from the fact that these maps are isomorphisms if $V$ is locally analytic or $h\dagger$-analytic, respectively.
\end{proof}

Our last goal is to show that $\D(\Rep_{(B, B^+)_\solid}^\la G)$ and $\D(\Rep_{(B, B^+)_\solid}^{h\dagger\text{-}\an} G)$ have canonical $t$-structures for which they identify with the derived category of their hearts. To this end, first recall that $\D((B, B^+)_\solid[G])$ admits a canonical $t$-structure. 

\begin{lem}
\label{lem:porat-cohdimla}
The cohomological dimensions of the endofunctors $V\mapsto V^\la$ and $V\mapsto V^{h\dagger\text{-}\an}$ of $\D((B, B^+)_\solid[G])$ are bounded by $\dim G$.
\end{lem}
\begin{proof}
By \cref{lem:porat-laopensub}, we may assume that $G=G_0$. Then first observe that the functor $V\mapsto C^\la(G, V)$ has cohomological dimension zero as $C^\la(G, B)$ is flat over $(B, B^+)_\solid$, and similarly for $V\mapsto C^{h\dagger\text{-}\an}(G, V)$. Indeed, this is because $C^\la(G, B)=\colim_h C^{h\text{-}\an}(G, B)$ is a colimit of $B$-Banach spaces. Thus, \cref{lem:porat-removenuclear} reduces us to proving that $\sHom_{(B, B^+)_\solid[G]}(B, -)$ has cohomological dimension bounded by $\dim G$, but this follows from the Lazard--Serre resolution, see \cite[Thm.\ 6.2]{Porat2}.
\end{proof}

\begin{lem}
\label{lem:porat-underivedla}
An object $V\in \D((B, B^+)_\solid[G])$ is locally analytic if and only if $H^i(V)$ is locally analytic for all $i\in\Z$. In particular, the $t$-structure on $\D((B, B^+)_\solid[G])$ induces a $t$-structure on $\D(\Rep^\la_{(B, B^+)_\solid} G)$. The same holds for $h\dagger$-analytic representations for any $h\in\R$.
\end{lem}
\begin{proof}
We only prove the second part, the first is proved by formally setting $h=\infty$ in the proof below. By \cref{lem:porat-laopensub}, we may without loss of generality assume that $G$ is compact. If $V$ is $h\dagger$-analytic, then $V=V^{h\dagger\text{-}\an}=\colim_{h'<h} V_{h'\text{-}\an}$ and hence $H^i(V)=\colim_{h'<h} H^i(V_{h'\text{-}\an})$. As $V_{h'\text{-}\an}$ and hence $H^i(V_{h'\text{-}\an})$ is a $D_{h'\text{-}\an}(G, B)$-module, we conclude that $H^i(V)$ is $h\dagger$-analytic by \cref{lem:porat-dlagenerators}. Conversely, assume that $H^i(V)$ is $h\dagger$-analytic for all $i\in\Z$. By \cref{lem:porat-cohdimla}, we may assume that $V$ is a bounded complex and then induction on the cohomological degree reduces us to the case $V=H^0(V)$ using stability of $h\dagger$-analytic representations under shifts and colimits, in which case the claim is clear.
\end{proof}

\begin{defi}
We let $\Rep^\la_{(B, B^+)_\solid} G$ and $\Rep^{h\dagger\text{-}\an}_{(B, B^+)_\solid} G$ denote the hearts of the $t$-structures on $\D(\Rep^\la_{(B, B^+)_\solid} G)$ and $\D(\Rep^{h\dagger\text{-}\an}_{(B, B^+)_\solid} G)$ from \cref{lem:porat-underivedla}, respectively.
\end{defi}

The following result shows that our notation is not misleading:

\begin{prop}
\label{prop:porat-heart}
The category $\Rep^\la_{(B, B^+)_\solid} G$ is a Grothendieck abelian category whose derived category identifies with $\D(\Rep^\la_{(B, B^+)_\solid} G)$. The same is true for $h\dagger$-analytic representations.
\end{prop}
\begin{proof}
We only prove the first part, the proof of the second is entirely analogous. To show that $\Rep^\la_{(B, B^+)_\solid} G$ is a Grothendieck abelian category, it remains to see that it has a small family of generators. However, the category $\Rep^\la_{(B, B^+)_\solid} G$ is generated by the objects
\begin{equation*}
(B, B^+)_\solid[G]\tensor_{(B, B^+)_\solid[G_0]} D_{h\text{-}\an}(G_0, B)\tensor_{(B, B^+)_\solid} (B, B^+)_\solid[S]
\end{equation*}
for $h\in\R$ and $S$ a light profinite set by \cref{lem:porat-laopensub} and \cref{lem:porat-dlagenerators}.

Now take any $V\in\Rep^\la_{(B, B^+)_\solid} G$ and an injective resolution $I^\bullet$ of $V$ as a $(B, B^+)_\solid[G]$-module. As the functor $V\mapsto H^0(V^\la)$ is right-adjoint to the inclusion $\Rep^\la_{(B, B^+)_\solid} G\rightarrow \D((B, B^+)_\solid[G])^\heartsuit$ by \cref{cor:porat-adjunction}, it preserves injective objects and the complex $H^0(I^{\bullet, \la})$ is an injective resolution of $V=H^0(V^\la)$ in $\Rep^\la_{(B, B^+)_\solid} G$. In particular, we conclude that the derived internal $\Hom$ in $\Rep^\la_{(B, B^+)_\solid} G$ agrees with the one in $\D((B, B^+)_\solid[G])^\heartsuit$. This yields the claim using the fact that $\D(\Rep^\la_{(B, B^+)_\solid} G)$ is left-complete by \cref{lem:porat-cohdimla}.
\end{proof}

The abelian categories $\Rep^\la_{(B, B^+)_\solid} G$ and $\Rep^{h\dagger\text{-}\an}_{(B, B^+)_\solid} G$ may be described as the categories of comodules over $C^\la(G, B)$ and $C^{h\dagger\text{-}\an}(G, B)$, respectively.

\begin{lem}
\label{lem:porat-comodules}
There is an equivalence between $\Rep^\la_{(B, B^+)_\solid} G$ and the abelian category of comodules over the endofunctor $V\mapsto C^\la(G, V)$ of $\D((B, B^+)_\solid)^\heartsuit$. Similarly, the category $\Rep^{h\dagger\text{-}\an}_{(B, B^+)_\solid} G$ is equivalent to the abelian category of comodules over the endofunctor $V\mapsto C^{h\dagger\text{-}\an}(G, V)$ of $\D((B, B^+)_\solid)^\heartsuit$.
\end{lem}
\begin{proof}
We only prove the first part, the second part is entirely analogous. First assume that $V\in\D((B, B^+)_\solid)^\heartsuit$ is a $C^\la(G, -)$-comodule. Then the coaction $V\rightarrow C^\la(G, V)$ yields a natural map $V\rightarrow\sHom_B(D^\la(G, B), V)$, which corresponds to a map $D^\la(G, B)\tensor_{(B, B^+)_\solid} V\rightarrow V$ by adjunction; here we write $D^\la(G, B)$ for the $B$-linear dual of $C^\la(G, B)$. In particular, precomposing this map with $(B, B^+)_\solid[G]\rightarrow D^\la(G, B)$ endows $V$ with the structure of a $(B, B^+)_\solid[G]$-module. Moreover, one easily sees that the coaction map $V\rightarrow C^\la(G, V)$ then lands in the invariants of the $\star_{1, 3}$-action of $G$ on the target. Thus, we conclude that $V$ is a retract of $V^\la$ and hence itself locally analytic, i.e.\ $V\in\Rep^\la_{(B, B^+)_\solid} G$. Conversely, given $V\in\Rep^\la_{(B, B^+)_\solid} G$, the orbit map $V\rightarrow C^\la(G, V)$ turns $V$ into a $C^\la(G, -)$-comodule and one easily sees that these constructions are mutually inverse, which concludes the proof.
\end{proof}

Finally, let us record the following useful criterion for checking that a given representation of $G$ is locally analytic: 

\begin{prop}
\label{prop:porat-recogniselocan}
Let $V^+\in \D(B^+_\solid[G])$ be $\pi$-complete and static and assume that there exists some $m\geq 0$ with
\begin{equation*}
(g-1)^mv\in \pi V^+
\end{equation*}
for all $v\in V^+$ and $g\in G$. Then the representation $V\coloneqq V^+\tensor_{B^+} B$ is locally analytic. The same is true if there is some $m\geq 0$ such that
\begin{equation*}
(g-1)^m v\in p V^+
\end{equation*}
for all $v\in V^+$ and $g\in G$.
\end{prop}
\comment{
Before we begin the proof, we establish an easy lemma which says that higher locally analytic vectors do not play a role if we want to detect local analyticity.

\begin{lem}
\label{lem:porat-underivedlocan}
Assume that $V\in \D((B, B^+)_\solid[G])$ is static. Then $V$ is locally analytic if and only if $H^0(V^\la)\cong V$ via the natural map.
\end{lem}
\begin{proof}
Without loss of generality, we may assume that $G$ is compact. If $V$ is locally analytic, we have $H^0(V^\la)\cong V^\la\cong V$ as $V$ is static by assumption. Conversely, assume that $H^0(V^\la)\cong V$. Due to $H^0(V^\la)=\colim_h H^0(V_{h\text{-}\an})$, this shows that $V$ is a colimit of $D_{h\text{-}\an}(G, B)$-modules and hence locally analytic by \cref{lem:porat-dlagenerators}.
\end{proof}

Now we can prove \cref{prop:porat-recogniselocan}.
}
\begin{proof}
First note that the second condition implies the first as $\pi$ divides $p$ in $B^+$ by our standing slope assumption. Without loss of generality, we can assume that $G=G_0$ is uniform. Choosing an isomorphism $G_0\cong \Z_p^d$, we obtain elements $g_1, g_2, \dots, g_d\in G_0$ such that any element $g\in G_0$ is of the form $g=\prod_i g_i^{x_i}$ for some $x_i\in \Z_p$. For any fixed $v\in V^+$, one easily shows by induction that
\begin{equation*}
gv=\sum_{\ul{n}} \binom{\ul{x}}{\ul{n}}(g_1-1)^{n_1}\cdots (g_d-1)^{n_d}v
\end{equation*}
for any $d$-tuple $\ul{x}=(x_1, \dots, x_d)$ of integers and then the formula also holds for $x_i\in\Z_p$ by continuity. By assumption, we know that $(g_1-1)^{n_1}\cdots (g_d-1)^{n_d}v$ is divisible by $\pi$ at least $|\ul{n}|/m$ times and hence the orbit map of $v$ yields an element in $C^\la(G, V)$ by definition. Thus, the assignment $v\mapsto (g\mapsto gv)$ yields a section of the map $V^\la\rightarrow V$ and we conclude that $V$ is a retract of a locally analytic representation, hence itself locally analytic. This finishes the proof.
\end{proof}

\begin{ex}
\label{ex:porat-adjointlocan}
For any uniform $p$-adic Lie group $G_0$ of dimension $d$, there is an associated Lie algebra $\Lambda(G_0)$ defined in \cite[§7.2]{ProPGroups}, whose underlying $\Z_p$-module is free of rank $d$. There is an adjoint action of $G_0$ on $\Lambda(G_0)$ which recovers the usual adjoint action on $\mathfrak{g}_0=\Lambda(G_0)\tensor_{\Z_p} \Q_p$. By \cite[§9, Exc.\ 10]{ProPGroups}, this action satisfies the condition from \cref{prop:porat-recogniselocan} and thus $\Lambda(G_0)\tensor_{\Z_p} B$ is a locally analytic representation of $G_0$.
\end{ex}

\section{Locally analytic representations over $\cal{N}_{|p|<1}/\R_{>0}$}

Let $G$ be a $p$-adic Lie group. We continue to fix an open subgroup $G_0\subseteq G$ such that $G_0\cong\Z_p^d$ as $p$-adic manifolds for some $d\geq 0$. Our aim is to interpret locally analytic representations of $G$ in the sense of the previous section as sheaves on the classifying stack of $G^\la$ and to give a similar stacky interpretation for $h\dagger$-analytic representations.

To this end, we first introduce an ``$h\dagger$-analytic analogue'' of the group stack $G^\la$. Namely, for any $k\geq 1$, let $m\coloneqq p^k$ and consider the Banach pair
\begin{equation}
\label{eq:locan-bb+}
(B_k, B_k^+)\coloneqq (\Z(\!(\pi^{1/m})\!)\langle\tfrac{p^m}{\pi}\rangle, \widetilde{\Z+B_k^{\circ\circ}})
\end{equation}
equipped with the pseudouniformiser $\pi^{1/m}$, where $\widetilde{\Z+B_k^{\circ\circ}}$ denotes the integral closure of $\Z+B_k^{\circ\circ}$ inside $B_k$. This has slope $\geq 1$, is residually of finite type and the analytic ring $(B_k, B_k^+)_\solid$ has the induced analytic ring structure from $\Z_\solid$. Moreover, there is a canonical map
\begin{equation*}
B_k=\Z(\!(\pi^{1/m})\!)\langle\tfrac{p^m}{\pi}\rangle\rightarrow \Z(\!(\pi^{1/m})\!)\langle p\rangle_{\leq (1/4)^{1/m}}\coloneqq \colim_{\epsilon>0} \Z(\!(\pi)\!)\langle p/\pi^{2/m-\epsilon}\rangle
\end{equation*}
and we define
\begin{equation*}
C^{(h-k)\dagger\text{-}\an}(G_0, \Z(\!(\pi^{1/m})\!)\langle p\rangle_{\leq (1/4)^{1/m}})\coloneqq C^{(h-k)\dagger\text{-}\an}(G_0, B_k)\tensor_{(B_k, B_k^+)_\solid} \Z(\!(\pi^{1/m})\!)\langle p\rangle_{\leq (1/4)^{1/m}}\;.
\end{equation*}
As $\Z(\!(\pi)\!)\rightarrow\Z(\!(\pi^{1/m})\!)$ is descendable, the map
\begin{equation*}
\AnSpec\Z_\solid(\!(\pi^{1/m})\!)\langle p\rangle_{\leq (1/4)^{1/m}}\rightarrow \cal{N}_{|p|\leq (1/4)^{1/m}}
\end{equation*}
is a surjection by \cref{cor:norms-pres} and a similar argument as in the proof of \cref{lem:man-descent} shows that $\AnSpec C^{(h-k)\dagger\text{-}\an}(G_0, \Z(\!(\pi^{1/m})\!)\langle p\rangle_{\leq (1/4)^{1/m}})$ descends to $\cal{N}_{|p|\leq (1/4)^{1/m}}$. We denote the descent by $(G_0^{h\dagger\text{-}\an})_{|p|\leq (1/4)^{1/m}}$ and observe that this has the structure of a group stack by virtue of the fact that $C^{(h-k)\dagger\text{-}\an}(G_0, B_k)$ is a Hopf algebra.

Now observe that the group stacks $(G_0^{h\dagger\text{-}\an})_{|p|\leq (1/4)^{1/m}}$ we have constructed are compatible under base change, i.e.\ there are canonical isomorphisms
\begin{equation*}
(G_0^{h\dagger\text{-}\an})_{|p|\leq (1/4)^{1/m}}\times_{\cal{N}_{|p|\leq (1/4)^{1/m}}} \cal{N}_{|p|\leq (1/4)^{1/m'}}\cong (G_0^{h\dagger\text{-}\an})_{|p|\leq (1/4)^{1/m'}}
\end{equation*}
for all $k'\leq k$, where $m'\coloneqq p^{k'}$. Indeed, given a Banach pair $(B, B^+)$ with pseudouniformiser $\pi$, replacing $\pi$ by $\pi^p$ replaces $C^{h\dagger\text{-}\an}$ by $C^{(h-1)\dagger\text{-}\an}$. Thus, the various $(G_0^{h\dagger\text{-}\an})_{|p|\leq (1/4)^{1/m}}$ glue to a group stack $G_0^{h\dagger\text{-}\an}$ over
\begin{equation*}
\colim_{m=p^k} \cal{N}_{|p|\leq (1/4)^{1/m}}\cong \cal{N}_{|p|<1}\;.
\end{equation*}
Finally, we define
\begin{equation*}
G^{h\dagger\text{-}\an}\coloneqq \coprod_{g\in G/G_0} (gG_0)^{h\dagger\text{-}\an}\;.
\end{equation*}

\begin{rem}
Note that $G^{h\dagger\text{-}\an}$ does \emph{not} descend further to $\cal{N}_{|p|<1}/\R_{>0}$ as the $\R_{>0}$-action changes the value of $h$.
\end{rem}

\begin{defi}
Let $(B, B^+)$ be a complete Tate Huber pair for which the canonical map $\AnSpec\hspace{1pt}(B, B^+)_\solid\rightarrow \cal{N}/\R_{>0}$ from \cref{rem:norms-maptatehuber} factors through $\cal{N}_{|p|<1}/\R_{>0}$. 
\begin{enumerate}[label=(\roman*)]
\item A $(B, B^+)$-linear \emph{locally analytic representation} of $G$ is defined to be an object of the category 
\begin{equation*}
\D(\AnSpec\hspace{1pt}(B, B^+)_\solid/G^\la)\;,
\end{equation*}
where the quotient is by the trivial action of $G^\la$. 

\item If the map $\AnSpec (B, B^+)_\solid\rightarrow\cal{N}_{|p|<1}/\R_{>0}$ is in addition equipped with a lift to $\cal{N}_{|p|<1}$, we define a $(B, B^+)$-linear \emph{$h\dagger$-analytic representation} of $G$ to be an object of 
\begin{equation*}
\D(\AnSpec\hspace{1pt}(B, B^+)_\solid/G^{h\dagger\text{-}\an})
\end{equation*}
for any $h\in\R$.
\end{enumerate}
\end{defi}

We now prove that this is compatible with the terminology from the previous section whenever the pair $(B, B^+)$ satisfies the conditions from loc.\ cit.

\begin{thm}
\label{thm:locan-compareporat}
Assume that $(B, B^+)$ is a complete Tate Huber pair over $(\Z(\!(\pi)\!), \Z[\![\pi]\!])$ of slope $\geq 1$ which is residually of finite type. Then pullback along $\AnSpec\hspace{1pt}(B, B^+)_\solid\rightarrow \AnSpec\hspace{1pt}(B, B^+)_\solid/G^\la$ induces an equivalence
\begin{equation*}
\D(\AnSpec\hspace{1pt}(B, B^+)_\solid/G^\la)\cong \D(\Rep^\la_{(B, B^+)_\solid} G)\;.
\end{equation*}
Similarly, pullback along $\AnSpec\hspace{1pt}(B, B^+)_\solid\rightarrow \AnSpec\hspace{1pt}(B, B^+)_\solid/G^{h\dagger\text{-}\an}$ induces an equivalence 
\begin{equation*}
\D(\AnSpec\hspace{1pt}(B, B^+)_\solid/G^{h\dagger\text{-}\an})\cong \D(\Rep^{h\dagger\text{-}\an}_{(B, B^+)_\solid} G)
\end{equation*}
for any $h\in\R$.
\end{thm}
\begin{proof}
We prove the first part, the second is entirely analogous. For simplicity, write $*\coloneqq\AnSpec\hspace{1pt}(B, B^+)_\solid$ and work over this base throughout. Noting that $G^\la/G_0^\la\cong G/G_0$ is the Betti stack of a discrete set and using \cref{lem:porat-laopensub}, we are reduced to the case $G=G_0$ by descent along $*/G_0^\la\rightarrow */G^\la$. In this case, by \cref{prop:porat-heart} and \cref{lem:porat-comodules}, we have to show that
\begin{equation*}
\D(*/G_0^\la)\cong \D(\CoMod_{C^\la(G_0, B)}(\D((B, B^+)_\solid)^\heartsuit))\;.
\end{equation*}
For this, note that $C^\la(G_0, B)$ is a flat $(B, B^+)_\solid$-module as it is a colimit of $B$-Banach spaces in the sense of \cite[Def.\ 3.12]{Porat2}. Together with the fact that $G_0^\la=\AnSpec C^\la(G_0, B)$ is prim as $C^\la(G_0, B)$ has the induced analytic ring structure from $(B, B^+)_\solid$, this implies that $\D(*/G_0^\la)$ identifies with the derived category of the abelian category of descent data along $*\rightarrow */G_0^\la$ by \cite[Prop.\ A.1.2]{MannThesis}. As the latter is equivalent to the abelian category of $C^\la(G_0, B)$-comodules in $\D((B, B^+)_\solid)^\heartsuit$, we are done.
\end{proof}

The above proposition lets us descend many results of the previous section to the stack of norms (or its quotient by $\R_{>0}$). An example of this is the following:

\begin{cor}
\label{cor:locan-embedgreps}
Let $\O[G]$ denote the algebra object in $\D(\cal{N}_{|p|<1}/\R_{>0})$ or $\D(\cal{N}_{|p|<1})$, respectively, obtained by pulling back $\Z_\solid[G]$ from $\AnSpec\Z_\solid$. Then there are fully faithful embeddings
\begin{equation*}
\D(*/G^\la)\hookrightarrow \Mod_{\O[G]}(\D(\cal{N}_{|p|<1}/\R_{>0}))\;, \hspace{0.3cm} \D(*/G^{h\dagger\text{-}\an})\hookrightarrow\Mod_{\O[G]}(\D(\cal{N}_{|p|<1}))
\end{equation*}
for all $h\in\R$, where $*=\cal{N}_{|p|<1}/\R_{>0}$ or $*=\cal{N}_{|p|<1}$, respectively. In particular, the pushforwards along
\begin{equation*}
*/G^\la\rightarrow \cal{N}_{|p|<1}/\R_{>0}\hspace{0.5cm}\text{and}\hspace{0.5cm}*/G^{h\dagger\text{-}\an}\rightarrow\cal{N}_{|p|<1}
\end{equation*}
compute condensed $G$-cohomology.
\end{cor}
\begin{proof}
In the construction of $G^{h\dagger\text{-}\an}$ above, we have given a family of Banach pairs $(B, B^+)$ over $(\Z(\!(\pi)\!), \Z[\![\pi]\!])$ which are residually of finite type and of slope $\geq 1$ together with maps $\AnSpec\hspace{1pt}(B, B^+)_\solid\rightarrow\cal{N}_{|p|<1}$ which jointly cover $\cal{N}_{|p|<1}$. Then the claim follows using \cref{thm:locan-compareporat} by descending the fact that $\D(\Rep^\la_{(B, B^+)_\solid} G)$ and $\D(\Rep^{h\dagger\text{-}\an}_{(B, B^+)_\solid} G)$ are full subcategories of $\D((B, B^+)_\solid[G])$. Finally, the last assertion is deduced by comparing internal $\Hom$ under the fully faithful embeddings above.
\end{proof}

Let us also note the following identification of pushforwards under the equivalence from \cref{thm:locan-compareporat}:

\begin{lem}
\label{lem:locan-pushforward}
Assume that $(B, B^+)$ is a complete Tate Huber pair over $(\Z(\!(\pi)\!), \Z[\![\pi]\!])$ of slope $\geq 1$ which is residually of finite type. For any $h\in\R$, pushforward along the map
\begin{equation*}
\AnSpec\hspace{1pt}(B, B^+)_\solid/G^\la\rightarrow \AnSpec\hspace{1pt}(B, B^+)_\solid/G^{h\dagger\text{-}\an}
\end{equation*}
identifies with the functor $(-)^{h\dagger\text{-}\an}$ under the equivalence from \cref{thm:locan-compareporat}. The same is true for pushforward along the map
\begin{equation*}
\AnSpec\hspace{1pt}(B, B^+)_\solid/G^{h'\dagger\text{-}\an}\rightarrow \AnSpec\hspace{1pt}(B, B^+)_\solid/G^{h\dagger\text{-}\an}
\end{equation*}
for any $h'\geq h$.
\end{lem}
\begin{proof}
We only prove the second part, the proof of the first part is entirely analogous. For this, note that pullback identifies with the forgetful functor
\begin{equation*}
\D(\Rep^{h\dagger\text{-}\an}_{(B, B^+)_\solid} G)\hookrightarrow \D(\Rep^{h'\dagger\text{-}\an}_{(B, B^+)_\solid} G)
\end{equation*}
and hence pushforward identifies with the right-adjoint of this forgetful functor, which is given by $(-)^{h\dagger\text{-}\an}$ by \cref{cor:porat-adjunction}, as desired.
\end{proof}

\begin{ex}
We give a nontrivial example of an object in the category $\D(*/\Z_p^{\times, \la})$, where $*\coloneqq\cal{N}_{|p|<1}/\R_{>0}$. For this, first note that there is a universal norm (up to rescaling) 
\begin{equation*}
|\cdot|: \P^1\rightarrow [0, \infty]/\R_{>0}
\end{equation*}
over $*=\cal{N}_{|p|<1}/\R_{>0}$ and we let $\overcirc{\DD}$ be the preimage of $[0, 1)/\R_{>0}$ under this map. Letting $T$ denote the coordinate on $\P^1$, we observe that there is an action of $\Z_p^{\times, \la}$ on $\overcirc{\DD}$ given by 
\begin{equation*}
x.T\coloneqq (1+T)^x-1=\sum_{n\geq 1} \binom{x}{n}T^n
\end{equation*}
for $x\in\Z_p^\times$; indeed, the right-hand side is a locally analytic function since $T^n$ has exponential decay on $\overcirc{\DD}$ with respect to the norm. Using this action, we obtain a map $f: \overcirc{\DD}/\Z_p^{\times, \la}\rightarrow */\Z_p^{\times, \la}$ and we can consider
\begin{equation*}
f_!\O\in\D(*/\Z_p^{\times, \la})\;.
\end{equation*}
Let us make this representation a little more explicit by pulling back along $\AnSpec\F_{p, \solid}(\!(\pi)\!)\rightarrow\cal{N}_{|p|<1}/\R_{>0}$. Then we have
\begin{equation*}
\overcirc{\DD}_{\F_p(\!(\pi)\!)}=\colim_k \AnSpec\F_{p, \solid}(\!(\pi)\!)\langle\tfrac{T^k}{\pi}\rangle\;,
\end{equation*}
which is exactly the complement of the closed embedding
\begin{equation*}
\AnSpec\F_{p, \solid}(\!(\pi)\!)\langle T^{\pm 1}\rangle_{\leq 1}\xrightarrow{i} \AnSpec\F_{p, \solid}(\!(\pi)\!)\langle T\rangle_{\leq 1}\;.
\end{equation*}
As both source and target of $i$ are weakly cohomologically proper over $\AnSpec \F_{p, \solid}(\!(\pi)\!)$, we can use excision to compute the compactly supported cohomology of $\DD_{\F_p(\!(\pi)\!)}$ as
\begin{equation*}
\fib(\F_p(\!(\pi)\!)\langle T\rangle_{\leq 1}\rightarrow \F_p(\!(\pi)\!)\langle T^{\pm 1}\rangle_{\leq 1})=T^{-1}\F_p(\!(\pi)\!)\langle T^{-1}\rangle_{\leq 1}[-1]
\end{equation*}
and hence $f_!\O$ is concentrated in degree $1$. Computing the action of $\Z_p^\times$ on the first few powers of $T^{-1}$, we get
\begin{equation*}
x.T^{-1}=x^{-1}T^{-1}\;, \hspace{0.5cm} x.T^{-2}=x^{-2}T^{-2}-(x-1)x^{-2}T^{-1}\;, \hspace{0.5cm} \dots
\end{equation*}
In particular, we see that the action of $\Z_p^\times$ is not smooth.
\end{ex}

\section{Comparison with smooth representations over $\F_p$}

In this section, we prove the following theorem:

\begin{thm}
\label{thm:smoothfp-main}
Let $G$ be a $p$-adic Lie group. For any solid $\F_p$-algebra $B$, there is a fully faithful embedding
\begin{equation*}
\widehat{\D}(\Rep^\sm_B G)\hookrightarrow \D((\cal{N}_B/\R_{>0})/G^\la)\;,
\end{equation*}
where $\widehat{\D}(\Rep^\sm_B G)$ denotes the left-complete derived category of smooth $G$-representations on solid $B$-modules.
\end{thm}

To prove this, we will make use of the fact that smooth representations have an interpretation in terms of stacks as well. Namely, recall from \cite[§II.1]{RealLLC} that there is a colimit-preserving functor
\begin{equation*}
(-)^\Betti: \mathrm{CondAni}^{\mathrm{light}}\rightarrow \{\text{analytic stacks over $\Z_\solid$}\}
\end{equation*}
which sends any light profinite set $S$ to $\AnSpec\Cont(S, \Z)$, where we endow $\Cont(S, \Z)$ with the induced analytic ring structure from $\Z_\solid$. (We have already made use of the Betti stack of $\R_{>0}$.)

\begin{defi}
For any locally profinite group $G$, we let $G^\sm$ denote the analytic stack $G^\Betti$.
\end{defi}

By definition, we can also describe $G^\sm$ as
\begin{equation*}
G^\sm\cong \coprod_{g\in G/G_0} \AnSpec C^\sm(gG_0, \Z)
\end{equation*}
for any compact open subgroup $G_0\subseteq G$, where $C^\sm(G_0, \Z)=\Cont(G_0, \Z)$ is the algebra of smooth $\Z$-valued functions on $G_0$ and we endow it with the induced analytic ring structure from $\Z_\solid$. The notation $G^\sm$ is justified by the following result:

\begin{prop}
For any locally profinite group $G$ and any solid $\Z$-algebra $B$, there is an equivalence of categories
\begin{equation*}
\D(\AnSpec B/G^\sm)\cong \widehat{\D}(\Rep^\sm_B G)\;,
\end{equation*}
where we use the induced analytic ring structure on $B$ from $\Z_\solid$.
\end{prop}
\begin{proof}
By definition, the assignment $X\mapsto \D(X^\Betti\times\AnSpec B)$ for light condensed anima $X$ agrees with the six-functor formalism of $\Lambda$-valued sheaves on light condensed anima from \cite[§3.5]{HeyerMann} for $\Lambda$ being the analytic ring $B$.\footnote{Note that all results and constructions from loc.\ cit.\ also work for $\Lambda$ being an analytic ring.} Then the claim is \cite[Prop.\ 5.1.12]{HeyerMann}.
\end{proof}

By virtue of the above proposition, we will realise the equivalence in \cref{thm:smoothfp-main} as pullback along the map
\begin{equation*}
(\cal{N}_B/\R_{>0})/G^\la\rightarrow (\cal{N}_B/\R_{>0})/G^\sm\rightarrow \AnSpec B/G^\sm
\end{equation*}
induced by the inclusion of smooth functions into locally analytic functions. Let us first recall the so-called Lazard--Serre and Kohlhaase resolutions in our setting.

\begin{prop}
\label{prop:smoothfp-lazardserre}
Let $G_0$ be a uniform $p$-adic Lie group of dimension $d$. For any solid $\Z_p$-algebra $B^+$, there is a resolution
\begin{equation*}
\begin{tikzcd}[column sep=small]
0\ar[r] & B^+_\solid[G_0]^{\binom{d}{d}}\ar[r] & B^+_\solid[G_0]^{\binom{d}{d-1}}\ar[r]& \cdots\ar[r] & B^+_\solid[G_0]^{\binom{d}{0}}\ar[r] & B^+
\end{tikzcd}
\end{equation*}
of the trivial representation of $G_0$.
\end{prop}
\begin{proof}
See \cite[Thm.\ 5.7]{SolidLocAnReps1}.
\end{proof}

\begin{prop}
\label{prop:smoothfp-kohlhaase}
Let $G_0$ be a uniform $p$-adic Lie group of dimension $d$ and let $h\geq 0$. For any complete Tate Huber pair $(B, B^+)$ over $(\Z(\!(\pi)\!), \Z[\![\pi]\!])$ of slope $\geq 1$ which is residually of finite type, there is a resolution
\begin{equation*}
\begin{tikzcd}[column sep=small]
0\ar[r] & D_{h\text{-}\an}(G_0, B)^{\binom{d}{d}}\ar[r] & D_{h\text{-}\an}(G_0, B)^{\binom{d}{d-1}}\ar[r]& \cdots\ar[r] & D_{h\text{-}\an}(G_0, B)^{\binom{d}{0}}\ar[r] & B
\end{tikzcd}
\end{equation*}
of the trivial representation of $G_0$ extending the Lazard--Serre resolution.
\end{prop}
\begin{proof}
See \cite[Prop.\ 6.5]{Porat2}.
\end{proof}

We now start working towards the proof of \cref{thm:smoothfp-main} by proving some preparatory lemmas.

\begin{lem}
\label{lem:smoothfp-desc}
Let $G_0$ be a uniform $p$-adic Lie group. Then the morphism 
\begin{equation*}
f: \cal{N}_{|p|<1}/\R_{>0}\rightarrow (\cal{N}_{|p|<1}/\R_{>0})/G_0^\la 
\end{equation*}
is prim and $f_*\O\in \D((\cal{N}_{|p|<1}/\R_{>0})/G_0^\la)$ is descendable.
\end{lem}
\begin{proof}
As primness can be checked locally on the target, it suffices to check that 
\begin{equation*}
G_0^\la\times_{\cal{N}_{|p|<1}/\R_{>0}}\AnSpec\Z_\solid(\!(\pi)\!)\langle\tfrac{p^k}{\pi}\rangle_{\leq 1}\rightarrow \AnSpec\Z_\solid(\!(\pi)\!)\langle\tfrac{p^k}{\pi}\rangle_{\leq 1} 
\end{equation*}
is prim for all $k\geq 1$, which follows from the fact that this is a map between affine analytic stacks with induced analytic ring structure. 

For the second part, observing that $f_*$ satisfies base change by primness of $f$, we prove descendability of $f_*\O$ after pullback along the maps
\begin{equation*}
\AnSpec\hspace{1pt}(B_k, B_k^+)_\solid\rightarrow \cal{N}_{|p|<1}
\end{equation*} 
from (\ref{eq:locan-bb+}), which jointly cover the target for $k\geq 1$, and note that our argument will descend. In other words, we show that $f_{k*}\O$ is descendable for all $k$, where
\begin{equation*}
f_k: \AnSpec\hspace{1pt}(B_k, B_k^+)_\solid\rightarrow \AnSpec\hspace{1pt}(B_k, B_k^+)_\solid/G_0^\la\;.
\end{equation*}
Unraveling definitions, we see that $f_{k*}\O$ identifies with the regular representation $C^\la(G_0, B_k)$ under the equivalence from \cref{thm:locan-compareporat}. Taking the dual of the Kohlhaase resolution from \cref{prop:smoothfp-kohlhaase} and then taking the colimit as $h\rightarrow\infty$, we obtain an exact sequence
\begin{equation*}
\begin{tikzcd}[column sep=small]
0\ar[r] & B_k\ar[r] & C^\la(G_0, B_k)^{\binom{d}{0}}\ar[r] & \cdots\ar[r] & C^\la(G_0, B_k)^{\binom{d}{d-1}}\ar[r] & C^\la(G_0, B_k)^{\binom{d}{d}}\ar[r] & 0
\end{tikzcd}
\end{equation*}
of locally analytic $G_0$-representations, where $d$ denotes the dimension of $G_0$. This shows that the trivial representation $B_k$ can be obtained from $C^\la(G_0, B_k)$ by taking finitely many direct sums, fibres and cofibres, hence $C^\la(G_0, B_k)$ is descendable, as desired.
\end{proof}

\begin{lem}
\label{lem:smoothfp-glagsmprim}
Let $G_0$ be a uniform $p$-adic Lie group. Then the morphism
\begin{equation*}
f: (\cal{N}_{\F_p}/\R_{>0})/G_0^\la\rightarrow (\cal{N}_{\F_p}/\R_{>0})/G_0^\sm\rightarrow \AnSpec \F_{p, \solid}/G_0^\sm
\end{equation*}
is prim.
\end{lem}
\begin{proof}
To check that $(\cal{N}_{\F_p}/\R_{>0})/G_0^\la\rightarrow (\cal{N}_{\F_p}/\R_{>0})/G_0^\sm$ is prim, it suffices to check primness of $\cal{N}_{\F_p}/\R_{>0}\rightarrow (\cal{N}_{\F_p}/\R_{>0})/G_0^\sm$ by combining \cref{lem:smoothfp-desc} with \cite[Cor.\ 4.7.5]{HeyerMann}. However, this is a $G_0^\sm$-torsor and $G_0^\sm=\AnSpec C^\sm(G_0, \Z)$ is an affine analytic stack equipped with the induced analytic ring structure from $\Z_\solid$, hence prim over $\AnSpec\Z_\solid$. It thus remains to check that $(\cal{N}_{\F_p}/\R_{>0})/G_0^\sm\rightarrow\AnSpec\F_{p, \solid}/G_0^\sm$ is prim as well, but this follows from \cref{lem:norms-primmodr}.
\end{proof}

\begin{lem}
\label{lem:smoothfp-cohzpla}
For any $d\geq 0$, consider the map
\begin{equation*}
g: (\Z_p^d)^\la\times_{\cal{N}_{|p|<1}/\R_{>0}} \cal{N}_{\F_p}/\R_{>0}\rightarrow \AnSpec\F_{p, \solid}\;.
\end{equation*}
Then we have
\begin{equation*}
g_*\O\cong \F_p\left[\binom{\ul{x}}{\ul{n}}: \ul{n}\geq 0\right]\;.
\end{equation*}
\end{lem}
\begin{proof}
Using a similar calculation as in the proofs of \cref{prop:norms-embed} and \cref{prop:norms-embedzp}, we see that $g_*\O$ is given by the degree zero subspace of $C^\la(\Z_p^d, \Z(\!(\pi)\!))\tensor_{\Z} \F_p$ with respect to the $\ol{\T}$-action (and the higher cohomology groups vanish). However, this is precisely given by the right-hand side of the claimed isomorphism due to \cref{lem:man-zpdla}.
\comment{
Using the presentation $\cal{N}_{\F_p}\cong \AnSpec\F_{p, \solid}(\!(\pi)\!)\,/\,\ol{\T}$ from \cref{cor:norms-pres} and applying \cref{lem:man-zpdla}, we see that the pushforward of $\O$ along 
\begin{equation*}
(\Z_p^d)^\la\times_{\cal{N}_{|p|<1}/\R_{>0}} \cal{N}_{\F_p}\rightarrow \cal{N}_{\F_p}
\end{equation*}
is given by the $\F_p(\!(\pi)\!)$-algebra
\begin{equation*}
\colim_{\epsilon>0} \F_p(\!(\pi)\!)\left\langle\pi^{\epsilon |\ul{n}|}\binom{\ul{x}}{\ul{n}}: \ul{n}\geq 0\right\rangle\;,
\end{equation*}
where $\ol{\T}$ acts on $\pi$ by multiplication, and our task is to compute the cohomology of this sheaf on $\cal{N}_{\F_p}$. However, by \cref{prop:norms-dnmodrasmodules}, this cohomology is given by the degree zero subspace with respect to the $\ol{\T}$-action (and the higher cohomology groups vanish), which yields the claim.
}
\end{proof}

\begin{lem}
\label{lem:smoothfp-zpladagger}
For any $d\geq 0$, the map
\begin{equation*}
(\Z_p^d)^\la\rightarrow (\Z_p^d)^\sm\times_{\AnSpec\Z_\solid} \cal{N}_{|p|<1}/\R_{>0}
\end{equation*}
is surjective with kernel
\begin{equation*}
(0\subseteq (\Z_p^d)^\la)^\dagger\coloneqq\lim_{\ul{x}\mapsto p\ul{x}} (\Z_p^d)^\la\;.
\end{equation*}
\end{lem}
\begin{proof}
The claim immediately reduces to the case $d=1$ by taking products. To prove surjectivity, we can base change along the jointly surjective maps
\begin{equation*}
\AnSpec\Z_\solid(\!(\pi)\!)\langle\tfrac{p^k}{\pi}\rangle_{\leq 1}\rightarrow\cal{N}_{|p|<1}/\R_{>0}
\end{equation*}
for $k\geq 1$ and then our task is to show that the map
\begin{equation*}
C^\sm(\Z_p, \Z(\!(\pi)\!))\tensor_{\Z(\!(\pi)\!)} \Z(\!(\pi)\!)\langle\tfrac{p^k}{\pi}\rangle_{\leq 1}\rightarrow C^\la(\Z_p, \Z(\!(\pi)\!))\tensor_{\Z(\!(\pi)\!)} \Z(\!(\pi)\!)\langle\tfrac{p^k}{\pi}\rangle_{\leq 1}
\end{equation*}
is descendable as all analytic ring structures in sight are induced from $\Z_\solid$. In fact, we claim that already $C^\sm(\Z_p, \Z(\!(\pi)\!))\rightarrow C^\la(\Z_p, \Z(\!(\pi)\!))$ is descendable. Indeed, as $C^\sm(\Z_p, \Z(\!(\pi)\!))=\colim_n C^\sm(\Z/p^n, \Z(\!(\pi)\!))$, we can use \cite[Prop.\ 2.7.2]{MannThesis} to reduce to showing that the maps $C^\sm(\Z/p^n, \Z(\!(\pi)\!))\rightarrow C^\la(\Z_p, \Z(\!(\pi)\!))$ are descendable of bounded index as $n\geq 1$ varies. However, for each $n$, the map $C^\sm(\Z/p^n, \Z(\!(\pi)\!))\rightarrow C^\la(\Z_p, \Z(\!(\pi)\!))$ has a section given by evaluating a locally analytic function on $\Z_p$ on a fixed set of representatives of $\Z/p^n$ in $\Z\subseteq\Z_p$ and is thus descendable of index $1$, as desired.

To prove the claim about the kernel, we observe that $*=\lim_{x\mapsto px} \Z_p^\sm=\lim_n (p^n\Z_p)^\sm$ since the map
\begin{equation*}
\colim_n C^\sm(p^n\Z_p, \Z)\xrightarrow{\mathrm{ev}_0} \Z
\end{equation*}
is an isomorphism as any locally constant function on $\Z_p$ must become constant on $p^n\Z_p$ for $n\gg 0$. As the kernel of $\Z_p^\la\rightarrow\Z_p^\sm$ is by definition the pullback of $\Z_p^\la$ along $*\rightarrow\Z_p^\sm$, the claim follows from $\Z_p^\la\times_{\Z_p^\sm} (p^n\Z_p)^\sm\cong (p^n\Z_p)^\la$, which is true by construction.
\end{proof}

We are now in a position to prove \cref{thm:smoothfp-main}.

\begin{proof}[Proof of \cref{thm:smoothfp-main}]
As discussed earlier, we will show that pullback along
\begin{equation*}
f: (\cal{N}_B/\R_{>0})/G^\la\rightarrow (\cal{N}_B/\R_{>0})/G^\sm\rightarrow \AnSpec B/G^\sm
\end{equation*}
is fully faithful. For this, we may replace $G$ by an open uniform subgroup $G_0\subseteq G$: Indeed, we can describe $\D(\AnSpec B/G^\sm)$ by descent from $\D(\AnSpec B/G_0^\sm)$ and all terms of the resulting \v{C}ech nerve will be finite disjoint unions of copies of $\AnSpec B/G_0^\sm$ since $G/G_0$ is discrete. As the analogous statement also holds for $\D((\cal{N}_B/\R_{>0})/G^\la)$, this establishes the desired reduction.

Replacing $G$ by $G_0$ from now on, we claim that $f$ is prim and has the property that $f_*\O\cong\O$; this will imply the desired full faithfulness by the projection formula. As the claimed property of $f$ is stable under base change, we may reduce to the case $B=\F_p$. Then the primness is the content of \cref{lem:smoothfp-glagsmprim}, so it remains to show that $f_*\O\cong\O$, which we may test after pullback along $\AnSpec \F_{p, \solid}\rightarrow \AnSpec \F_{p, \solid}/G_0^\sm$. However, pulling back $f$ along this map yields
\begin{equation*}
f': (\cal{N}_{\F_p}/\R_{>0})/(1\subseteq G_0^\la)^\dagger\rightarrow \AnSpec \F_{p, \solid}
\end{equation*}
by \cref{lem:smoothfp-zpladagger} and hence our task is to show that
\begin{equation*}
R\Gamma((\cal{N}_{\F_p}/\R_{>0})/(1\subseteq G_0^\la)^\dagger, \O)\cong \F_p\;.
\end{equation*}
Computing the left-hand side via the \v{C}ech nerve of the cover $(\cal{N}_{\F_p}/\R_{>0})\rightarrow (\cal{N}_{\F_p}/\R_{>0})/(1\subseteq G_0^\la)^\dagger$, it then suffices to prove that
\begin{equation*}
R\Gamma(((1\subseteq G_0^\la)^\dagger)^m_{\F_p}, \O)\cong \F_p
\end{equation*}
for all $m\geq 0$. As $G_0$ is uniform, we have $((1\subseteq G_0^\la)^\dagger)^m\cong (0\subseteq (\Z_p^{mn})^\la)^\dagger$ as analytic stacks (but not group stacks!), where $n$ is the dimension of $G_0$, and thus it is furthermore sufficient to show
\begin{equation*}
R\Gamma((0\subseteq (\Z_p^d)^\la)^\dagger_{\F_p}, \O)\cong \F_p
\end{equation*}
for all $d\geq 0$.

To carry out this computation, first recall from \cref{lem:smoothfp-zpladagger} that
\begin{equation*}
(0\subseteq (\Z_p^d)^\la)^\dagger=\lim_{\ul{x}\mapsto p\ul{x}} (\Z_p^d)^\la
\end{equation*}
and observe that $(\Z_p^d)^\la$ and hence also $(0\subseteq (\Z_p^d)^\la)^\dagger$ is relatively affine over $\cal{N}_{|p|<1}/\R_{>0}$, i.e.\ its pullback to any affine analytic stack over $\cal{N}_{|p|<1}/\R_{>0}$ is affine. As $\cal{N}_{\F_p}/\R_{>0}$ is prim over $\F_p$ and hence taking cohomology preserves colimits, we thus see that
\begin{equation*}
\begin{split}
R\Gamma((0\subseteq (\Z_p^d)^\la)^\dagger_{\F_p}, \O)\cong \colim_{\ul{x}\mapsto p\ul{x}} R\Gamma((\Z_p^d)^\la_{\F_p}, \O)\cong \colim_{\ul{x}\mapsto p\ul{x}} \F_p\left[\binom{\ul{x}}{\ul{n}}: \ul{n}\geq 0\right]
\end{split}
\end{equation*}
by \cref{lem:smoothfp-cohzpla}. Proving that this is isomorphic to $\F_p$ now amounts to showing that, for any nonzero $\ul{n}\geq 0$, there is some $i\gg 0$ with $\binom{p^i\ul{x}}{\ul{n}}=0$, which clearly reduces to the one-variable case $d=1$. 

In this case, our task is to prove that, for any $n>0$, there is some $i$ such that 
\begin{equation}
\label{eq:smoothfp-expandpichoosen}
\binom{p^i x}{n}=\sum_{j\leq n} \lambda_j\binom{x}{j}
\end{equation}
with $\lambda_j\in p\Z$ for all $j$. The coefficients $\lambda_j$ in this expansion are uniquely determined integers by \cref{prop:app-intbasis} and the main claim from the proof of \cref{lem:app-binompseries} shows that they satisfy
\begin{equation*}
v_p(\lambda_j)+\frac{n}{p-1}\geq ij\;.
\end{equation*}
Thus, taking $i>n/(p-1)$ ensures $v_p(\lambda_j)\geq 1$ for all $j\geq 1$ and it remains to take care of $\lambda_0$. However, the left-hand side of (\ref{eq:smoothfp-expandpichoosen}) vanishes for $x=0$ while setting $x=0$ on the right-hand side yields $\lambda_0$, so we conclude $\lambda_0=0$.
\end{proof}

\begin{rem}
One easily sees that the argument above works just the same if $B$ is merely a $\Z/p^n$-algebra for some $n\geq 1$, i.e.\ whenever $p^nB=0$ for some $n\geq 1$. In other words, the fully faithful embedding 
\begin{equation*}
\widehat{\D}(\Rep^\sm_B G)\hookrightarrow \D((\cal{N}_B/\R_{>0})/G^\la)\;,
\end{equation*}
from \cref{thm:smoothfp-main} generalises to all solid $p$-power torsion rings $B$.
\end{rem}

\section{Poincaré duality}

We continue to fix a $p$-adic Lie group $G$ and a uniform open subgroup $G_0\subseteq G$. In this last part, we establish some properties of the classifying stacks $*/G^\la$ and $*/G^{h\dagger\text{-}\an}$ with respect to the six-functor formalism on analytic stacks. In particular, we prove Poincaré duality for cohomology of locally analytic representations, i.e.\ we show that $*/G^\la$ and $*/G^{h\dagger\text{-}\an}$ are cohomologically smooth (at least for $h$ large enough). Here and in the following, we will always write $*\coloneqq \cal{N}_{|p|<1}$, but of course the statements descend to $\cal{N}_{|p|<1}/\R_{>0}$ in the case of $G^\la$ as they are always local on the target.

\begin{prop}
\label{prop:duality-prim}
Assume that $G$ is compact. For any real numbers $h'\geq h$, all morphisms in the diagram
\begin{equation*}
*/G^\la\rightarrow */G^{h'\dagger\text{-}\an}\rightarrow */G^{h\dagger\text{-}\an}\rightarrow *
\end{equation*}
are weakly cohomologically proper in the sense of \cite[Def.\ 3.1.19]{dRStack}. In particular, they are prim and there is an identification between lower-$!$ and lower-$*$.
\end{prop}
\begin{proof}
We first show that $f: */G^\la\rightarrow *$ is weakly cohomologically proper. For this, we first observe that pulling back the diagonal $\Delta_f: */G^\la\rightarrow */(G^\la\times G^\la)$ of $f$ along $*\rightarrow */(G^\la\times G^\la)$ yields the morphism $G^\la\rightarrow *$. As $G^\la$ is an affine analytic stack with induced analytic ring structure from $\Z_\solid$, this shows that $\Delta_f$ is weakly cohomologically proper by \cite[Lem.\ 3.1.21]{dRStack}. Then it remains to show that $f$ is prim. For this, we first observe that $g: */G_0^\la\rightarrow */G^\la$ is prim as it pulls back to $G/G_0\rightarrow *$, which is a finite disjoint union of points, and that $g_*\O$ is descendable in $\D(*/G^\la)$. Indeed, after pulling back to $\AnSpec (B, B^+)_\solid$ for any Tate Huber pair $(B, B^+)$ over $(\Z(\!(\pi)\!), \Z[\![\pi]\!])$ which has slope $\geq 1$ and is residually of finite type, we can identify $g_*\O$ with the locally analytic representation $(B, B^+)_\solid[G/G_0]$ of $G$ under the equivalence from \cref{thm:locan-compareporat} and then the map $B\rightarrow (B, B^+)_\solid[G/G_0]$ obtained by adjunction has a $G$-equivariant section given by sending each coset in $G/G_0$ to $1$. Thus, we can reduce to the case $G=G_0$ by \cite[Cor.\ 4.7.5]{HeyerMann}. Here, combining \cref{lem:smoothfp-desc} with another application of \cite[Cor.\ 4.7.5]{HeyerMann} reduces us to showing primness of the identity on the point, which is clear.

An analogous argument as in the previous paragraph also shows weak cohomological properness of $*/G^{h\dagger\text{-}\an}\rightarrow *$ for any $h\in\R$; indeed, note that also the proof of \cref{lem:smoothfp-desc} readily adapts to the case of $G_0^{h\dagger\text{-}\an}$ in place of $G_0^\la$. It thus remains to show that the maps $*/G^\la\rightarrow */G^{h'\dagger\text{-}\an}$ and $*/G^{h'\dagger\text{-}\an}\rightarrow */G^{h\dagger\text{-}\an}$ are weakly cohomologically proper. However, since we already know that both source and target of these maps are weakly cohomologically proper over the point, it suffices to check that the diagonal of $*/G^{h'\dagger\text{-}\an}\rightarrow *$ or $*/G^{h\dagger\text{-}\an}\rightarrow *$, respectively, is weakly cohomologically proper. As this has already been established in the first paragraph, we are done.
\end{proof}

We now turn to establishing Poincaré duality, which requires a bit more effort. We begin by making some preparations.

\begin{lem}
\label{lem:duality-dualisingsheaf}
For any solid $\Z_p$-algebra $B^+$, the $B^+$-linear dual
\begin{equation*}
L\coloneqq \sHom_{B^+_\solid[G_0]}(B^+, B^+_\solid[G_0])^\vee
\end{equation*}
of the continuous group cohomology of $B^+_\solid[G_0]$ is a free $B^+$-module of rank $1$ concentrated in degree $-d$, where $d$ is the dimension of $G_0$.
\end{lem}
\begin{proof}
By the Lazard--Serre resolution from \cref{prop:smoothfp-lazardserre}, we have
\begin{equation*}
\sHom_{B^+_\solid[G_0]}(B^+, B^+_\solid[G_0])\cong (B^+_\solid[G_0]^{\binom{d}{0}}\rightarrow B^+_\solid[G_0]^{\binom{d}{1}}\rightarrow\dots\rightarrow B^+_\solid[G_0]^{\binom{d}{d}})\;,
\end{equation*}
where the first term $B_\solid^+[G_0]^{\binom{d}{0}}$ sits in degree zero. As the Lazard--Serre resolution is constructed from a Koszul resolution, it is self-dual in the sense that the right-hand side of the above isomorphism is concentrated in degree $d$, where it is given by a free $B^+$-module of rank $1$. This finishes the proof.
\end{proof}

Recalling that we use $\O[G_0]$ to denote the pullback of $\Z_{p, \solid}[G_0]$ along $\cal{N}_{|p|<1}\rightarrow\AnSpec\Z_{p, \solid}$, we deduce that
\begin{equation*}
L\coloneq \sHom_{\O[G_0]}(\O, \O[G_0])^\vee
\end{equation*}
is an invertible object in $\D(\cal{N}_{|p|<1})$ by descent along the charts $\AnSpec\hspace{1pt}(B_k, B_k^+)_\solid\rightarrow\cal{N}_{|p|<1}$ from (\ref{eq:locan-bb+}). We note that $L$ admits the structure of an $\O[G_0]$-module induced by right multiplication on $\O[G_0]$.

\begin{lem}
\label{lem:duality-dualisingsheaflocan}
The $\O[G_0]$-module $L$ in $\D(\cal{N}_{|p|<1})$ is locally analytic, i.e.\ it lies in the image of the fully faithful embedding 
\begin{equation*}
\D(*/G_0^\la)\hookrightarrow \Mod_{\O[G_0]}(\D(\cal{N}_{|p|<1}))
\end{equation*}
from \cref{cor:locan-embedgreps}. If $G_0$ is abelian, then $L$ is in fact the trivial representation.
\end{lem}
\begin{proof}
We check the claim after pulling back to $\AnSpec\hspace{1pt}(B_k, B_k^+)_\solid\rightarrow\cal{N}_{|p|<1}$ for all $k\geq 1$. By \cref{lem:duality-dualisingsheaf}, we already know that $L$ is a character of $G_0$ and we claim that it is actually the determinant of the dual of the adjoint representation of $G_0$ from \cref{ex:porat-adjointlocan}. This would imply both claims: The adjoint representation is locally analytic by \cref{ex:porat-adjointlocan} and it is trivial if $G_0$ is abelian. To determine the $G_0$-action on $L$, we may extend scalars along $\Z_p\rightarrow\Q_p$ as $B_k$ is $p$-torsionfree, which then reduces us to showing that
\begin{equation*}
\sHom_{(\Q_p, \Z_p)_\solid[G_0]}(\Q_p, (\Q_p, \Z_p)_\solid[G_0])^\vee
\end{equation*}
is the determinant of the dual of the adjoint representation of $G_0$. However, this is true by \cite[Eq. (24)]{SolidLocAnReps1}, so we are done.
\end{proof}

We are now ready to prove that $*/G^\la$ is cohomologically smooth over $*=\cal{N}_{|p|<1}$. Indeed, we will show more explicitly that $L$ is the dualising sheaf on $*/G_0^\la$. 

\begin{thm}
\label{thm:duality-main}
The map
\begin{equation*}
*/G^\la\rightarrow *
\end{equation*}
is cohomologically smooth. Moreover, for any $h\in\R$, the map
\begin{equation*}
*/G^{h\dagger\text{-}\an}\rightarrow *
\end{equation*}
is suave; it is even cohomologically smooth \dots
\begin{enumerate}[label=(\roman*)]
\item \dots if $G$ is abelian or
\item \dots after base change to $\cal{N}_{|p|<1-\epsilon}$ for some $\epsilon>0$ and $h$ sufficiently large depending on $\epsilon$.
\end{enumerate}
\end{thm}
\begin{proof}
First note that we can reduce to the case $G=G_0$ in all cases: Indeed, cohomological smoothness may be checked smooth locally on the source and the map $*/G_0^\la\rightarrow */G^\la$ becomes a finite disjoint union of points after pullback along $*\rightarrow */G^\la$; the same argument works for $*/G^{h\dagger\text{-}\an}$. Moreover, we may also check all assertions after pullback along the charts $\AnSpec\hspace{1pt}(B_k, B_k^+)_\solid\rightarrow \cal{N}_{|p|<1}$ from (\ref{eq:locan-bb+}); as $\cal{N}_{|p|<1-\epsilon}$ will eventually lie in the image of a single one of these charts for $k\gg 0$ depending on $\epsilon$, assertion (ii) then turns into the claim that $\AnSpec\hspace{1pt}(B_k, B_k^+)_\solid/G^{h\dagger\text{-}\an}\rightarrow \AnSpec\hspace{1pt}(B_k, B_k^+)_\solid$ will be cohomologically smooth for $h\gg 0$ large enough depending on $k$. Thus, write $*\coloneqq \AnSpec\hspace{1pt}(B_k, B_k^+)_\solid$ from now on; for simplicity, we will also abbreviate $(B_k, B_k^+)$ by $(B, B^+)$.

To prove the first assertion, consider the maps $g: */G_0^\la\rightarrow *$ and $f: *\rightarrow */G_0^\la$. Our goal is to apply \cite[Lem.\ 4.2.2.(i)]{dRStack} to prove that $g$ is cohomologically smooth with dualising sheaf $L$. Our task is to produce maps $s: f_!\O\rightarrow L$ and $\eta: g_!L\rightarrow \O$ such that $\eta\circ g_!s$ is the identity. To this end, recall from \cref{prop:duality-prim} that $g_!\cong g_*$ and note that $f$ is weakly cohomologically proper as well by \cite[Lem.\ 3.1.21]{dRStack} as it is a $G_0^\la$-torsor and $G_0^\la$ is affine and equipped with the induced analytic ring structure from $\Z_\solid$. Using that $g_*$ computes condensed $G$-cohomology by \cref{cor:locan-embedgreps} and base change for $f_*$, which holds as $f$ is prim, we see that our task unwinds to producing maps
\begin{equation*}
s: C^\la(G_0, B)\rightarrow L\;, \hspace{0.5cm} \eta: \sHom_{(B, B^+)_\solid[G_0]}(B, L)\rightarrow B
\end{equation*}
with the desired compatibility, where $s$ is supposed to be $G_0$-equivariant.

To produce $s$, we recall that $C^\la(G_0, B)=\colim_h C_{h\text{-}\an}(G_0, B)$ and hence it suffices to produce compatible maps $C_{h\text{-}\an}(G_0, B)\rightarrow L$. As both the source and target are reflexive, it then suffices to produce compatible maps
\begin{equation*}
\sHom_{(B, B^+)_\solid}(B, (B, B^+)_\solid[G_0])=L^\vee\rightarrow D_{h\text{-}\an}(G_0, B)\;.
\end{equation*}
However, recalling from the proof of \cref{lem:duality-dualisingsheaf} that the left-hand side is given by the complex
\begin{equation}
\label{eq:duality-lvee}
(B, B^+)_\solid[G_0]^{\binom{d}{0}}\rightarrow (B, B^+)_\solid[G_0]^{\binom{d}{1}}\rightarrow\dots\rightarrow (B, B^+)_\solid[G_0]^{\binom{d}{d}}\;,
\end{equation}
where the first term sits in degree zero, we see that this admits a canonical map to $(B, B^+)_\solid[G_0]$ given by projecting onto the degree zero term of the above complex. This yields the desired maps after postcomposing with $(B, B^+)_\solid[G_0]\rightarrow D_{h\text{-}\an}(G_0, B)$ and it is clear that this is $G_0$-equivariant.

Now we produce $\eta$. For this, we use the Lazard--Serre resolution from \cref{prop:smoothfp-lazardserre} to write $\sHom_{(B, B^+)_\solid[G_0]}(B, L)$ as the complex
\begin{equation*}
L^{\binom{d}{0}}\rightarrow L^{\binom{d}{1}}\rightarrow\dots\rightarrow L^{\binom{d}{d-1}}\rightarrow L^{\binom{d}{d}}\;,
\end{equation*}
where the first term sits in degree zero. However, using the explicit complex computing $L^\vee$ from (\ref{eq:duality-lvee}), we immediately identify the above as $L^\vee\tensor_{(B, B^+)_\solid[G_0]} L$ and now the desired map
\begin{equation*}
\sHom_{(B, B^+)_\solid[G_0]}(B, L)\cong L^\vee\tensor_{(B, B^+)_\solid[G_0]} L\rightarrow B
\end{equation*}
is the adjoint of the identity on $L^\vee$. It is not hard to check that the two maps we have constructed satisfy $\eta\circ g_!s=\id$, as desired.

To prove the second part of the statement, recall from \cref{prop:duality-prim} that $f_h: */G_0^\la\rightarrow */G_0^{h\dagger\text{-}\an}$ is weakly cohomologically proper and hence $f_{h, !}\cong f_{h, *}$ preserves suave objects by \cite[Lem.\ 4.5.16]{HeyerMann}. As $f_{h, *}$ is given by taking $h\dagger$-analytic vectors by \cref{lem:locan-pushforward} and the trivial representation is $h\dagger$-analytic for all $h\in\R$, we conclude that $f_{h, !}\O\cong \O\in\D(*/G_0^{h\dagger\text{-}\an})$ and hence $*/G_0^{h\dagger\text{-}\an}$ is suave. Using \cite[Lem.\ 4.4.9.(ii)]{HeyerMann}, we see that the dualising complex is given by $f_{h, !}L=L^{h\dagger\text{-}\an}$ and hence we can conclude that $*/G_0^{h\dagger\text{-}\an}$ is cohomologically smooth whenever $L$ is already an $h\dagger$-analytic representation. As $L$ is given by the determinant of the dual of the adjoint representation of $G_0$ by the proof of \cref{lem:duality-dualisingsheaflocan}, it is trivial if $G_0$ is abelian and hence in particular $h\dagger$-analytic for all $h\in\R$; this proves (i). To establish (ii), we note that the character $\chi: G\rightarrow B^\times$ defining the action of $G$ on $L$ is locally analytic since $L$ is and due to 
\begin{equation*}
C^\la(G_0, B)=\colim_h C^{h\dagger\text{-}\an}(G_0, B)
\end{equation*}
must in fact be $h\dagger$-analytic for some $h\gg 0$. Consequently, the representation $L$ is $h\dagger$-analytic for some $h\gg 0$, which finishes the proof.
\end{proof}

\appendix

\section{Some facts about binomial coefficients}

Here we collect some miscellaneous results involving binomial coefficients used in the body of the paper. While some of these are standard, others are tailored to our applications; in any case, the proofs are elementary and combinatorial in flavour. 

\subsection{Integer-valued polynomials}

Recall that a polynomial $f$ in $d$ variables with rational coefficients is called \emph{integer-valued} if $f(x_1, \dots, x_d)$ is an integer whenever $x_1, \dots, x_d\in\Z$.

\begin{prop}
\label{prop:app-intbasis}
The ring of integer-valued polynomials in one variable has a $\Z$-basis given by the polynomials $\binom{x}{n}$ for $n\geq 0$.
\end{prop}
\begin{proof}
The following proof is taken from \cite[Prop.\ I.1.1]{IntegerPolys}. Clearly, the polynomials $\binom{x}{n}, n\geq 0$ are a $\Q$-basis of $\Q[x]$ and are all integer-valued, so it suffices to show that any integer-valued polynomial $f$ is a $\Z$-linear combination of these. Thus, write
\begin{equation*}
f=\sum_{j=0}^d \lambda_j\binom{x}{j}
\end{equation*}
for uniquely determined $\lambda_j\in\Q$. By induction on $j$, we will prove that, in fact, $\lambda_j\in\Z$. Indeed, we have $\lambda_0=f(0)\in\Z$ and, assuming that $\lambda_j\in\Z$ for $j\leq k<n$, we obtain an integer-valued polynomial
\begin{equation*}
g_k=f-\sum_{j=0}^k \lambda_j\binom{x}{j}=\sum_{j=k+1}^d \lambda_j\binom{x}{j}
\end{equation*}
and it is clear that $\lambda_{k+1}=g(k+1)\in\Z$, which completes the induction step.
\end{proof}

\begin{rem}
\label{rem:app-intbasiszp}
The same proof shows that the polynomials $\binom{x}{n}$ form a $\Z_p$-basis of the ring of $\Z_p$-valued polynomials over $\Q_p$ in one variable.
\end{rem}

\begin{prop}
\label{prop:app-intbasismult}
The ring of integer-valued polynomials in $d$ variables has a $\Z$-basis given by the polynomials
\begin{equation*}
\binom{x_1}{n_1}\binom{x_2}{n_2}\cdots\binom{x_d}{n_d}
\end{equation*}
for $n_1, \dots, n_d\geq 0$.
\end{prop}
\begin{proof}
See also \cite[Prop.\ XI.1.12]{IntegerPolys}. We use induction on $d$, the base case $d=1$ being the content of \cref{prop:app-intbasis}. Again, the polynomials from the statement are clearly all integer-valued and form a $\Q$-basis of $\Q[x_1, \dots, x_d]$, so any integer-valued polynomial $f\in\Q[x_1, \dots, x_d]$ may be expressed as
\begin{equation*}
f(x_1, \dots, x_d)=\sum_{\ul{n}} a_{\ul{n}} \binom{x_1}{n_1}\cdots \binom{x_d}{n_d}\;,
\end{equation*}
where $\ul{n}=(n_1, \dots, n_d)$ is a multiindex. Our task is to show that $a_{\ul{n}}\in\Z$.

To prove this, we use induction on the $x_d$-degree of $f$. In the base case, the polynomial $f$ is independent of $x_d$ and hence an integer-valued polynomial in $d-1$ variables, so the claim follows by the induction hypothesis of the ``outer'' induction over $d$. For the induction step, note that
\begin{equation*}
f(x_1, \dots, x_{d-1}, x_d+1)-f(x_1, \dots, x_{d-1}, x_d)=\sum_{\ul{n}: n_d=0} a_{\ul{n}} \binom{x_1}{n_1}\cdots \binom{x_{d-1}}{n_{d-1}}\binom{x_d}{n_d-1}
\end{equation*}
is integer-valued as well and has lower $x_d$-degree than $f$. Thus, we conclude $a_{\ul{n}}\in\Z$ for all $\ul{n}$ by the induction hypothesis and this finishes the proof.
\end{proof}

\begin{rem}
\label{rem:app-intbasiszpmult}
Again, the same proof shows that the polynomials $\binom{x_1}{n_1}\cdots \binom{x_d}{n_d}$ form a $\Z_p$-basis of the ring of $\Z_p$-valued polynomials over $\Q_p$ in $d$ variables.
\end{rem}

\subsection{Mahler expansions}

We recall the following result from $p$-adic analysis:

\begin{thm}[Mahler]
\label{thm:app-mahler}
There is a bijection
\begin{equation*}
\{\text{continuous functions $f: \Z_p\rightarrow\Q_p$}\}\overset{1:1}{\longleftrightarrow}\{\text{nullsequences $(a_n)_n$ in $\Q_p$}\}
\end{equation*}
given by sending a nullsequence $(a_n)_n$ to the function $f: \Z_p\rightarrow\Q_p$ defined by
\begin{equation*}
f(x)\coloneqq \sum_n a_n\binom{x}{n}\;.
\end{equation*}
Conversely, the sequence $(a_n)_n$ can be recovered from $f$ via
\begin{equation*}
a_n=\sum_{\ell\leq n} (-1)^{n-\ell}\binom{n}{\ell}f(\ell)\;.
\end{equation*}
\end{thm}
\begin{proof}
This is the main result of \cite{Mahler}.
\end{proof}

Using the rate of decay of the Mahler coefficients $a_n$, one can characterise the open balls on which $f$ is actually an analytic function.

\begin{thm}[Amice]
\label{thm:app-amice}
Let $f: \Z_p\rightarrow\Q_p$ be any continuous function. Then $f$ is analytic on each ball $x+p^h\Z_p$ for $x\in \Z/p^h$ if and only if
\begin{equation*}
v_p(a_n)-v_p(\lfloor n/p^h\rfloor !)\rightarrow \infty
\end{equation*}
as $n\rightarrow\infty$. In particular, $f$ is locally analytic if and only if $v_p(a_n)$ grows at least linearly.
\end{thm}
\begin{proof}
The first part is \cite[§10, Cor.\ 2]{Amice}. The second part follows since $v_p(\lfloor n/p^h\rfloor!)\sim \tfrac{n}{p^h(p-1)}$ as $n$ grows, i.e.\
\begin{equation*}
\frac{v_p(\lfloor n/p^h\rfloor!)}{n/p^h(p-1)}\rightarrow 1
\end{equation*}
as $n\rightarrow\infty$.
\end{proof}

\begin{rem}
\label{rem:app-amicemult}
The previous two theorems also have analogues in several variables. In particular, it is also proved in \cite{Amice} that any locally analytic function $f: \Z_p^d\rightarrow\Q_p$ is given by
\begin{equation*}
f(x_1, \dots, x_d)=\sum_{\ul{n}} a_{\ul{n}}\binom{\ul{x}}{\ul{n}}
\end{equation*}
for $p$-adic integers $a_{\ul{n}}$ whose $p$-adic valuation grows at least linearly in $|\ul{n}|\coloneqq n_1+\dots+n_d$, where $\ul{n}=(n_1, \dots, n_d)$ is a multiindex and we set $\binom{\ul{x}}{\ul{n}}\coloneqq \binom{x_1}{n_1}\cdots \binom{x_d}{n_d}$.
\end{rem}

\subsection{Miscellany}

\begin{lem}
\label{lem:app-prodbinom}
For all $m, n\geq 0$, we have
\begin{equation}
\label{eq:app-prodbinom}
\binom{x}{m}\binom{x}{n}=\sum_j \binom{x}{j}\binom{j}{m+n-j}\binom{2j-m-n}{j-n}
\end{equation}
as polynomials in $x$, where we use the convention that, for $k, \ell\in\mathbb{Z}$, the binomial coefficient $\binom{k}{\ell}$ is only nonzero if $0\leq \ell\leq k$.
\end{lem}
\begin{proof}
Since both sides are polynomials in $x$, it suffices to prove that they agree for all $x\in\mathbb{N}$. In that case, the left-hand side counts the number of pairs $(A, B)$ of subsets of a set $S$ of size $x$ with $|A|=m, |B|=n$. However, we can also count this number in a different way: we first choose the union $A\cup B\subseteq S$, afterwards select the elements of $A\cup B$ that lie in both $A$ and $B$ and finally select those elements of $(A\cup B)\setminus (A\cap B)$ that belong to $A$. If $|A\cup B|=j$, this means that we first select $j$ elements from $S$, for which there are a total of $\binom{x}{j}$ possibilities; then we select the $m+n-j$ elements contained in $A\cap B$ from the $j$ elements of $A\cup B$, for which there are $\binom{j}{m+n-j}$ possibilities; finally we select the remaining $j-n$ elements of $(A\cup B)\setminus (A\cap B)$ which belong to $A$, for which there are $\binom{2j-m-n}{j-n}$ possibilities since $|(A\cup B)\setminus (A\cap B)|=2j-m-n$. Summing over $j=|A\cup B|$, this yields the expression on the right-hand side.
\end{proof}

\begin{lem}
\label{lem:app-binompseries}
For each $n\geq 0$, there is an identity of formal power series of the form
\begin{equation*}
\begin{split}
\binom{\sum_{i\geq 0} p^ix_i}{n}=\sum_{|\ul{j}|\leq n} \lambda_{\ul{j}}\binom{\ul{x}}{\ul{j}}\;,
\end{split}
\end{equation*}
where each $\lambda_{\ul{j}}$ is an integer, $\ul{j}=(j_0, j_1, \dots)$ is a multiindex with almost all entries zero, for which we set $|\ul{j}|\coloneqq \sum_i j_i$, and we write $\binom{\ul{x}}{\ul{j}}\coloneqq \prod_i \binom{x_i}{j_i}$. Moreover, we have
\begin{equation*}
\frac{1}{k}v_p(\lambda_{\ul{j}})+\epsilon n\geq \sum_i \frac{i j_i}{k}
\end{equation*}
for any $k\geq 1$ and $\epsilon>0$ satisfying $\epsilon k\geq \frac{1}{p-1}$.
\end{lem}
\begin{proof}
Notice that, by double counting, we have
\begin{equation*}
\binom{\sum_{i\geq 0} p^ix_i}{n}=\sum_{\substack{(j_i)_i\in\mathbb{N} \\ \sum_i j_i=n}} \prod_{i\geq 0} \binom{p^ix_i}{j_i}
\end{equation*}
and hence, by \cref{prop:app-intbasis}, it suffices to prove the following
\bigskip

\textbf{Claim.} For any $i, j\geq 0$, the binomial coefficient $\binom{p^i x}{n}$ is an integer linear combination of $x, \binom{x}{2}, \dots$ with the property that the coefficient $\lambda_j$ of $\binom{x}{j}$ satisfies 
\begin{equation*}
\frac{1}{k}v_p(\lambda_j)+\epsilon n\geq \frac{ij}{k}\;.
\end{equation*}

\bigskip

\textit{Proof of the claim.} We first show that, if we expand $\binom{p^i x}{n}$ as a polynomial in $x$ with rational coefficients, then the coefficient $\lambda_j$ in front of $x^j$ satisfies
\begin{equation}
\label{eq:app-padicvalcoeffs}
\frac{1}{k}v_p(\lambda_j)+\epsilon n\geq \frac{ij}{k}\;.
\end{equation}
To this end, note that expanding the brackets in the expression
\begin{equation*}
\binom{p^ix}{n}=\frac{p^ix(p^ix-1)\cdots (p^ix-(j-1))}{n!}
\end{equation*}
yields
\begin{equation*}
\lambda_j=\pm p^{ij}\sum_{1\leq m_1<m_2<\dots<m_j\leq n-1} \frac{1}{m_1m_2\cdots m_j}
\end{equation*}
and hence, using the given condition on $k$ and $\epsilon$, our task lies in showing
\begin{equation*}
v_p\left(\sum_{1\leq m_1<\dots<m_j\leq n-1} \frac{1}{m_1\cdots m_j}\right)\geq -\frac{n}{p-1}\;.
\end{equation*}
However, as there are at most $n/p^t$ numbers below $n$ divisible by $p^t$ for each $t\geq 1$, the left-hand side is bounded below by
\begin{equation*}
-\frac{n}{p}-\frac{n}{p^2}-\dots=-\frac{n}{p-1}\;,
\end{equation*}
as desired.

The fact that the polynomial $\binom{p^i x}{n}=\sum_{j\leq n} \lambda_j x^j$ with rational coefficients can be rewritten as an integer linear combination of $x, \binom{x}{2}, \dots$ now follows from \cref{prop:app-intbasis} and we have to check that this rewriting does not disturb the inequality on the $p$-adic valuations of the coefficients that we have established. Going through the proof of loc.\ cit., we see that this rewriting starts by writing
\begin{equation*}
\sum_{j\leq n} \lambda_j x^j=\lambda_n'\binom{x}{n}+\sum_{j<n} \lambda'_j x^j\;.
\end{equation*}
Noting that 
\begin{equation*}
\lambda_n'=n!\cdot\lambda_n\;,
\end{equation*}
we see that $v_p(\lambda_n')\geq v_p(\lambda_n)$ and hence the inequality (\ref{eq:app-padicvalcoeffs}) is still satisfied for $j=n$ and $\lambda_j'$ in place of $\lambda_j$. Moreover, expanding 
\begin{equation*}
\binom{x}{n}=\frac{x(x-1)\cdots(x-(n-1))}{n!}
\end{equation*}
as a rational polynomial in $x$, we see that the $p$-adic valuation of each coefficient is at least as large as the $p$-adic valuation of the leading coefficient. From this, we conclude
\begin{equation*}
v_p(\lambda_j'-\lambda_j)\geq v_p(\lambda_n)
\end{equation*}
for each $j<n$ and hence, using that the $\lambda_j$ satisfy (\ref{eq:app-padicvalcoeffs}), we get
\begin{equation*}
v_p(\lambda_j')\geq \min\{v_p(\lambda_j), v_p(\lambda_n)\}\geq \min\{ij-\epsilon k n, in-\epsilon k n\}=ij-\epsilon k n\;,
\end{equation*}
i.e.\ the inequality (\ref{eq:app-padicvalcoeffs}) is still satisfied for all $j<n$ with $\lambda_j'$ in place of $\lambda_j$. 

Now one repeats the argument iteratively: next, we have to write
\begin{equation*}
\sum_{j<n} \lambda_j'x^j=\lambda_{n-1}''\binom{x}{n-1}+\sum_{j<n-1} \lambda_j''x^j
\end{equation*}
and can use the same argument to show that the $\lambda_j''$ still satisfy (\ref{eq:app-padicvalcoeffs}) etc. This can be done until the whole polynomial is rewritten as a linear combination of $x, \binom{x}{2}, \dots$ and since our initial polynomial in $x$ was integer-valued, this will produce an integer linear combination by \cref{prop:app-intbasis}. As the argument above shows that the coefficients satisfy the desired inequality, we are done.
\end{proof}

\begin{lem}
\label{lem:app-genfunc}
For any $k\geq 0$, there is an identity of formal power series
\begin{equation*}
\frac{1}{(1-T)^{k+1}}=\sum_n \binom{n+k}{k}T^n\;.
\end{equation*}
\end{lem}
\begin{proof}
Using the binomial expansion, we obtain
\begin{equation*}
\frac{1}{(1-T)^{k+1}}=\sum_n \binom{-k-1}{n}(-T)^n=\sum_n \binom{n+k}{n}T^n=\sum_n \binom{n+k}{k}T^n\;,
\end{equation*}
as desired.
\end{proof}

\begin{lem}
\label{lem:app-valuationmahlercoeffs}
Let $a\in \{0, 1, \dots, p-1\}$ and $n\geq 0$. Setting
\begin{equation*}
b_{mn}\coloneqq \sum_{a+p\ell\leq m} (-1)^{m-a-p\ell}\binom{m}{a+p\ell}\binom{\ell}{n}
\end{equation*}
for any $m\geq 0$, we have
\begin{equation*}
v_p(b_{mn})\geq \frac{m-pn}{p-1}-1\;.
\end{equation*}
\end{lem}
\begin{proof}
We consider the generating function $f_n(T)\coloneqq \sum_m b_{mn}T^m$. By \cref{lem:app-genfunc}, we have
\begin{equation*}
\begin{split}
f_n(T)&=\sum_m\sum_{a+p\ell\leq m} (-1)^{m-a-p\ell}\binom{m}{a+p\ell}\binom{\ell}{n}T^m=\sum_\ell \binom{\ell}{n}T^{a+p\ell}\sum_{m\geq a+p\ell} \binom{m}{a+p\ell}(-T)^{m-(a+p\ell)} \\
&=\sum_\ell \binom{\ell}{n}T^{a+p\ell}\frac{1}{(1+T)^{a+p\ell+1}}=\frac{T^{a+pn}}{(1+T)^{a+pn+1}}\sum_\ell \binom{\ell}{n}\left(\frac{T}{1+T}\right)^{p(\ell-n)} \\
&=\frac{T^{a+pn}}{(1+T)^{a+pn+1}}\frac{1}{\left(1-\frac{T^p}{(1+T)^p}\right)^{n+1}}=\frac{T^{a+pn}(1+T)^{p-a-1}}{((T+1)^p-T^p)^{n+1}}\;.
\end{split}
\end{equation*}
Noting that $(T+1)^p-T^p=1-pT(\,\dots)$, expanding the denominator into a power series then yields
\begin{equation*}
f_n(T)=T^{a+pn}(1+T)^{p-a-1}\sum_\ell \binom{-n-1}{\ell}(pT(\,\dots))^\ell\;.
\end{equation*}
As the $(\,\dots)$ term here is a polynomial in $T$ of degree $p-2$ here, the coefficient in front of the degree $j$ term of the sum must be divisible by $p$ at least $j/(p-1)$ times for every $j\geq 0$. Accounting for the degree $pn+(p-1)$ factor $T^{a+pn}(1+T)^{p-a-1}$ in front, this means that the degree $m$ coefficient of $f_n(T)$ is divisible by $p$ at least
\begin{equation*}
\frac{m-pn-(p-1)}{p-1}=\frac{m-pn}{p-1}-1
\end{equation*}
times, as desired.
\end{proof}

\bibliographystyle{alpha}
\bibliography{References}
\end{document}